\documentclass[10pt]{article}

\usepackage{amssymb,amsmath,amsthm}
\usepackage{amsfonts}
\usepackage{amsmath, amssymb} 
\usepackage{graphicx}
\usepackage[usenames]{color}
\usepackage{caption}
\usepackage{bm,mathrsfs}
\usepackage{url}
\usepackage{subfigure}
\usepackage{algorithm,algorithmic}
\usepackage{dsfont,verbatim}
\usepackage{epsfig}
\usepackage[margin=1.2in]{geometry}
\usepackage{enumerate}

\graphicspath{{./pics/}}
\allowdisplaybreaks
\usepackage{hyperref}
\hypersetup{
	colorlinks=true,   
	linkcolor=blue,    
	citecolor=blue,    
	filecolor=magenta, 
	urlcolor=cyan      
}

\usepackage{indentfirst}

\usepackage{authblk}

\newtheorem{theorem}{Theorem}[section]
\newtheorem{lemma}{Lemma}[section]

\newtheorem{remark}{Remark}[section]

\numberwithin{equation}{section}

\def\div{\mbox{div}}

\def\dfrac{\displaystyle\frac}

\def\dint{\displaystyle\int}
\def\dsum{\displaystyle\sum}
\def\R{\mathrm R}

\def\R{{\Bbb R}}

\def\u{\textbf{u}}

\def\v{\textbf{v}}
\def\g{\textbf{g}}

\def\f{\textbf{f}}
\def\V{\textbf{V}}
\def\H{\textbf{H}}
\def\L{\textbf{L}}
\def\W{\textbf{W}}

\def\e{\textbf{e}}

\def\n{\textbf{n}}

\def\P{\textbf{P}}
\def\R{\textbf{R}}

\def\RT{ \textbf{RT}}
\def\nn{\nonumber}

\title{\bf Error analysis of a Euler finite element scheme for Natural convection model with variable density}

\author[a]{Li Hang}
\author[a]{Chenyang Li \thanks{Corresponding author.\\  
E-mail address: 52275500026@stu.ecnu.edu.cn (C. Li)}}
\affil[a]{School of Mathematical Sciences, East China Normal University, Shanghai 200241, China.}

\begin{document}

\maketitle

\begin{abstract}

In this paper, we derive first-order Euler finite element discretization schemes for a time-dependent natural convection model with variable density (NCVD). The model is governed by the variable density Navier–Stokes equations coupled with a parabolic partial differential equation that describes the evolution of temperature. Stability and error estimate for the velocity, pressure, density and  temperature in $L^2$-norm are proved by using finite element approximations in space and finite differences in time. Finally, the numerical results are showed to support the theoretical analysis.

\textbf{Keywords}:
Natural convention, Variable density, Mixed finite element, Error estimate
\end{abstract}

\pagestyle{myheadings}
\thispagestyle{plain}

\section{Introduction}

Natural convection arises when temperature gradients in a fluid cause spatial variations in density. Under the influence of gravity, these density differences generate buoyancy forces that drive fluid motion. This phenomenon is observed in many engineering and geophysical applications such as atmospheric flows, oceanic circulations, and building ventilation. In this paper, we consider the  Natural Convection model with variable density (NCVD)
which are governed by the following nonlinear coupled system
in $\Omega\times (0, T]$:
\begin{eqnarray}
         \rho_t+  \nabla\cdot ( \rho \u) &=&0, \label{ns1}  \\
         \label{ns2}
         \rho \u_t-\mu\Delta \u+\rho(\u\cdot \nabla )\u+\nabla p&=&\textbf{f}, \\
         \label{ns3}
         \nabla\cdot \u&=&0,\\
         \label{ns4}
         \rho (\theta_t+ \textbf{u} \cdot \nabla \theta ) - \kappa \Delta \theta &=&g.
        \end{eqnarray}
where $\Omega \subset \R^3$ is a convex polyhedron domain,
$\f$ and $g$ are given body force and $\mu>0$ is the viscosity coefficient, $\kappa>0$ is the thermal conductivity parameter. 
In the above system (\ref{ns1})-(\ref{ns3}),
the unknown functions are the density $\rho$, the velocity field $\u$ and the pressure $p$, the temperature $\theta$.

The system (\ref{ns1})-(\ref{ns3}) are supplemented the following
initial-boundary conditions:
\begin{align}\label{BC}
\left\{\begin{array}{lll}
\rho(x,0)=\rho_0(x), \quad \rho(x,t)|_{\Gamma_{in}}=b(x,t),\\
\u(x,0)=\u_0(x), \quad \u(x,t)|_{\Gamma}=\textbf{a}(x,t),\\
\theta (x,0) =\theta_0(x), \quad \theta(x,t) | _{\Gamma} =0,
\end{array}\right.
\end{align}
where $\Gamma:=\partial\Omega$ is the boundary, and
$\Gamma_{\text{in}}$ is the general inflow boundary defined
by $\Gamma_{\text{in}}=\{ x\in\Gamma: \ \g\cdot\n<0 \}$.
For the reason of simplicity,  we consider
the homogeneous Dirichlet boundary condition for the velocity, i.e. $\textbf{a}(x,t)=0$,
which means that the boundary is impermeable, i.e., $\Gamma_{\text{in}}=\emptyset$.

In addition, we assume that there is no vacuum
state inside the domain and there exist two positive constants $\rho_0^{\min}$ and $\rho_0^{\max}$ such that
\begin{align}
        \label{compatibility}
        0< \min_{x\in\overline\Omega}\rho_0(x):=\rho_0^{\min}   \leq \rho_0(x) \leq \rho_0^{\max}:= \max_{x\in\overline\Omega}\rho_0(x).
\end{align}
Furthermore, the initial value $\u_0$ satisfies the incompressible condition $\nabla\cdot\u_0=0$ in $\Omega$.

When the temperature variation is relatively small, the Boussinesq approximation can be applied, where the density is considered constant in all terms except for the buoyancy term in the momentum equation. Many researchers have studied the natural convention model with constant density \cite{vahl1983,huang2015,su2014,su2017,su20172,wu2016,wu2017,zhang2016}.
However, in many geophysical flows, the temperature differences are sufficiently large to induce significant density variations, thereby rendering the Boussinesq approximation invalid.

The authors constructed unconditionally stable Gauge-Uzawa fnite element schemes for natural convection problem with variable density in \cite{wu2017}, the proposed schemes lead to a sequence of decoupled elliptic equations
to solve at each step, which are very eﬃcient and easy to implement. 
A novel characteristic variational multiscale finite element method  was introduced in \cite{wang2019}, which combines advantages of both the characteristic and variational multiscale methods within a variational framework for solving the incompressible natural convection
problem with variable density. The authors presented a new variant of the smoothed particle hydrodynamics simulations for natural convection problem with variable density in \cite{szewc2011}.
A novel fractional time-stepping finite element approaches was presented for solving
incompressible natural convection problems with variable density \cite{wu2020}, the main merit of these methods is that it only need to solve one
Poisson equation per time step for the pressure, which is computationally more eﬃcient. We attempt to develop eﬃcient numerical methods and give mathematical
analysis based the above research.

Based on the discussion above, we will study the back Euler finite element discrete scheme for natural convection model with variable density, we will prove the stable and convergent analysis of the proposed schemes for NCVD problem. In the proposed fully discrete scheme, nonlinear terms were treated by a linearized semi-implicit approximation such that it is easy for implementation.

The remainder of this paper is organized as follows. Notations, along with time discretization, are introduced
in Section 2. We develop a first-order Euler finite element discrete scheme for incompressible NC problems with variable density in Section 3. In Section 4, we prove the error estimate  of the first-order Euler finite element  algorithms. In Section 5,
Then numerical experiments
illustrating the performance of the methods are reported. Finally, we end with a short conclusion in Section 6.

\section{Preliminaries}
\label{Preliminaries}
For $k\in\mathbb N^+$ and $1\leq p\leq +\infty$, we use $W^{k,p}(\Omega)$ to denote the classical Sobolev space.
The norm in $W^{k,p}(\Omega)$ is denoted by $\|\cdot\|_{W^{k,p}}$ defined by a classical way.
Denote $W_0^{k,p}(\Omega)$ be the subspace of $W^{k,p}(\Omega)$ where the
functions have zero trace on $\partial\Omega$.
Especially, $W^{0,p}(\Omega)$ is the Lebesgue space $L^p(\Omega)$.
When $p=2$, $W^{k,2}(\Omega)$ is the Hilbert space which is simply denoted by $H^k(\Omega)$.
The boldface notations $\H^k(\Omega), \W^{k,p}(\Omega)$ and $\L^p(\Omega)$ are
used to denote the vector-value Sobolev spaces corresponding to
 $H^k(\Omega)^3, W^{k,p}(\Omega)^3$ and $L^p(\Omega)^3$, respectively.

Denote
\begin{align*} W & =H^1(\Omega), \quad \V =\H^1_0(\Omega), \quad \V_0=\{\v\in \V, \ \nabla\cdot\v=0 \}, \\
        M & =L_0^2(\Omega)=\{ q\in
        L^2(\Omega), \ \dint_\Omega qd x=0 \},\quad        F=H^1_0(\Omega).
\end{align*}
The norm in $\V$ can be defined by
$$\|\v\|_V=\|\nabla\v\|_{L^2}=\left(\dint_\Omega|\nabla\v|^2dx\right)^{1/2} \quad \forall \ \v\in\V.$$

For simplicity, we denote the inner products of both
$L^2(\Omega)$ and $\textbf{L}^2(\Omega)$
by $(\cdot,\cdot)$, namely,
\begin{align*}
	\begin{split}
		&(u,v)=\int_\varOmega u(x)v(x) d x \quad \forall  \, u,v\in L^2(\Omega),\\
		&({\bf u},{\bf v})=\int_\Omega {\bf u}(x)\cdot {\bf v}(x) dx \quad \forall \, {\bf u},{\bf v}\in {\bf L}^2(\Omega) .
	\end{split}
\end{align*}

The discrete Gronwall inequality established in \cite{evance2022,heywood1990} will be used frequently in the following.

\begin{lemma}
        \label{GronwallLemma}
         Let $a_k, b_k$ and $\gamma_k$  be the nonnegative numbers such that
        \begin{align}\label{growninequality-discrete}a_n+\tau\dsum\limits_{k=0}^n b_k\leq \tau \dsum\limits_{k=0}^n \gamma_ka_k+ B
         \quad \mbox{for} \ n\geq 0.\end{align}
        Suppose $\tau\gamma_k<1$ and set $\sigma_k=(1-\tau\gamma_k)^{-1}$. Then there holds
        \begin{align}\label{gronwall2}a_n+\tau\dsum\limits_{k=0}^n b_k\leq \exp(\tau\dsum\limits_{k=0}^n
                \gamma_k \sigma_k)B
         \quad \mbox{for} \ n\geq 0. \end{align}
         	\begin{remark}
         	If the sum on the right-hand side of (\ref{growninequality-discrete}) extends only up to $n-1$, then the estimate (\ref{gronwall2}) still holds for all $k \geq 1$ with $\sigma_k=1$.
         \end{remark}
        \end{lemma}

\subsection{An equivalent system}

The system (\ref{ns1})-(\ref{ns3}) is written in convection form is difficult to analysis, by introducing $\sigma = \sqrt{\rho}$ and the following relation \cite{wu2017,guemond2000,chen2024}
\begin{eqnarray*}
        \dfrac{\partial\rho}{\partial t} + \nabla\cdot(\rho\u)
        & = & 2\sigma \left( \dfrac{\partial\sigma}{\partial t} +\nabla\cdot(\sigma\u) \right)=0, \\
        \rho \dfrac{\partial\u}{\partial t}&=&\sigma\dfrac{\partial (\sigma\u)}{\partial t} + \dfrac \u 2 \nabla\cdot(\rho\u),\\
        \rho \frac{\partial \theta}{\partial t} &=&\sigma \frac{\partial(\sigma \theta)}{\partial t} + \frac{\theta}{2} \nabla \cdot (\rho \textbf{u}).
\end{eqnarray*}
we rewrite NCVD problem (\ref{ns1})-(\ref{ns3}) to the following equivalent system:
\begin{eqnarray}
        \label{newns1}
 \sigma_t+  \nabla\cdot ( \sigma \u) &=&0,\\
 \label{newns2}
 \sigma(\sigma \u)_t-\mu\Delta \u+\rho(\u\cdot \nabla )\u+\frac{1}{2}\u\nabla\cdot (\rho\u)+\nabla p&=&\f,\\
 \label{newns3}
 \nabla\cdot \u&=&0,\\
 \label{eq17}
 \sigma (\sigma \theta)_t -\kappa \Delta \theta+\rho (\textbf{u}\cdot \nabla) \theta +\frac{1}{2} \theta \nabla \cdot (\rho \textbf{u}) &=&g,
\end{eqnarray}


For any sequence of functions $\{g^n\}_{n=0}^N$, we denote
$$D_\tau g^{n}=\dfrac{g^{n}-g^{n-1}}{\tau} \quad\mbox{for}  \ 1\leq n\leq N.$$

Start with $\sigma^0=\sigma_0$ and $\u^0=\u_0$.
 For $0\leq n\leq N-1$, the exact $(\sigma, \textbf{u}, p,\theta )$ solution satisfies the following variational formulation

Find $\sigma^{n+1}\in W$ such that
\begin{align}\label{ExactScheme1}
        (D_\tau \sigma^{n+1},r) + (\nabla\cdot (\sigma^{n+1}\u^n),r)=(R^{n+1}_{\sigma} ,r),
        \qquad \forall \ r\in W
\end{align}
and
\begin{multline}\label{ExactScheme2}
       (\sigma^{n+1}  D_\tau (\sigma^{n+1}\u^{n+1}),\v) + \mu (\nabla \u^{n+1},\nabla \v)
        +   (\rho^{n+1}(\u^n\cdot\nabla)\u^{n+1},\v) 
     + \left(\dfrac{\u^{n+1}}2 \nabla\cdot(\rho^{n+1}\u^n),\v \right)  \\
   -(\nabla\cdot\v, p^{n+1})+(\nabla\cdot\u^{n+1},q)-
        =( \f^{n+1},\v)+(R^{n+1}_{\textbf{u}},\textbf{v} ) ,
\end{multline}
for any $\v\times q\in \V\times M$.

and 
\begin{multline}\label{ExactScheme3}
       (\sigma^{n+1}  D_\tau (\sigma^{n+1}\theta^{n+1}),w) + \mu (\nabla \theta^{n+1},\nabla w)
        +   (\rho^{n+1}(\u^n\cdot\nabla)\theta^{n+1},w) \\
     + \left(\dfrac{\theta^{n+1}}2 \nabla\cdot(\rho^{n+1}\u^n),w \right)  
        =( g^{n+1},w)+(R^{n+1}_{\theta},w ) 
\end{multline}
for any $w\in F$.

The truncation error function are given by
\begin{align*}
R^{n+1}_{\sigma} &= D_{\tau} \sigma^{n+1}- \sigma_t(t_{n+1}) - \nabla \sigma^{n+1}(\textbf{u}^{n+1}-\textbf{u}^n),\\
R^{n+1}_{\u}&= \sigma^{n+1} D_{\tau} (\sigma^{n+1}\textbf{u}^{n+1}) - \sigma(t_{n+1}) (\sigma \textbf{u}) (t_{n+1}) \\
&+\rho^{n+1} (\textbf{u}^n-\textbf{u}^{n+1}) \cdot \nabla \textbf{u}^{n+1}+ \frac{\textbf{u}^{n+1}}{2} (\textbf{u}^n-\textbf{u}^{n+1}) \nabla \rho^{n+1},\\
R^{n+1}_{\theta} &= \sigma^{n+1} D_{\tau} (\sigma^{n+1} \theta^{n+1})  - \sigma(t_{n+1}) (\sigma \theta) (t_{n+1}) \\
&+ \rho^{n+1} (\textbf{u}^n-\textbf{u}^{n+1}) \nabla \theta^{n+1} + \frac{\theta^{n+1}}{2}(\textbf{u}^n - \textbf{u}^{n+1}) \nabla \rho^{n+1},
\end{align*}

Assume that the solutions to the system (\ref{newns1})-(\ref{newns3}) satisfy the following regularities.

{\bf (A1)}: \ Assume that the prescribed data $\f$, $\u_0$ and $\rho_0$ satisfy
	\begin{align*}  
		\rho_0\in H^3(\Omega), \quad 	\u_0\in\V_0\cap\H^{2}(\Omega), \quad \theta_0 \in H^1_0(\Omega)\cap H^{2}(\Omega),\quad
		\f\in L^\infty(0,T; \H^1(\Omega)).
	\end{align*}

{\bf (A2)}: \
Assume that the solution $(\rho, \u, p)$ is sufficiently smooth such that
	\begin{eqnarray*}
		\rho &\in & L^\infty(0, T; H^{3}(\Omega)), \ \rho_t\in L^2(0,T; H^2(\Omega)), \  \rho_{tt}\in L^2(0,T; L^2(\Omega)), \\
		\u &\in & L^\infty(0, T; \H^{3}(\Omega)\cap \V_0), \
		\u_t \in L^\infty(0,T; \H^1(\Omega))\cap L^2(0,T; \H^2(\Omega)),
		\\
		\u_{tt} &\in & L^2(0,T; \L^2(\Omega)), \  p \in L^\infty(0, T; H^2(\Omega)\cap M),\\
		\theta& \in& L^{\infty} (0,T;H^3(\Omega)),  \  \theta_t \in L^{\infty} (0,T;H^2(\Omega)), \ \theta_{tt} \in L^{\infty} (0,T;L^2(\Omega)).
	\end{eqnarray*}

\begin{remark}
As we know that the solution can achieve the $H^2$ regularity if the initial data is suficiently smooth in a convex domain such as rectangle. But whether the solution can have the $H^3$ regularity in a convex polygon domain is still an open problem. We make this regularity assumptions merely to mainly focuses on the error analysis, the strong
regularity conditions have recently been assumed in \cite{MR2802553,li2023error,MR4221298}.
\end{remark}
By the regularity assumption $(\textbf{A2})$ and the Taylor expression, we have 
\begin{align}\label{eq18}
	\tau \sum_{i=1}^{N-1} (\| R^{n+1}_{\sigma} \|^2_{L^2} + \|R^{n+1}_{\u} \|^2_{L^2}+R^{n+1}_{\theta} \|^2_{L^2} ) \leq C \tau^4,
\end{align}

\section{Finite element approximations}
\label{finiteelementapproximation}

\subsection{Finite element discretization scheme}

In this section, we present the finite element discretization of equations (\ref{newns1})-(\ref{eq17}). Let $\mathcal{T}_h = \{K_j\}_{j=1}^L$ denote a quasi-uniform tetrahedral partition of $\Omega$, where the mesh size is given by $ h = \max_j \{ \operatorname{diam} K_j \} $.  For the velocity field $\mathbf{u}$ and the pressure $p$, we employ the mini element ($P_1b-P_1$), ensuring stability and accuracy. The density $\rho$  and the temperature $\theta$ are approximated using the piecewise linear Lagrange element ($P_1$).
The finite element subspaces of $\V, M$ and $W, F$ are denoted by $\V_h\subset\V,$ $M_h\subset M$ and
$W_h\subset W$, $F_h \in F$, respectively.
For this choice, the finite element spaces $\V_h$ and $M_h$  satisfy the discrete inf-sup condition.
Furthermore, we introduce the $\H(\div, \Omega)$ conforming Raviart-Thomas finite element spaces of order $1$ by
\begin{eqnarray*}
\RT_h&=&\{ \u_h\in \H(\div, \Omega), \ \u_h|_K\in P_1(K)^3 + x P_1(K), \ \forall \ K\in \mathcal T_h \},\\
\RT_{0h}&=&\{ \u_h\in \RT_h, \ \nabla\cdot\u_h=0 \ \mbox{in} \ \Omega \ \mbox{and} \ \u_h\cdot \n=0 \ \mbox{on} \ \partial\Omega  \}.
\end{eqnarray*}
We denote by $\P_{0h}$ the $L^2$-orthogonal projection operator from $\L^2(\Omega)$ to $\RT_{0h}$ defined by
\begin{align*}
        (\u-\P_{0h}\u, \v_h)=0,\qquad \forall \ \v_h\in \RT_{0h}, \ \u\in\L^2(\Omega).
\end{align*}

For $1\leq n\leq N$, we
introduce the following  projection operators $(\R_h, Q_h):\V\times M\rightarrow \V_h\times M_h$, $T_h : F \rightarrow F_h$ and
 $\Pi_h: W\rightarrow W_h$ respectively, by
\begin{align*}
( \nabla(\R_h\u^n-\u^n), \nabla\v_h) - (\nabla\cdot \v_h, Q_h p^n-p^n)&=0, \quad \forall \ \v_h\in\V_h,\\
(\nabla\cdot(\R_h\u^n-\u^n), q_h)&=0,\quad\forall \ q_h\in M_h
\end{align*}
and
\begin{align*}
        (\Pi_h \sigma^n-\sigma^n, r_h)  =0,\quad\forall \ r_h\in W_h,\\
  	(\nabla (T_h \theta ^n - \theta^n), \nabla w_h) =0 , \quad  \forall w_h \in F_h.
  \end{align*}
where $\sigma^n$ , $\u^{n}$ and $ \theta^n$ are solutions to (\ref{ExactScheme1})-(\ref{ExactScheme3}).

Denote
$$\rho^n-\rho_h^n=\rho^n-\Pi_h\rho^n+\Pi_h\rho^n-\rho_h^n=-e^n_{\rho}+e^n_{\rho,h},$$
$$\sigma^n-\sigma_h^n=\sigma^n-\Pi_h\sigma^n+\Pi_h\sigma^n-\sigma_h^n=-e^n_{\sigma}+e_{\sigma,h}^n,$$
$$\u^n-\u^n_h=\u^n-\R_h\u^n+\R_h\u^n-\u^n_h=-\e^n_{\textbf{u}}+\e_{\textbf{u},h}^n,$$
$$p^n-p^n_h=p^n-Q_h p^n+Q_h p^n-p^n_h=-e^n_p+e^n_{p,h},$$
$$
\theta^n-\theta^n_h=\theta-T_h \theta^n+T_h \theta^n-\theta^n_h=-e^n_{\theta}+e^n_{\theta,h}.
$$
By the regularities assumption (\textbf{A2}) of $(\sigma^n, \u^n, p^n, \theta^n)$,
the following approximations hold:
\begin{eqnarray}
        \label{u_projection_error}
  \|\e^n_{\textbf{u}}\|_{L^2}+h\|\nabla  \e^n_{\textbf{u}}\|_{L^2}+h\| e^n_{p}\|_{L^2} &\leq& Ch^{2}(\|\u^n\|_{H^2}+\|p^n\|_{H^1}), \\
     \label{App2}
 \| e^n_{\sigma}\|_{L^2}+\|e^n_{\rho}\|_{L^2} +h\|e^n_{\sigma}\|_{H^1} +h\|e^n_{\rho}\|_{H^1}        &\leq& Ch^3\|\sigma^n\|_{H^3},\\
 \label{eq22}
 \| e^n_{\theta} \|_{L^2} + h \| \nabla  e^n_{\theta} \|_{L^2} &\leq &C h^2 \| \theta^n \|_{H^2}.
 \label{projection_error_theta}
\end{eqnarray}
Furthermore, one has
\begin{eqnarray}
        \label{App9}  \|D_{\tau} e^n_{\sigma}\|_{L^2}  &  \leq &
                Ch^2 \|D_{\tau}\sigma^n\|_{H^{2}} , \\
                \label{App8}
         \|D_{\tau} \e^n_{\textbf{u}} \|_{L^2} & \leq &
                Ch^2(  \|D_{\tau}\u^n\|_{H^{2}}+\|D_{\tau}p^n\|_{H^{1}}),\\
                 \|D_{\tau} e^n_{\theta}\|_{L^2}  &  \leq &
                Ch^2 \|D_{\tau}\theta^n\|_{H^{2}}. \label{D_tau_theta_error}
       \end{eqnarray}

We denote by $\P_{1h}$ the standard Raviart-Thomas projection from $\H(\div,\Omega)$ onto $\RT_h$, which
satisfies the following properties (cf. \cite{thomee2006}):
\begin{eqnarray*}
   (\nabla\cdot \P_{1h}\u, v_h)  & = &   (\nabla\cdot  \u, v_h),\qquad \forall \ v_h\in P_1(\mathcal T_h),\\
   \|\u-\P_{1h}\u\|_{L^2} & \leq & C h^l \|\u\|_{H^{l}} ,\qquad \forall \ \u\in \H^l(\Omega), l=1,2,
\end{eqnarray*}
where $P_1(\mathcal T_h)\subset H^1(\Omega)$ is the finite element space of functions which
are the piecewise linear polynomials on each $K\in \mathcal T_h$.
For the time discrete solution $\u^n$, since $\nabla\cdot\u^n=0$ in $\Omega$ and $\u^n\cdot \n =0 $ on $\partial\Omega$,
then
\begin{align*}
        \nabla\cdot\P_{1h}\u^n=0 \ \mbox{in} \ \Omega \quad \mbox{and} \quad \P_{1h}\u^n\cdot \n=0 \ \mbox{on} \ \partial\Omega,
\end{align*}
which imply that $\P_{1h}\u^n\in \RT_{0h}$. By noticing the definition of the $L^2$-projection $\P_{0h}$, there holds that
\begin{align}
        \label{ph_projection_error}
    \|\u^n-\P_{0h}\u^n\|_{L^2}  \leq   \|\u^n-\P_{1h}\u^n\|_{L^2}   \leq Ch^2.
\end{align}

The following inverse inequality
will be  used frequently \cite{scott1994}:
\begin{align}
	&\| \u_h \|_{W^{m,q}} \leq C h^{  l-m + n( \frac{1}{q} -\frac{1}{p}  )} \| \u_h\| _{W^{l,p}}, \quad \forall~ \u_h \in \V_h,  \label{inverse1}\\
&	\| \rho_h \|_{W^{m,q}} \leq C h^{  l-m + n( \frac{1}{q} -\frac{1}{p}  )} \| \rho_h\| _{W^{l,p}}, \quad \forall~ \rho_h \in W_h,\label{inverse2}\\
&	\| \theta_h \|_{W^{m,q}} \leq C h^{  l-m + n( \frac{1}{q} -\frac{1}{p}  )} \| \theta_h\| _{W^{l,p}}, \quad \forall~ \theta_h \in F_h.\label{eq20}
\end{align}

Start with $\u_h^0=I_h\u_0$, $\sigma_h^0=J_h\sigma_0$  and $\theta^0_h=K_h \theta_0$ where $I_h$, $J_h$ and $K_h$ all are interpolation operator from $\V$ onto $\V_h$, $\W$ onto $\W_h$ and $F$ onto $F_h$ respectively. Then
\begin{eqnarray}
        \label{u_inverse_inequality} \|\u_0-\u_h^0\|_{L^2}+h\|\nabla(\u_0-\u_h^0)\|_{L^2}&\leq &Ch^{2}\|\u_0\|_{H^2}, \\
        \label{App5}  \|\sigma_0-\sigma_h^0\|_{L^2}+h\|\sigma_0-\sigma_h^0\|_{H^1}&\leq &Ch^{2}\|\sigma_0\|_{H^2},\\
        \|\theta_0-\theta_h^0\|_{L^2}+h\|\nabla(\theta_0-\theta_h^0)\|_{L^2}&\leq &Ch^{2}\|\theta_0\|_{H^2}.
\end{eqnarray}

 For $0\leq n\leq N-1$, the finite element approximations of (\ref{ExactScheme1})-(\ref{ExactScheme3})
are described as follows.

  \textbf{Step I:} \ For given $\sigma_h^{n+1}\in W_h$ and $\u_h^n\in \V_h$, we find $\sigma_h^{n+1}\in W_h$ by
\begin{align}\label{spaceScheme1}
        (D_\tau \sigma_h^{n+1},r_h) + (\nabla\sigma_h^{n+1}\cdot \P_{0h}\u_h^n,r_h)=0, \quad \forall \ r_h \in W_h.
\end{align}

\textbf{Step II:} \ We find
$(\u_h^{n+1}, p_h^{n+1})\in \V_h\times M_h$ by
\begin{multline}\label{spaceScheme2}
      ( \sigma_h^{n+1}  D_\tau (\sigma_h^{n+1}\u_h^{n+1}),\v_h)  + \mu( \nabla\u_h^{n+1},\nabla\v_h)
        +  ( \rho_h^{n+1}(\u_h^n\cdot\nabla)\u_h^{n+1},\v_h) + \dfrac{1}{2}(\u_h^{n+1} \nabla\cdot(\rho_h^{n+1}\u_h^n),\v_h)\\
  -(\nabla\cdot \v_h, p_h^{n+1} )+(\nabla\cdot\u_h^{n+1}, q_h)= (\f^{n+1},\v_h)
\end{multline}
for any $(\v_h, q_h)\in \V_h\times M_h$, where $\rho_h^{n+1}=(\sigma_h^{n+1})^2$.

\textbf{Step III:} Find $\theta^{n+1}_h \in K_h $ such that
\begin{multline}\label{spaceScheme3}
      ( \sigma_h^{n+1}  D_\tau (\sigma_h^{n+1}\theta_h^{n+1}),w_h)  + \nu( \nabla \theta^{n+1},\nabla w_h)
        +  ( \rho_h^{n+1}(\u_h^n\cdot\nabla)\theta_h^{n+1},w_h) \\
     + \dfrac{1}{2}(\theta_h^{n+1} \nabla\cdot(\rho_h^{n+1}\u_h^n),w_h)
      = (\textbf{g}^{n+1},w_h), \quad \forall w_h \in F_h.
\end{multline}

\begin{remark}
        In (\ref{spaceScheme1}), 
        The post-processed velocity $\P_{0h}\u_h^n$ is used to preserve the
        unconditional stability of numerical scheme $(\ref{spaceScheme1}) - (\ref{spaceScheme3})$.
\end{remark}

\subsection{Stability Result}

\begin{lemma}\label{stabilitylemma}
	For $0 \leq n \leq N-1 $ and  $\tau >0 $ , $h>0$, the finite element discrete scheme $(\ref{spaceScheme1}) - (\ref{spaceScheme3})$ has a unique solution  $(\sigma^{n+1}_h, \textbf{u}^{n+1}_h, p^{n+1}_h, \theta^{n+1}_h)  \in  W_h \times \V_h\times M_h \times K_h $. Moreover, the discrete energy inequalities hold
	\begin{align}
\| \sigma ^{n+1}_h \| ^2_{L^2}& \leq \| \sigma ^{0}_h \| ^2_{L^2} , \label{stability1}\\
			\| \sigma^{n+1}_h \theta^{n+1}_h \| ^2_{L^2} +
			\kappa \tau \sum \limits^{N-1} \limits_{n=0}     \| \nabla \theta^{n+1}_h \| ^2 _{L^2}& \leq \tau \sum \limits^{N-1} \limits_{n=0} \| g^{n+1}\|^2_{L^2} +\| \sigma^0_h \theta^0_h \| ^2_{L^2} .\label{stability2}
	\end{align}
	and 
	\begin{align}\label{stability3}
			\| \sigma^{n+1}_h \u^{n+1}_h \| ^2_{L^2} +
			2 \tau \sum \limits^{N-1} \limits_{n=0}    \mu \| \nabla \u^{n+1}_h \| ^2 _{L^2}
			\leq \| \sigma^0_h \u^0_h \| ^2_{L^2} +\tau \sum \limits^{N-1} \limits_{n=0}\| \f^{n+1} \|^2 _{L^2}. 
	\end{align}
\end{lemma}

\begin{proof}

Letting $r_h=2\tau \sigma^{n+1}_h$ in (\ref{spaceScheme1}), we have
\begin{align}
\| \sigma ^ {n+1}_h \| ^2_{L^2} - \| \sigma ^n_h \| ^2 _{L^2}+\| \sigma^{n+1}_h - \sigma ^n_h \|^2_{L^2} =0
\end{align}
by using $$(\nabla \sigma ^{n+1}_h \cdot P_{0h} \u^n_h,\sigma^{n+1}_h)=  \frac{1}{2}\dint_{\Omega}\nabla | \sigma^{n+1}_h  | ^2 \cdot P_{0h} \u^n_h dx=- \frac{1}{2}\int_\Omega | \sigma^{n+1}_n |^2 \nabla \cdot P_{0h} \u^n_h dx =0.$$

Taking the sum gives (\ref{stability1}).
 Setting $w_h=2\tau \sigma^{n+1}_h \theta^{n+1}_{h}$  in (\ref{spaceScheme3}), we have
\begin{align}
	\| \sigma ^{n+1}_h \theta^{n+1}_h \| ^2 _{L^2} - \| \sigma ^{n}_h \theta^{n}_h \| ^2 _{L^2} +\| \sigma ^{n+1}_h \theta^{n+1}_h -  \sigma ^{n}_h \theta^{n}_h \| ^2 _{L^2} + 2\kappa \tau \| \nabla \theta^{n+1}_h \| ^2 _{	L^2 } 
	=  \tau \| g^{n+1}\|_{L^2}^2 + \kappa \tau \| \nabla \theta^{n+1}_h\|^2_{L^2},
\end{align}
where
\begin{align}
2 \tau (\rho ^{n+1} _h (\u^n_h \cdot \nabla ) \theta^{n+1}_h, \theta^{n+1}_h ) &= \tau ( \rho ^{n+1} _h \u^n_h ,\nabla  \cdot | \theta^{n+1}_h | ^2) = - \tau (\theta^{n+1}_h \nabla \cdot (\rho^{n+1}_h   \u^{n}_h), \theta^{n+1}_h),\nonumber\\
2\tau( g^{n+1},\theta^{n+1}_h ) &\leq   \tau \| g^{n+1}\|_{L^2}^2 + \kappa \tau \| \nabla \theta^{n+1}_h\|^2_{L^2},
\end{align}

Taking the sum gives (\ref{stability2}). Setting $(v_h,q_h)=2\tau (\u^{n+1}_h,p^{n+1}_h)$  in (\ref{spaceScheme2}), we have :
\begin{align}
\| \sigma ^{n+1}_h \u^{n+1}_h \| ^2 _{L^2} - \| \sigma ^{n}_h \u^{n}_h \| ^2 _{L^2} +\| \sigma ^{n+1}_h \u^{n+1}_h -  \sigma ^{n}_h \u^{n}_h \| ^2 _{L^2} + 2\mu \tau \| \nabla \u^{n+1}_h \| ^2 _{	L^2 }
= 2\tau (\f^{n+1} , \u^{n+1}_h),
\end{align}
by using
$$2 \tau (\rho ^{n+1} _h (\u^n_h \cdot \nabla ) \u^{n+1}_h, \u^{n+1}_h ) = \tau ( \rho ^{n+1} _h \u^n_h ,\nabla  \cdot | \u^{n+1}_h | ^2) = - \tau (\u^{n+1}_h \nabla \cdot (\rho^{n+1}_h   \u^{n}_h), \u^{n+1}_h)$$
and 
\begin{align}
	2\tau (\f^{n+1} , \u^{n+1}_h) \leq \tau \|\textbf{f}^{n+1}\|_{L^2} +  \mu \tau \| \nabla \textbf{u}^{n+1}\|_{L^2}^2.
\end{align}
Taking the sum gives (\ref{stability3}), we complete the proof of Lemma \ref{stabilitylemma}. Furthermore, since the sub-problems $(\ref{spaceScheme1}) - (\ref{spaceScheme3})$ are linear problem, the discrete energy inequalities not only ensure the unconditional
stability of the proposed algorithm but also imply the existence and uniqueness of numerical solution
$(\sigma^{n+1}_h, \textbf{u}^{n+1}_h, p^{n+1}_h, \theta^{n+1}_h)$ to the back Euler finite discrete scheme $(\ref{spaceScheme1}) - (\ref{spaceScheme3})$.

\end{proof}

\section{Error Estimate}
Now we will continue the main work of this paper, we need to estimate $\|\e^{n+1}_{\u,h}\|_{L^2}, \|e^{n+1}_{\sigma,h}\|_{L^2}$ and $\|e^{n+1}_{\theta}\|_{L^2}$  based on the mathematical induction method.
Letting $(r,\textbf{v},w) =(r_h,\textbf{v}_h,w_h) $ and taking the difference between$ ~(\ref{ExactScheme1}) -(\ref{ExactScheme3}) $ and $( \ref{spaceScheme1})-(\ref{spaceScheme3}) $, then we get the following error equation:

\begin{align}\label{errorequation1}
        &(D_\tau e_{\sigma,h}^{n+1},r_h) + 
 (\nabla\sigma^{n+1}\cdot (\u^n- \P_{0h}\u_h^n),r_h) 
+  ( \nabla e^{n+1}_{\sigma,h}\cdot (\P_{0h}\u_h^n-\u^n),r_h) ,\\
+& (\nabla e_{\sigma,h}^{n+1}\cdot \u^n,r_h) 
- (\nabla e^{n+1}_\sigma \cdot (\textbf{P}_{0h} \u^n_h-\u^n),r_h)- (\nabla e^{n+1}_\sigma \cdot \u^n,r_h)
         =(R^{n+1}_{\sigma} ,r_h), \qquad \forall \ r_h \in W_h \nonumber
\end{align}
and
\begin{align}\label{errorequation2}
      &( \sigma_h^{n+1}  D_\tau (\sigma_h^{n+1}\e_{\textbf{u},h}^{n+1}),\v_h)
      + \mu( \nabla\e_{\textbf{u},h}^{n+1},\nabla\v_h)-(\nabla\cdot \v_h,e_{p,h}^{n+1} ) +(\nabla\cdot \e_{\textbf{u},h}^{n+1},q_h)     \nonumber \\
=& (\sigma^{n+1}_h D_{\tau} (\sigma^{n+1}_h \e^{n+1}_{\textbf{u}}) ,\textbf{v}_h )+(e^{n+1}_{\sigma} D_{\tau}(\sigma^{n+1} \textbf{u}^{n+1}),\textbf{v}_h)+(\sigma^{n+1}_h e^{n+1}_{\sigma} D_{\tau}\textbf{u}^{n+1} ,\textbf{v}_h)\nonumber\\
-& (\sigma^{n+1}_{h}e^{n+1}_{\sigma,h} D_{\tau}\textbf{u}^{n+1} ,\textbf{v}_h  )+(\sigma^{n+1}_h D_{\tau}e^{n+1}_{\sigma} \textbf{u}^n,\textbf{v}_h) - (\sigma^{n+1}_h D_{\tau}e^{n+1}_{\sigma,h} \textbf{u}^n,\textbf{v}_h)\nonumber\\
-&(e^{n+1}_{\sigma,h} D_{\tau} (\sigma^{n+1}\textbf{u}^{n+1}) ,\textbf{v}_h ) +( e^{n+1}_{\rho} (\textbf{u}^n\cdot \nabla)\textbf{u}^{n+1},\textbf{v}_h)-( e^{n+1}_{\rho,h} (\textbf{u}^n\cdot \nabla)\textbf{u}^{n+1},\textbf{v}_h)\nonumber\\
+& (\rho^{n+1}_h (\e^n_{\textbf{u}} \cdot \nabla) \textbf{u}^{n+1},\textbf{v}_h)- (\rho^{n+1}_h (\e^n_{\textbf{u},h} \cdot \nabla) \textbf{u}^{n+1},\textbf{v}_h) + ( \rho^{n+1}_h ( \textbf{u}^n_h \cdot \nabla) \e^{n+1}_{\textbf{u}} ,\textbf{v}_h)\nonumber\\
-&( \rho^{n+1}_h ( \textbf{u}^n_h \cdot \nabla) \e^{n+1}_{\textbf{u},h} ,\textbf{v}_h)+ \frac{1}{2} ( \textbf{u}^{n+1} \nabla \cdot (e^{n+1}_{\rho} \textbf{u}^n),\textbf{v}_h) -\frac{1}{2} ( \u^{n+1} \nabla \cdot (e^{n+1}_{\rho,h} \textbf{u}^n),\textbf{v}_h)\nonumber\\
+&\frac{1}{2} ( \textbf{u}^{n+1} \nabla \cdot (\rho^{n+1}_h \e^{n}_{\textbf{u} })  ,\textbf{v}_h) - \frac{1}{2} ( \textbf{u}^{n+1} \nabla \cdot (\rho^{n+1}_h \e^{n}_{\textbf{u},h})  ,\textbf{v}_h)+\frac{1}{2} (e^{n+1}_{\textbf{u},h} \nabla \cdot (\rho^{n+1}_h \textbf{u}^n_h),\textbf{v}_h) \nonumber\\
+& \frac{1}{2} (\e^{n+1}_{\textbf{u}} \nabla \cdot (\rho^{n+1}_h \textbf{u}^n_h),\textbf{v}_h)  + (R^{n+1}_{\textbf{u}},\textbf{v}_h )  =\sum\limits^{20}\limits_{i=1}(X_i ,\textbf{v}_h), \quad \forall  \ (\v_h, q_h) \in \V_h\times M_h, 
\end{align}
and
\begin{align}\label{errorequation3}
&(\sigma^{n+1}_h D_{\tau} (\sigma^{n+1} e^{n+1}_{\theta,h})  ,w_h)  + \kappa (\nabla e^{n+1}_{\theta},\nabla w_h)\nonumber\\
=& (\sigma^{n+1}_h D_{\tau} (\sigma^{n+1}_h e^{n+1}_{\theta}) ,w_h )+(e^{n+1}_{\sigma} D_{\tau}(\sigma^{n+1} \theta^{n+1}),w_h)+(\sigma^{n+1}_h e^{n+1}_{\sigma} D_{\tau}\theta^{n+1} ,w_h)\nonumber\\
-& (\sigma^{n+1}_{h}e^{n+1}_{\sigma,h} D_{\tau}\theta^{n+1} ,w_h  )+(\sigma^{n+1}_h D_{\tau}e^{n+1}_{\sigma} \theta^n,w_h) - (\sigma^{n+1}_h D_{\tau}e^{n+1}_{\sigma,h} \theta^n,w_h)\nonumber \\
-&(e^{n+1}_{\sigma,h} D_{\tau} (\sigma^{n+1}\theta^{n+1}) ,w_h ) +( e^{n+1}_{\rho} (\textbf{u}^n\cdot \nabla)\theta^{n+1},w_h)-( e^{n+1}_{\rho,h} (\textbf{u}^n\cdot \nabla)\theta^{n+1},w_h)\nonumber\\
+& (\rho^{n+1}_h (e^n_{\textbf{u}} \cdot \nabla) \theta^{n+1},w_h)- (\rho^{n+1}_h (e^n_{\textbf{u},h} \cdot \nabla) \theta^{n+1},w_h) + ( \rho^{n+1}_h ( \textbf{u}^n_h \cdot \nabla) e^{n+1}_{\theta} ,w_h)\nonumber\\
-&( \rho^{n+1}_h ( \textbf{u}^n_h \cdot \nabla) e^{n+1}_{\theta,h} ,w_h)+ \frac{1}{2} ( \theta^{n+1} \nabla \cdot (e^{n+1}_{\rho} \textbf{u}^n),w_h) -\frac{1}{2} ( \theta^{n+1} \nabla \cdot (e^{n+1}_{\rho,h} \textbf{u}^n),w_h)\nonumber\\
+&\frac{1}{2} ( \theta^{n+1} \nabla \cdot (\rho^{n+1}_h e^{n}_{\textbf{u} })  ,w_h) - \frac{1}{2} ( \theta^{n+1} \nabla \cdot (\rho^{n+1}_h e^{n}_{\textbf{u},h})  ,w_h)+\frac{1}{2} (e^{n+1}_{\theta,h} \nabla \cdot (\rho^{n+1}_h \textbf{u}^n_h),w_h)\nonumber\\
+& \frac{1}{2} (e^{n+1}_{\theta} \nabla \cdot (\rho^{n+1}_h \textbf{u}^n_h),w_h) + (R^{n+1}_{\theta},w_h )  = \sum\limits^{20}\limits_{i=1}(Y_i ,w_h), \quad \forall w_h \in F_h.
\end{align}

Before the estimate of $\e^{n+1}_{\u,h}$ and $e^{n+1}_{\theta,h}$, we need to give the following lemmas.
\begin{lemma}\label{lemma1}
        Under the assumptions ({\bf A1}) and  ({\bf A2}),
        there exists small small $\tau_1\leq \tau_0$ such that
        when $\tau\leq \tau_1$, one has
        \begin{equation}\label{eq1}
                \| \e_{\sigma,h}^{n+1}\|^2_{L^2}+\sum_{n=0}^{N-1}\| \e_{\sigma,h}^{n+1}-\e_{\sigma,h}^{n}\|^2_{L^2}
        \leq C(\tau^2+h^4)+C\tau\sum_{n=0}^m\| \u^n-\P_{0h}\u_h^n\|^2_{L^2} ,
              \end{equation}
        for all $0\leq m\leq N-1$.
\end{lemma}

\begin{proof} Taking $r_h=2\tau e^{n+1}_{\sigma,h}$ in (\ref{errorequation1}) yields
\begin{eqnarray*}
&&\| e_{\sigma,h}^{n+1}\| ^2_{L^2}-\| e_{\sigma,h}^{n}\| ^2_{L^2}+\|e_{\sigma,h}^{n+1}-e_{\sigma,h}^{n}\| ^2_{L^2} \\
&\leq& C\tau\| \nabla\sigma^{n+1}\| _{L^{\infty}}\| \u^n-\P_{0h}\u_h^n\| _{L^2}\| e_{\sigma,h}^{n+1}\| _{L^2}
+C\tau\| \nabla e^{n+1}_{\sigma}\| _{L^{2}}\| \u^n-\P_{0h}\u_h^n\| _{L^2}\| e_{\sigma,h}^{n+1}\| _{L^\infty}\nn\\
&& +C\tau\| \nabla e^{n+1}_{\sigma}\| _{L^{2}}\| \u^n\| _{L^{\infty}}\| e_{\sigma,h}^{n+1}\| _{L^2} + C \tau\|R^{n+1}_{\sigma}\|_{L^2}\|e^{n+1}_{\sigma}\|_{L^2}\\
&\leq& C\tau h^4+C\tau\| \u^n-\P_{0h}\u_h^n\|^2_{L^2}+C\tau\| e_{\sigma,h}^{n+1}\| ^2_{L^2}+ C\tau \|R^{n+1}_{\sigma}\|^2_{L^2},
\end{eqnarray*}
where we noted
$$(\nabla e_{\sigma,h}^{n+1}\cdot (\P_{0h}\u_h^n-\u^n),e_{\sigma,h}^{n+1}) +(\nabla e_{\sigma,h}^{n+1}\cdot \u^n,e_{\sigma,h}^{n+1})=0 $$
by using the integration by parts.
Summing up the above inequality and using the discrete Gronwall inequality
in Lemma \ref{GronwallLemma}, there exists some small $\tau_1\leq \tau_0$,
such that the inequality (\ref{eq1}) holds. Thus, we complete the proof of Lemma \ref{lemma1}.
\end{proof}

\begin{lemma} \label{lemma2}
        Under the assumptions in Lemma \ref{lemma1}, for $0\leq n\leq N-1$,  if
\begin{align}\label{E}\| \P_{0h}\u_h^n-\u^n\| _{L^{2}}\leq C_p h^{2}
\end{align}
for some $C_p>0$,
then for sufficiently small $h$, we have
\begin{align}
        \label{eq7}
\| D_{\tau} e_{\sigma,h}^{n+1}\|^2_{L^2}\leq Ch^2 +C \|R^{n+1}_\sigma\|^2_{L^2}+
        C\tau h^{-2}\sum_{i=0}^{n}\| \u^n-\P_{0h}\u_h^n\| ^2_{L^2}.
\end{align}
\end{lemma}

\begin{proof}
By the time step condition $\tau \leq C h^2$, it follows from (\ref{eq1}) and (\ref{E}) that
	\begin{align}
		\label{etaL2I}
		\|e^{n+1}_{\sigma,h}\|_{L^2}  \leq C (1+C_p) h^2.
	\end{align}
	Taking $r_h=D_{\tau}e^{n+1}_{\sigma,h}$ in (\ref{errorequation1}) yields
	\begin{align}\label{D}
		\| D_{\tau}e^{n+1}_{\sigma,h}\| ^2_{L^2}\leq \sum_{j=1}^{6} |(J_{ih}^{n+1},D_{\tau}e^{n+1}_{\sigma,h})|.
	\end{align}
	The right-hand side of (\ref{D}) can be bounded term by term as follows. For the first term, one has
	\begin{eqnarray*}
		|(J_{1h}^{n+1},D_{\tau}e^{n+1}_{\sigma,h})|
		&\leq & \| \nabla\sigma^{n+1}\| _{L^{\infty}}\| \u^n-\P_{0h}\u_h^n\| _{L^2}\| D_{\tau}e^{n+1}_{\sigma,h}\| _{L^2}
		\\
		&\leq & \frac{1}{10}\| D_{\tau}e^{n+1}_{\sigma,h}\| _{L^2}^2+ C \| \u^n-\P_{0h}\u_h^n\| ^2_{L^2}\\
		&\leq & \frac{1}{10}\| D_{\tau}e^{n+1}_{\sigma,h}\| _{L^2}^2+ C h^4,
	\end{eqnarray*}
	
Adapting the above same technique, by using (\ref{etaL2I}) and the time step condition $\tau \leq C h^2$, we have
	\begin{eqnarray*}
		|(J_{2h}^{n+1},D_{\tau}e^{n+1}_{\sigma,h}) |
		&\leq & C  \| \P_{0h}\u_h^n-\u^n\| _{L^{2}}\|\nabla e^{n+1}_{\sigma,h}\| _{L^2}\| D_{\tau}e^{n+1}_{\sigma,h}\| _{L^\infty} \\
		&\leq & C C_p h^{-\frac{1}{2}}\| e^{n+1}_{\sigma,h}\| _{L^2}\| D_{\tau}e^{n+1}_{\sigma,h}\| _{L^2} \\
		&\leq & \frac{1}{10}\| D_{\tau}e^{n+1}_{\sigma,h}\| _{L^2}^2+ C C_p^2 h^{-1}\| e^{n+1}_{\sigma,h}\| _{L^2}^2 \\
		&\leq & \frac{1}{10}\| D_{\tau}e^{n+1}_{\sigma,h}\| _{L^2}^2+ C h^2
	\end{eqnarray*}
for sufficiently small $h$ such that $ C^2 C^2_p(1+C_p)^2 h \leq 1$.
	\begin{eqnarray*}
	|(J_{3h}^{n+1},D_{\tau}e^{n+1}_{\sigma,h}) |
	&\leq & C h^{-1}\| e^{n+1}_{\sigma,h}\| _{L^2}\| D_{\tau}e^{n+1}_{\sigma,h}\| _{L^2} \\
	&\leq & \frac{1}{10}\| D_{\tau}e^{n+1}_{\sigma,h}\| _{L^2}^2+ Ch^{-2}\| e^{n+1}_{\sigma,h}\| _{L^2}^2
\end{eqnarray*}

By the inverse inequality, we estimate the second term by
\begin{eqnarray*}
	|(J_{4h}^{n+1},D_{\tau}e^{n+1}_{\sigma,h})  |
	&\leq & \| \nabla e^{n+1}_{\sigma}\| _{L^{2}}\| \u^n-\P_{0h}\u_h^n\| _{L^2}\| D_{\tau}e^{n+1}_{\sigma,h}\| _{L^\infty} \nn\\
	&\leq& Ch^{-\frac{3}{2}} \|D_{\tau} e^{n+1}_{\sigma,h} \| h^2 \| \sigma^{n+1}\|_{H^3} \| \u^n-\P_{0h}\u_h^n\| _{L^2}  \nn\\
	&\leq & \frac{1}{10}\| D_{\tau}e^{n+1}_{\sigma,h}\| _{L^2}^2+Ch\| \u^n-\P_{0h}\u_h^n\| ^2_{L^2}\\
	&\leq & \frac{1}{10}\| D_{\tau}e^{n+1}_{\sigma,h}\| _{L^2}^2+ C C_p h^5.
\end{eqnarray*}

	For the last two terms in (\ref{D}), we can estimate by
	\begin{eqnarray*}
		|(J_{5h}^{n+1},D_{\tau}e^{n+1}_{\sigma,h}) |
		&\leq & C\| \nabla e^{n+1}_{\sigma}\| _{L^2}\| D_{\tau}e^{n+1}_{\sigma,h}\| _{L^2} \\
		&\leq & \frac{1}{10}\| D_{\tau}e^{n+1}_{\sigma,h}\| _{L^2}^2+ Ch^4,\\
		|(J^{n+1}_{6h},D_{\tau}e^{n+1}_{\sigma,h}   )| &\leq &C \| R^{n+1}_{\sigma}\|_{L^2}\| D_{\tau}e^{n+1}_{\sigma,h}\| _{L^2} \\
		&\leq & C \|R^{n+1}_\sigma\|^2_{L^2} + \frac{1}{10} \| D_{\tau}e^{n+1}_{\sigma,h} \|^2_{L^2}.
			\end{eqnarray*}
			
	Substituting the above inequalities into (\ref{D}) and using (\ref{eq1}), we get
	\begin{eqnarray*}
		\| D_{\tau}e^{n+1}_{\sigma,h}\|^2_{L^2}
		& \leq & Ch^2+C \|R^{n+1}_\sigma\|^2_{L^2} + Ch^{-2}\| e^{n+1}_{\sigma,h}\| _{L^2}^2 \\
		& \leq & Ch^2+C \|R^{n+1}_\sigma\|^2_{L^2} + C\tau h^{-2}\sum_{i=0}^{n}\| \u^n-\P_{0h}\u_h^n\| ^2_{L^2}.
	\end{eqnarray*}
	Thus, we complete the proof of Lemma \ref{lemma2}.
\end{proof}

Next, we give the estimate of $\| \e_{\u,h}\|_{L^2}$ and $\|e_{\theta,h}\|_{L^2}$ by the mathematical induction.
\begin{lemma}\label{lemma3}
	Under the assumptions (A1), (A2) and the time step condition $\tau \leq C h^2$, there exists small constants $\tau_0 >0$ and $h_0 >0$ such that when $\tau \leq \tau_0$ and $h<h_0$, there holds
	\begin{align}\label{u_n_error}
		\|\e^{n+1}_{\textbf{u},h} \|_{L^2}^2 + \tau \sum_{n=0}^{N-1} \| \nabla \e^{n+1}_{\textbf{u},h} \|^2_{L^2} \leq C_0^2 (\tau^2+h^4)\\
\label{theta_n_error}
		\|e^{n+1}_{\theta,h} \|_{L^2}^2 + \tau \sum_{n=0}^{N-1} \| \nabla e^{n+1}_{\theta,h} \|^2_{L^2} \leq C_0^2 (\tau^2+h^4)
	\end{align}
	for all $1 \leq n \leq N-1$.
\end{lemma}
We need to  prove the validity of  Lemma \ref{lemma3}  by the mathematical induction and the time step condition $\tau \leq C h^2$, we can assume that
\begin{align}
		\|e^n_{\textbf{u},h} \|_{L^2}^2 + \tau \sum_{n=0}^{N-1} \| \nabla e^n_{\textbf{u},h} \|^2_{L^2} \leq C_0^2 h^4,\label{nsre-1}\\
			\|e^n_{\theta,h} \|_{L^2}^2 + \tau \sum_{n=0}^{N-1} \| \nabla e^n_{\theta,h} \|^2_{L^2} \leq C_0^2 h^4,\label{nsre-5}
\end{align}
By the inverse inequality, we have 
\begin{align}
	\|e^n_{\textbf{u},h} \|_{L^{\infty}} \leq C h^{-\frac{3}{2}} \| e^n_{\textbf{u},h} \|_{L^2} \leq C C_0 h^{\frac{1}{2}} \leq C, \label{eq4}\\
	\|e^n_{\theta,h} \|_{L^{\infty}} \leq C h^{-\frac{3}{2}} \| e^n_{\theta,h} \|_{L^2} \leq C C_0 h^{\frac{1}{2}} \leq C,\label{eq14}
\end{align}
thus 
\begin{align}
	\| \textbf{u}^n_{h}\|_{L^{\infty}} \leq  \|\textbf{u}^n \|_{L^{\infty}}  +\| e^n_{\textbf{u}} \|_{L^{\infty}}+ \| e^n_{\textbf{u},h} \|_{L^{\infty}} \leq C,\label{eq5}\\
	\| \nabla \u^n_h \|_{L^3}\leq C \|\nabla \e^n_{\u,h}  \|_{L^3} +C \| \nabla \u^n\|_{L^3} \leq C +C h^{-\frac{3}{2}} \| \e_{\u,h}^n\|_{L^2}\leq C,\label{eq9}\\
	\| \nabla \u^n_h \|_{L^6} \leq C \|\nabla \e^n_{\u,h}  \|_{L^6} +C \| \nabla \u^n\|_{L^6}\leq C + C h^{-2}\| \e_{\u,h}^n\|_{L^2}\leq C,\label{nsre-3}\\
	\| \theta^n_{h}\|_{L^{\infty}} \leq  \| \theta^n \|_{L^{\infty}}  +\| e^n_{\theta} \|_{L^{\infty}}+ \| e^n_{\theta,h} \|_{L^{\infty}} \leq C,\label{eq15}\\
	\| \nabla \theta^n_h \|_{L^3}\leq C \|\nabla e^n_{\theta,h}  \|_{L^3} +C \| \nabla \theta^n\|_{L^3} \leq  C +C h^{-\frac{3}{2}} \| e_{\sigma,h}^n\|_{L^2}\leq C,\label{nsre-2}\\
		\| \nabla \theta^n_h \|_{L^6} \leq C \|\nabla e^n_{\theta,h}  \|_{L^6} +C \| \nabla \theta^n\|_{L^6}\leq C + C h^{-2}\| e_{\theta,h}^n\|_{L^2}\leq C.\label{nsre-6}
\end{align} 

From (\ref{ph_projection_error}) and  (\ref{nsre-1}) , we have 
\begin{align}\label{eq6}
	\| \textbf{u}^n- \textbf{P}_{0h} \textbf{u}^n_h \|^2_{L^2} & \leq  \| \textbf{u}^n - \textbf{P}_{0h} \textbf{u}^n \|_{L^2}^2 + \| \textbf{P}_{0h} \textbf{u}^n - \textbf{P}_{0h} \textbf{u}^n_h \|^2_{L^2} \nonumber\\
	&\leq  Ch^4 + C \| \textbf{u}^n- \textbf{u}^n_h\|^2_{L^2} \nonumber\\
	&\leq C (1+C_0)^2 h^4, 
\end{align}

According to (\ref{eq6}) and (\ref{eq1}) and the time step condition $\tau \leq C h^2$, we have
\begin{align}
	\| e^{n+1}_{\sigma,h} \|^2 \leq & C(\tau^2+h^4)+C\tau\sum_{n=0}^m\| \u^n-\P_{0h}\u_h^n\|^2_{L^2}  \nonumber \\
	\leq & C (1+C_0)^2h^4,\label{eq10}\\
	\| \sigma^{n+1} - \sigma^{n+1}_h \|_{L^{\infty}}  \leq& \| e^{n+1}_{\sigma}\|_{L^{\infty}} +   \| e^{n+1}_{\sigma,h}\|_{L^{\infty}} \nonumber\\
	\leq & C h^{- \frac{3}{2}} (\| e^{n+1}_{\sigma}\|_{L^{2}} +   \| e^{n+1}_{\sigma,h}\|_{L^{2}})\nn\\
	\leq & C (1+C_0)h^{\frac{1}{2}},
\end{align}

Furthermore
\begin{align} \label{eq8}
	 \sqrt{\rho^{min}} - C (1+C_0)h^{\frac{1}{2}}  \leq \sigma^{n+1}_{h} \leq  \sqrt{\rho^{max}} +C (1+C_0)h^{\frac{1}{2}} \leq C,
\end{align}

By (\ref{eq7}) and (\ref{eq6}) and taking sufficiently small $h$, we have 
\begin{align}\label{eq2}
	\| D_{\tau} e^{n+1}_{\sigma} \|_{L^3} \leq C h^{-\frac{1}{2}} 	\| D_{\tau} e^{n+1}_{\sigma} \|_{L^2} \leq C C_0 h^{\frac{1}{2}}\leq C,
\end{align}

By (\ref{eq10}) and inverse inequality (\ref{inverse2}), we can get
\begin{align}\label{nsre-4}
	\| \nabla \sigma^{n+1}_h\|_{L^3} \leq & \| \nabla e^{n+1}_{\sigma,h}\|_{L^3}+\| \Pi_h \sigma^{n+1}\|_{L^3}\notag\\
	\leq & C h^{-\frac{3}{2}} \| e^{n+1}_{\sigma,h} \|_{L^2} +\|\sigma^{n+1} \|_{L^3}\\
	\leq & C (1+C_0)h^{\frac{1}{2}} +C \leq C. \notag
\end{align}

\subsection{Estimate of $\| \e^{n+1}_{\u,h}\|^2_{L^2}$}

Letting $(\textbf{v}_h, q_h) = 2 \tau (e^{n+1}_{\textbf{u},h} , e^{n+1}_p ) $ in (\ref{errorequation2}),  applying H\"{o}lder inequality and Young inequality, we can estimate $\sum\limits^{20}\limits_{i=1}(X_i ,(e^{n+1}_{\textbf{u},h})$ as follows
\begin{itemize}
	\item Estimate of $2\tau (X_1,e^{n+1}_{\textbf{u},h} )$
\end{itemize}

By recombining this term and  using (\ref{eq2}) (\ref{App8}) (\ref{eq8}), Lagrange’s mean value theorem, we have
\begin{align*}
	\begin{split}
		2\tau (X_1 , e^{n+1}_{\textbf{u},h}) =&2 \tau (\sigma ^{n+1} _h D_{\tau} (\sigma ^{n+1} _h e^{n+1}_{\textbf{u}}), e^{n+1}_{\textbf{u},h})\\
		\leq&
		2 \tau | (\rho ^{n+1}_h D_{\tau} e^{n+1}_{\textbf{u}} , e^{n+1} _{\textbf{u},h})|+ 2 \tau |(\sigma ^{n+1} _h e^n_{\textbf{u}} D_{\tau} e^{n+1}_{\sigma,h} ,e^{n+1} _{\textbf{u},h} )|\\
		&+2 \tau| (\sigma^{n+1}_{h} e^n_{\textbf{u}} D_{\tau} (\Pi_h \sigma^{n+1}) , e^{n+1}_{\textbf{u},h})|\\
		\leq & C \tau \| \rho^{n+1}_h\|_{L^{\infty}} \| D_{\tau} e^{n+1}_{\u} \|_{L^2} \| e^{n+1}_{\u,h}\|_{L^2} + C\tau \|\sigma^{n+1}_h\|_{L^{\infty}} \| e^n_\u\|_{L^2} \|D_{\tau} e^{n+1}_{\sigma,h}\|_{L^3}\|e^{n+1}_{\u,h}\|_{L^6} \\
		&+ C\tau\| \sigma^{n+1}_h\|_{L^{\infty}} \| e^n_\u\|_{L^2} \| D_{\tau} (\Pi_h \sigma^{n+1}) \|_{L^3} \|e^{n+1}_{\u,h}\|_{L^6} \\
		\leq&
		C\tau h^4 \| D_{\tau} \textbf{u}^{n+1} \| ^2_{H^2} +C\tau h^4 +\frac{\mu\tau}{30}  \| \nabla e^{n+1}_{\textbf{u},h} \| ^2_{L^2},
	\end{split}
\end{align*}
\begin{itemize}
	\item Estimate of $2\tau (X_2,e^{n+1}_{\textbf{u},h} )$
\end{itemize}

By recombining this term and using (\ref{App2}), Lagrange’s mean value theorem,  one has 
\begin{align*}
	\begin{split}
		2\tau(X_2,e^{n+1}_{\u,h}) =&2\tau (e^{n+1}_\sigma D_{\tau}(\sigma^{n+1}\textbf{u}^{n+1}) , e^{n+1}_{\textbf{u}})\\
		\leq&  2\tau |( e^{n+1}_{\sigma} \sigma^{n+1} D_{\tau}\textbf{u}^{n+1}   , e^{n+1}_{\textbf{u},h})| + 2\tau | ( e^{n+1}_\sigma (D_{\tau}\sigma^{n+1} ) \textbf{u}^{n+1} ,e^{n+1}_{\textbf{u},h} ) |\\
		\leq & C \tau h^4+\frac{\mu \tau}{30} \| \nabla e^{n+1}_{\textbf{u},h}\|^2_{L^2}, 
	\end{split}
\end{align*}

Similar to  $2\tau (X_2,e^{n+1}_{\textbf{u},h} )$, we get 
\begin{itemize}
	\item Estimate of $2\tau (X_3,e^{n+1}_{\textbf{u},h} )$
\end{itemize}

\begin{align*}
	\begin{split}
		2\tau(X_3,e^{n+1}_{\textbf{u},h}) =& 2 \tau(\sigma^{n+1}_h e^{n+1}_{\sigma}D_{\tau}\textbf{u}^{n+1},e^{n+1}_{\textbf{u},h} )\\
		\leq & C\tau h^4 + +\frac{\mu \tau}{30} \| \nabla e^{n+1}_{\textbf{u},h}\|^2_{L^2}, 
	\end{split}
\end{align*}

\begin{itemize}
	\item Estimate of $2\tau (X_4,e^{n+1}_{\textbf{u},h} )$
\end{itemize}
\begin{align*}
	\begin{split}
		2\tau(X_4,e^{n+1}_{\textbf{u},h}) =& 2 \tau(\sigma^{n+1}_h e^{n+1}_{\sigma,h}D_{\tau}\textbf{u}^{n+1},e^{n+1}_{ \textbf{u},h} )\\
		\leq & C\tau h^4 + C\tau \|e^{n+1}_{\sigma,h}\|^2_{L^2},
	\end{split}
\end{align*}
\begin{itemize}
	\item Estimate of $2\tau (X_5,e^{n+1}_{\u,h} )$
\end{itemize}
\begin{align*}
	\begin{split}
		2\tau(X_5,e^{n+1}_{\textbf{u},h}) =&2\tau (\sigma^{n+1}_h D_{\tau}e^{n+1}_{\sigma} \textbf{u} ^n ,e^{n+1}_{\u,h})\\
		\leq & C \tau h^4  \|D_{\tau}\sigma^{n+1} \|_{H^2}^2+ C \tau \| \sigma^{n+1}_h e^{n+1}_{\textbf{u},h} \|^2_{L^2},
	\end{split}
\end{align*}
where we use (\ref{App9}), (\ref{eq8}).
\begin{itemize}
	\item Estimate of $2\tau (X_6,e^{n+1}_{\u,h} )$
\end{itemize}
\begin{align}\label{eq3}
	\begin{split}
		2\tau(X_6,e^{n+1}_{\textbf{u},h}) =& 2 \tau(\sigma^{n+1}_h D_{\tau}e^{n+1}_{\sigma,h} \textbf{u} ^n ,e^{n+1}_{\textbf{u},h})\\
		\leq & 2\tau|( \sigma^{n+1}_h D_{\tau}e^{n+1}_{\sigma,h} e^n_{\textbf{u}}, e^{n+1}_{\textbf{u},h}  )| +2\tau|( \sigma^{n+1}_h D_{\tau}e^{n+1}_{\sigma,h} e^n_{\textbf{u},h}, e^{n+1}_{\textbf{u},h}  )| \\
		& +2\tau|( \sigma^{n+1}_h D_{\tau}e^{n+1}_{\sigma,h} \textbf{u}^n_h, e^{n+1}_{\textbf{u},h}  )| \\
		\leq & C\tau h^4 + \frac{\mu \tau}{30} \|\nabla e^{n+1}_{\textbf{u},h}\|^2_{L^2} + C\tau \|\sigma^n_h e^n_{\u,h}\|^2_{L^2} +2\tau|( \sigma^{n+1}_h D_{\tau}e^{n+1}_{\sigma,h} \textbf{u}^n_h, e^{n+1}_{\textbf{u},h}  )|,
	\end{split}
\end{align}
where we use (\ref{eq2}),(\ref{eq8}), (\ref{nsre-1}), (\ref{eq5}).

In order to estimate the last term $2\tau|( \sigma^{n+1}_h D_{\tau}e^{n+1}_{\sigma,h} \textbf{u}^n_h, e^{n+1}_{\textbf{u},h}  )|$ in (\ref{eq3}), we introduce the piece wise constant finite element space
\begin{align*}
	W_h^0=\{q_h \in L^2(\Omega) |  q_h \in P_0(K) , \quad \forall K \in \mathcal{T}_h\}.
\end{align*}
Let $R_{h}$ be the $L^2$ projection operator from $L^2(\Omega)$ onto $W_h^0$. Then there holds
\begin{align}	\label{Rh}
	\|q-R_hq \|_{L^2} \leq Ch\|q\|_{H^1}\quad
	\mbox{and}\quad
	\|R_h q\|_{L^2}\leq \|q\|_{L^2} .
\end{align}
By (\ref{eq5}),(\ref{eq9}) , one has
\begin{align*}
	\|\nabla(\u_h^n\cdot \e_{\u,h}^{n+1})\|_{L^2}
	\leq \|\nabla\u_h^n\|_{L^3}\|\e^{n+1}_{\u,h}\|_{L^6}+\|\u_h^n\|_{L^\infty}\|\nabla\e^{n+1}_{\u,h}\|_{L^2}
	\leq  C \|\nabla\e_{\u,h}^{n+1}\|_{L^2}.
\end{align*}

It follows from (\ref{Rh}) that 
\begin{align}
	\label{Rh1}
	\|(\u_h^n\cdot \e_{\u,h}^{n+1})-R_h(\u_h^n\cdot \e_{\u,h}^{n+1})\|_{L^2}
	\leq Ch \|\nabla\e_{\u,h}^{n+1}\|_{L^2}.
\end{align}

Thus
\begin{eqnarray}\label{eq13}
	2\tau |(\sigma_h^{n+1}D_{\tau} e^{n+1}_{\sigma,h} \u_h^n,\e_{\u,h}^{n+1}) | & \leq &
	2\tau |(\sigma_h^{n+1}D_{\tau}e^{n+1}_{\sigma,h} , R_h(\u_h^n\cdot \e_{\u,h}^{n+1}) ) | \nn \\
	& &+  2\tau |(\sigma_h^{n+1}D_{\tau}e^{n+1}_{\sigma,h} , (\u_h^n\cdot \e_{\u,h}^{n+1})-R_h(\u_h^n\cdot \e_{\u,h}^{n+1}) ) |  \nn \\
	& \leq &       2\tau |(\sigma_h^{n+1}D_{\tau}e^{n+1}_{\sigma,h} , R_h(\u_h^n\cdot \e_{\u,h}^{n+1}) ) | \nn \\
	& & +\frac{\mu\tau}{30}\| \nabla\e_{\u,h}^{n+1}\| ^2_{L^2} + C \tau h^2 \|D_\tau e^{n+1}_{\sigma,h} \|_{L^2}^2 .
\end{eqnarray}

Taking $r_h= 2\tau \sigma^{n+1}_h R_h (\u^n_h \cdot e^{n+1}_{\u,h})$ in (\ref{errorequation1}), we have 

\begin{align}\label{eq11}
	2\tau | (D_{\tau}e^{n+1}_{\sigma,h} ,   \sigma^{n+1}_h R_h (\u^n_h \cdot e^{n+1}_{\u,h} ) | = &2\tau |(\nabla\sigma^{n+1}\cdot (\u^n- \P_{0h}\u_h^n),\sigma^{n+1}_h R_h (\u^n_h \cdot e^{n+1}_{\u,h} ) | \nonumber  \\
	+ &2\tau|( \nabla e^{n+1}_{\sigma,h}\cdot (\P_{0h}\u_h^n-\u^n),\sigma^{n+1}_h R_h (\u^n_h \cdot e^{n+1}_{\u,h} ) |\nonumber\\
	+&2\tau| (\nabla e_{\sigma,h}^{n+1}\cdot \u^n,\sigma^{n+1}_h R_h (\u^n_h \cdot e^{n+1}_{\u,h}) | \nonumber\\
	+& 2\tau|(\nabla e^{n+1}_\sigma \cdot (\textbf{P}_{0h}-\u^n),\sigma^{n+1}_h R_h (\u^n_h \cdot e^{n+1}_{\u,h})| \nonumber\\
+&2\tau| (\nabla e^{n+1}_\sigma \cdot \u^n,\sigma^{n+1}_h R_h (\u^n_h \cdot e^{n+1}_{\u,h})| \nonumber\\
	+&2\tau|(R^{n+1}_{\sigma} ,\sigma^{n+1}_h R_h (\u^n_h \cdot e^{n+1}_{\u,h})| \nonumber\\
	=& 2\tau \sum\limits^{6}\limits_{i=1}|(Z_i ,\sigma^{n+1}_h R_h (\u^n_h \cdot e^{n+1}_{\u,h})|,
\end{align}
\begin{itemize}
	\item Estimate of $2\tau |(Z_1,\sigma^{n+1}_h R_h (\u^n_h \cdot e^{n+1}_{\u,h} )|$
\end{itemize}
\begin{align*}
	\begin{split}
	2\tau |(Z_1,\sigma^{n+1}_h R_h (\u^n_h \cdot e^{n+1}_{\u,h} )| =&2\tau |(\nabla\sigma^{n+1}\cdot (\u^n- \P_{0h}\u_h^n),\sigma^{n+1}_h R_h (\u^n_h \cdot e^{n+1}_{\u,h} ) |\\
		\leq&
	C\tau \| \nabla \sigma^{n+1} \|_{L^{\infty}}  \| \sigma^{n+1}_h\|_{L^{\infty}} \|  \u^n- \P_{0h}\u_h^n\|_{L^2}  \|R_h ( \u^n_h \cdot e^{n+1}_{\u,h})\|_{L^2}\\
		\leq &
		C \tau \|  \u^n- \P_{0h}\u_h^n\|_{L^2}  \|R_h (\u^n_h \cdot e^{n+1}_{\u,h})\|_{L^2}\\
			\leq &
			C \tau  \|  \u^n- \P_{0h}\u_h^n\|_{L^2}^2+C \tau \| \sigma^{n+1}_he^{n+1}_{\u,h}\|^2_{L^2},
	\end{split}
\end{align*}
where we use (\ref{eq5}), (\ref{eq8}).
\begin{itemize}
	\item Estimate of $2\tau |(Z_2,\sigma^{n+1}_h R_h (\u^n_h \cdot e^{n+1}_{\u,h} )|$
\end{itemize}

By (\ref{inverse1})  and (\ref{eq20}), we notice that 
\begin{align}\label{eq19}
	\| \textbf{P}_{0h} \u^n_h - \u^n \|^2_{L^3} \leq& \|  \textbf{P}_{0h} \u^n_h -\R_h\u^n \|^2_{L^3}+ \| \R_h\u^n - \u^n\|^2_{L^3}\notag \\
	\leq & \| \textbf{P}_{0h} \u^n_h-\u^n\|^2_{L^3} + \| \u^n -\R_h\u^n \|^2_{L^3}+\| \R_h\u^n - \u^n\|^2_{L^3}\notag\\
	\leq &   C h^{-1}  (\| \textbf{P}_{0h} \u^n_h-\u^n\|^2_{L^2} + \| \u^n -\R_h\u^n \|^2_{L^2}) + Ch^3 \notag\\
	\leq &C(1+C_0)^2h^3,
\end{align}
thus for sufficiently small $h$ such that $(1+C_0) h^{\frac{1}{2}} \leq 1$, we have 
\begin{align*}
	\begin{split}
		&2\tau |(Z_2,\sigma^{n+1}_h R_h (\u^n_h \cdot e^{n+1}_{\u,h} )|\\ =&2\tau|( \nabla e^{n+1}_{\sigma,h}\cdot (\P_{0h}\u_h^n-\u^n),\sigma^{n+1}_h R_h (\u^n_h \cdot e^{n+1}_{\u,h} ) |\\
		\leq&
		C\tau \| \nabla  e^{n+1}_{\sigma,h} \|_{L^{\infty}}  \| \sigma^{n+1}_h\|_{L^{\infty}} \|  \u^n- \P_{0h}\u_h^n\|_{L^2}  \|R_h ( \u^n_h \cdot e^{n+1}_{\u,h})  -\u^n_h \cdot e^{n+1}_{\u,h}\|_{L^2}\\
		&+	C\tau \| \nabla  e^{n+1}_{\sigma,h} \|_{L^{2}}  \| \sigma^{n+1}_h\|_{L^{\infty}} \|  \u^n- \P_{0h}\u_h^n\|_{L^3}  \|\u^n_h\|_{L^{\infty}}\|  e^{n+1}_{\u,h}\|_{L^6}\\
		\leq & C\tau h^{-\frac{5}{2}} \|e^{n+1}_{\sigma_h} \|_{L^2}C_p h^3 \| \nabla e^{n+1}_{\u,h}\|_{L^2}+C h^{-1} \| e^{n+1}_{\sigma_h}\|_{L^2} C (1+C_0)h^{\frac{3}{2}} \| \nabla e^{n+1}_{\u,h}\|_{L^2}  \\
		\leq &
		C \tau (1+C_0) h^{\frac{1}{2}} \|e^{n+1}_{\sigma,h} \|_{L^2}  \| \nabla e^{n+1}_{\u,h}\|_{L^2}  \\
		\leq &
		C \tau  \|e^{n+1}_{\sigma,h} \|_{L^2}^2+ \frac{\mu \tau}{30} \|\nabla e^{n+1}_{\u,h} \|_{L^2}^2,
	\end{split}
\end{align*}

\begin{itemize}
	\item Estimate of $2\tau |(Z_3,\sigma^{n+1}_h R_h (\u^n_h \cdot e^{n+1}_{\u,h} )|$
\end{itemize}
\begin{align}\label{eq16}
	\begin{split}
		2\tau |(Z_3,\sigma^{n+1}_h R_h (\u^n_h \cdot e^{n+1}_{\u,h} )| =&2\tau| (\nabla e_{\sigma,h}^{n+1}\cdot \u^n,\sigma^{n+1}_h R_h (\u^n_h \cdot e^{n+1}_{\u,h}) | \\
		\leq&
		2 \tau |  (\nabla e_{\sigma,h}^{n+1}\cdot \u^n,\sigma^{n+1}_h  (R_h (\u^n_h \cdot e^{n+1}_{\u,h} ) - \u^n_h \cdot e^{n+1}_{\u,h})  |\\
		&+2 \tau |  (\nabla e_{\sigma,h}^{n+1}\cdot \u^n,\sigma^{n+1}_h  ( \u^n_h \cdot e^{n+1}_{\u,h})  |\\
			\leq & C  \tau\| \u^n\|_{L^{\infty}} \| \sigma^{n+1}_h \|_{L^{\infty}} \| \nabla e^{n+1}_{\sigma,h} \|_{L^2} \| R_h (\u^n_h \cdot e^{n+1}_{\u,h} ) - \u^n_h \cdot e^{n+1}_{\u,h}\|_{L^2}\\
			&+ C  \tau\| \nabla \u^n\|_{L^3}  \|\u^n_h\|_{L^{\infty}} \| e^{n+1}_{\sigma,h}\|_{L^2} \| \sigma^{n+1}_h\|_{L^{\infty}} \|e^{n+1}_{\textbf{u},h} \|_{L^2} \\
			&+ C  \tau\|e^{n+1}_{\sigma,h} \|_{L^2} \| \u^n\|_{L^{\infty}}   \| \u^n_h\|_{L^{\infty}} \| \nabla \sigma^{n+1}_h \|_{L^3} \| e^{n+1}_{\u,h} \|_{L^6}\\
			&+ C  \tau\|e^{n+1}_{\sigma,h} \|_{L^2} \| \u^n\|_{L^{\infty}} \|  \sigma^{n+1}_h \|_{L^\infty}\| \nabla (\u^n_h \cdot e^{n+1}_{\u,h}) \|_{L^6}\\
			\leq &C \tau \|e^{n+1}_{\sigma,h} \|_{L^2}\| \nabla e^{n+1}_{\u,h} \|_{L^2}\\
			\leq &  \frac{\mu \tau}{30} \|\nabla e^{n+1}_{\u,h} \|_{L^2}^2 + C\tau \| e^{n+1}_{\sigma,h} \|^2_{L^2},
		\end{split}
\end{align}
where we use the integration by parts and (\ref{nsre-3}), (\ref{nsre-4}).

\begin{itemize}
	\item Estimate of $2\tau |(Z_4,\sigma^{n+1}_h R_h (\u^n_h \cdot e^{n+1}_{\u,h} )|$
\end{itemize}
For sufficiently small $h$ such that $Ch^{\frac{1}{2}} \leq 1$, we can get 
\begin{align*}
	\begin{split}
		2\tau |(Z_4,\sigma^{n+1}_h R_h (\u^n_h \cdot e^{n+1}_{\u,h} )| =&2\tau|(\nabla e^{n+1}_\sigma \cdot (\textbf{P}_{0h}\u^n_h-\u^n),\sigma^{n+1}_h R_h (\u^n_h \cdot e^{n+1}_{\u,h})|\\
		\leq& C \tau \| \nabla e^{n+1}_{\sigma}\|_{L^\infty}\|\textbf{P}_{0h}\u^n_h-\u^n\|_{L^2}  \| R_h (\u^n_h \cdot e^{n+1}_{\u,h})\|_{L^2} \|\sigma^{n+1}_h\|_{\infty} \\
	\leq & C  \tau h^{\frac{1}{2}} \|  \u^n- \P_{0h}\u_h^n\|_{L^2} \| R_h (\u^n_h \cdot e^{n+1}_{\u,h})\|_{L^2}\\
		\leq &	C \tau  \|  \u^n- \P_{0h}\u_h^n\|_{L^2}^2+C\tau \| \sigma^{n+1} _h e^{n+1}_{\u,h}\|^2_{L^2},
	\end{split}
\end{align*}
where we use (\ref{eq5}), (\ref{eq8}), (\ref{inverse2}) and (\ref{App2}).

\begin{itemize}
	\item Estimate of $2\tau |(Z_5,\sigma^{n+1}_h R_h (\u^n_h \cdot e^{n+1}_{\u,h} )|$
\end{itemize}
\begin{align*}
	\begin{split}
		2\tau |(Z_5,\sigma^{n+1}_h R_h (\u^n_h \cdot e^{n+1}_{\u,h} )| =&2\tau| (\nabla e^{n+1}_\sigma \cdot \u^n,\sigma^{n+1}_h R_h (\u^n_h \cdot e^{n+1}_{\u,h})|\\
		\leq& C\tau   \| \nabla e^{n+1}_{\sigma} \|_{L^2} \|  \u^n \|_{L^{\infty}} \| \sigma^{n+1}_h\|_{L^{\infty}} \|  R_h (\u^n_h \cdot e^{n+1}_{\u,h})\|_{L^2}\\
		\leq & C \tau h^4 + C \tau \| \sigma^{n+1} _h e^{n+1}_{\u,h}\|^2_{L^2},
	\end{split}
\end{align*}
where we use (\ref{App2}), (\ref{eq5}), (\ref{eq8}).
\begin{itemize}
	\item Estimate of $2\tau |(Z_6,\sigma^{n+1}_h R_h (\u^n_h \cdot e^{n+1}_{\u,h} )|$
\end{itemize}
\begin{align*}
	\begin{split}
		2\tau |(Z_6,\sigma^{n+1}_h R_h (\u^n_h \cdot e^{n+1}_{\u,h} )| =&2\tau|(R^{n+1}_{\sigma} ,\sigma^{n+1}_h R_h (\u^n_h \cdot e^{n+1}_{\u,h})| \\
		\leq & C \tau \| R^{n+1}_{\sigma}\|_{L^2} \|\sigma^{n+1}_h\|_{L^{\infty}}  \|  R_h (\u^n_h \cdot e^{n+1}_{\u,h})\|_{L^2}\\
		\leq  &C \tau  \| R^{n+1}_{\sigma}\|_{L^2}^2 + C \tau  \| \sigma^{n+1} _h e^{n+1}_{\u,h}\|^2_{L^2},
	\end{split}
\end{align*}
Substituting these estimates into (\ref{eq11}), we have
\begin{align}
		2\tau | (D_{\tau}e^{n+1}_{\sigma,h} ,   \sigma^{n+1}_h R_h (\theta^n_h  e^{n+1}_{\theta,h} ) | 
		 \leq & C \tau h^4+C \tau \| \sigma^{n+1}_he^{n+1}_{\u,h}\|^2_{L^2}+  \frac{2\mu \tau}{30} \|\nabla e^{n+1}_{\u,h} \|_{L^2}^2 + C\tau \| e^{n+1}_{\sigma,h} \|^2_{L^2} \nonumber\\
		&+ C \tau  \| R^{n+1}_{\sigma}\|_{L^2}^2+C \tau h^2 \|D_\tau e^{n+1}_{\sigma,h} \|_{L^2}^2+C \tau  \|  \u^n- \P_{0h}\u_h^n\|_{L^2}^2.
\end{align}
thus
\begin{align}
	\begin{split}
		2\tau(X_6,e^{n+1}_{\textbf{u},h}) =& 2 \tau(\sigma^{n+1}_h D_{\tau}e^{n+1}_{\sigma,h} \theta ^n ,e^{n+1}_{\textbf{u},h})\\
		\leq &   C\tau \|\sigma^n_h e^n_{\u,h}\|^2_{L^2} +C \tau h^4+C \tau \| \sigma^{n+1}_he^{n+1}_{\u,h}\|^2_{L^2}+  \frac{3\mu \tau}{30} \|\nabla e^{n+1}_{\u,h} \|_{L^2}^2 +  C\tau \| e^{n+1}_{\sigma,h} \|^2_{L^2}\\
			&+ C \tau  \| R^{n+1}_{\sigma}\|_{L^2}^2+C \tau h^2 \|D_\tau e^{n+1}_{\sigma,h} \|_{L^2}^2+C \tau  \|  \u^n- \P_{0h}\u_h^n\|_{L^2}^2.
	\end{split}
\end{align}
\begin{itemize}
	\item Estimate of $2\tau (X_7,e^{n+1}_{\u,h} )$
\end{itemize}
\begin{align*}
	\begin{split}
		2\tau(X_7,e^{n+1}_{\textbf{u},h}) =&2\tau (e^{n+1}_{\sigma,h} D_{\tau} (\sigma^{n+1} \textbf{u} ^{n+1}) ,e^{n+1}_{\u,h})\\
		\leq & 2 \tau | (e^{n+1}_{\sigma,h} \sigma^{n+1}D_{\tau} \u^{n+1} ,e^{n+1}_{\textbf{u},h} )| + 2 \tau | (e^{n+1}_{\sigma,h} \u^{n+1} D_{\tau} \sigma^{n+1} ,e^{n+1}_{\textbf{u},h} )|\\
		\leq &C \tau \| (e^{n+1}_{\sigma,h}\|_{L^2} \|  \sigma^{n+1}\|_{L^{\infty}}  \|D_{\tau} \u^{n+1}\|_{L^3} \| e^{n+1}_{\textbf{u},h}\|_{L^6}\\
		&+ C \tau \| (e^{n+1}_{\sigma,h}\|_{L^2} \|  \u^{n+1}\|_{L^{\infty}}  \|D_{\tau} \sigma^{n+1}\|_{L^3} \| e^{n+1}_{\textbf{u},h}\|_{L^6}\\
		\leq & C \tau  \|e^{n+1}_{\sigma,h} \|^2_{L^2} + \frac{\mu \tau}{30} \|\nabla e^{n+1}_{\u,h} \|_{L^2}^2.
 	\end{split}
\end{align*}
where we use Lagrange’s mean value theorem.
\begin{itemize}
	\item Estimate of $2\tau (X_8,e^{n+1}_{\u,h} )$
\end{itemize}
\begin{align*}
	\begin{split}
		2\tau(X_8,e^{n+1}_{\textbf{u},h}) =&2\tau(e^{n+1}_\rho (\u^n \cdot  \nabla ) \u^{n+1} ,e^{n+1}_{\textbf{u},h}  )\\
		\leq & C  \| e^{n+1}_\rho\|_{L^2} \|   \| \u^n\|_{L^{\infty}} \| \u^{n+1}\|_{L^3} \| e^{n+1}_{\textbf{u},h}  \|_{L^6}\\
		\leq &C \tau h^4 +  \frac{\mu \tau}{30} \|\nabla e^{n+1}_{\u,h} \|_{L^2}^2,
			\end{split}
\end{align*}

\begin{itemize}
	\item Estimate of $2\tau (X_9,e^{n+1}_{\u,h} )$
\end{itemize}
\begin{align*}
	\begin{split}
		2\tau(X_9,e^{n+1}_{\textbf{u},h}) =&2\tau(e^{n+1}_{\rho,h} (\u^n \cdot  \nabla ) \u^{n+1} ,e^{n+1}_{\textbf{u},h}  )\\
		\leq & C \tau  \| e^{n+1}_{\rho,h}\|_{L^2} \|   \| \u^n\|_{L^{\infty}} \| \u^{n+1}\|_{L^3} \| e^{n+1}_{\textbf{u},h}  \|_{L^6}\\
		\leq &C \tau \|e^{n+1}_{\sigma,h}\|_{L^2}^2+  \frac{\mu \tau}{30} \|\nabla e^{n+1}_{\u,h} \|_{L^2}^2+C\tau h^4,
	\end{split}
\end{align*}
where we use (\ref{eq8}) and 
\begin{align}\label{eq12}
	\|e^{n+1}_{\rho,h} \|_{L^2}=&\|  \Pi_h \rho^{n+1} - \rho^{n+1}_h \|_{L^2} \nonumber \\
	=& \| \Pi_h \rho^{n+1} -\rho^{n+1} + \rho^{n+1} - \rho^{n+1}_h \|_{L^2} \nonumber\\
	\leq &C h^2 + \| \sigma^{n+1} + \sigma^{n+1}_h \|_{L^{\infty}} \| \sigma^{n+1} - \sigma^{n+1}_h\|_{L^2} \nonumber\\
	\leq & C h^2 + C \|e^{n+1}_{\sigma,h}\|_{L^2}.
\end{align}

\begin{itemize}
	\item Estimate of $2\tau (X_{10},e^{n+1}_{\u,h} )$
\end{itemize}
\begin{align*}
	\begin{split}
		2\tau(X_{10},e^{n+1}_{\textbf{u},h}) =&2\tau( \rho^{n+1}_h ( e^n_{\u} \cdot \nabla) \u^{n+1} ,e^{n+1}_{\u,h} )\\
		\leq & C \tau \| \rho^{n+1}_h\|_{L^{\infty}} \| e^n_{\u}\|_{L^2}  \| \nabla \u^{n+1} \|_{L^3} \| e^{n+1}_{\u,h}\|_{L^6}\\
		\leq & C \tau h^4 + \frac{\mu \tau}{30} \|\nabla e^{n+1}_{\u,h} \|_{L^2}^2,
	\end{split}
\end{align*}
where we use (\ref{u_projection_error}). 
\begin{itemize}
	\item Estimate of $2\tau (X_{11},e^{n+1}_{\u,h} )$
\end{itemize}
\begin{align*}
	\begin{split}
		2\tau(X_{11},e^{n+1}_{\textbf{u},h}) =&2\tau( \rho^{n+1}_h ( e^n_{\u,h} \cdot \nabla) \u^{n+1} ,e^{n+1}_{\u,h} )\\
		\leq & C \tau \| \rho^{n+1}_h\|_{L^{\infty}} \| e^n_{\u,h}\|_{L^2}  \| \nabla \u^{n+1} \|_{L^3} \| e^{n+1}_{\u,h}\|_{L^6}\\
		\leq & C \tau \| \sigma^n_h e^n_{\u,h}\|_{L^2}^2 + \frac{\mu \tau}{30} \|\nabla e^{n+1}_{\u,h} \|_{L^2}^2,
	\end{split}
\end{align*}
where we use (\ref{eq8}).
\begin{itemize}
	\item Estimate of $2\tau (X_{12},e^{n+1}_{\u,h} )$
\end{itemize}
\begin{align*}
	\begin{split}
		2\tau(X_{12},e^{n+1}_{\textbf{u},h}) =&2\tau( \rho^{n+1}_h ( \u^n_h \cdot \nabla)e^{n+1}_{\u} ,e^{n+1}_{\u,h} )\\
		\leq & C \tau \| \nabla \rho^{n+1}_h\|_{L^{3}}  \|  \u^{n}_h \|_{L^\infty}  \|  e^{n+1}_{\u}\|_{L^2} \| e^{n+1}_{\u,h}\|_{L^6}\\
		&+  C \tau \|  \rho^{n+1}_h\|_{L^{\infty}}  \| \nabla \u^{n}_h \|_{L^3}  \|  e^{n+1}_{\u}\|_{L^2} \| e^{n+1}_{\u,h}\|_{L^6}\\
		\leq & C \tau h^4 + \frac{\mu \tau}{30} \|\nabla e^{n+1}_{\u,h} \|_{L^2}^2,
	\end{split}
\end{align*}
where we use integration by parts and (\ref{eq9}), (\ref{nsre-4}).

\begin{itemize}
	\item Estimate of $2\tau (X_{13},e^{n+1}_{\u,h} )$
\end{itemize}
\begin{align*}
	\begin{split}
		2\tau(X_{13},e^{n+1}_{\textbf{u},h}) =&2\tau( \rho^{n+1}_h ( \u^n_h \cdot \nabla)e^{n+1}_{\u,h} ,e^{n+1}_{\u,h} )\\
		\leq & C \tau \| \rho^{n+1}_h\|_{L^{\infty}}  \|  \u^{n}_h \|_{L^\infty}  \| \nabla e^{n+1}_{\u,h}\|_{L^2} \| e^{n+1}_{\u,h}\|_{L^2}\\
		\leq & C \tau \| \sigma^{n+1}_h e^{n+1}_{\u,h}\|_{L^2}^2 + \frac{\mu \tau}{30} \|\nabla e^{n+1}_{\u,h} \|_{L^2}^2,
	\end{split}
\end{align*}

\begin{itemize}
	\item Estimate of $2\tau (X_{14},e^{n+1}_{\u,h} )$
\end{itemize}
\begin{align*}
	\begin{split}
		2\tau(X_{14},e^{n+1}_{\textbf{u},h}) =&\tau( \textbf{u}^{n+1} \nabla \cdot (e^{n+1}_{\rho} \textbf{u}^n),e^{n+1}_{\u,h})\\
		\leq & C \tau \| \nabla \u^{n+1}\|_{L^3} \|e^{n+1}_{\rho} \|_{L^2} \| \u^n\|_{L^{\infty}} \| e^{n+1}_{\u,h} \|_{L^6} \\
		& + C \tau \|  \u^{n+1}\|_{L^\infty} \|e^{n+1}_{\rho} \|_{L^2} \| \u^n\|_{L^{\infty}} \|\nabla e^{n+1}_{\u,h} \|_{L^2} \\
		\leq  &C \tau h^4 + \frac{\mu \tau}{30} \|\nabla e^{n+1}_{\u,h} \|_{L^2}^2,
	\end{split}
\end{align*}
where we use integration by parts.

\begin{itemize}
	\item Estimate of $2\tau (X_{15},e^{n+1}_{\u,h} )$
\end{itemize}

By using the above same method and (\ref{eq12}), we can get
\begin{align*}
	\begin{split}
		2\tau(X_{15},e^{n+1}_{\textbf{u},h}) =&\tau( \textbf{u}^{n+1} \nabla \cdot (e^{n+1}_{\rho,h} \textbf{u}^n),e^{n+1}_{\u,h} )\\
		\leq & C \tau \| \nabla \u^{n+1}\|_{L^3} \|e^{n+1}_{\rho,h} \|_{L^2} \| \u^n\|_{L^{\infty}} \| e^{n+1}_{\u,h} \|_{L^6} \\
		& + C \tau \|  \u^{n+1}\|_{L^\infty} \|e^{n+1}_{\rho,h} \|_{L^2} \| \u^n\|_{L^{\infty}} \| \nabla e^{n+1}_{\u,h} \|_{L^2} \\
		\leq  &C \tau   \| e^{n+1}_{\sigma,h}\|^2_{L^2} + \frac{\mu \tau}{30} \|\nabla e^{n+1}_{\u,h} \|_{L^2}^2,
	\end{split}
\end{align*}

\begin{itemize}
	\item Estimate of $2\tau (X_{16},e^{n+1}_{\u,h} )$
\end{itemize}
\begin{align*}
	\begin{split}
		2\tau(X_{16},e^{n+1}_{\textbf{u},h}) =&\tau( \textbf{u}^{n+1} \nabla \cdot (\rho^{n+1}_h e^n_{\u}),e^{n+1}_{\u,h} )\\
		\leq & C \tau \| \nabla \u^{n+1}\|_{L^3} \|  \rho^{n+1}_h \|_{L^\infty} \| e^n_{\u}\|_{L^{2}} \| e^{n+1}_{\u,h} \|_{L^6} \\
		&+ C \tau \|  \u^{n+1}\|_{L^\infty} \|  \rho^{n+1}_h \|_{L^\infty} \| e^n_{\u}\|_{L^{2}} \| \nabla e^{n+1}_{\u,h} \|_{L^2} \\
		\leq & C \tau h^4 + \frac{\mu \tau}{30} \|\nabla e^{n+1}_{\u,h} \|_{L^2}^2,
	\end{split}
\end{align*}
where we use integration by parts and (\ref{u_projection_error}).
\begin{itemize}
	\item Estimate of $2\tau (X_{17},e^{n+1}_{\u,h} )$
\end{itemize}
\begin{align*}
	\begin{split}
		2\tau(X_{17},e^{n+1}_{\textbf{u},h}) =&\tau( \textbf{u}^{n+1} \nabla \cdot (\rho^{n+1}_h e^n_{\u,h}),e^{n+1}_{\u,h} )\\
		\leq & C \tau \| \nabla \u^{n+1}\|_{L^3} \|  \rho^{n+1}_h \|_{L^\infty} \| e^n_{\u,h}\|_{L^{2}} \| e^{n+1}_{\u,h} \|_{L^6} \\
		&+ C \tau \|  \u^{n+1}\|_{L^\infty} \|  \rho^{n+1}_h \|_{L^\infty} \| e^n_{\u,h}\|_{L^{2}} \| \nabla e^{n+1}_{\u,h} \|_{L^2} \\
		\leq & C \tau \| \sigma^{n}_h e^{n}_{\u,h}\|_{L^2}^2+ \frac{\mu \tau}{30} \|\nabla e^{n+1}_{\u,h} \|_{L^2}^2,
	\end{split}
\end{align*}
where we use integration by parts and (\ref{eq8}).
\begin{itemize}
	\item Estimate of $2\tau (X_{18},e^{n+1}_{\u,h} )$
\end{itemize}
\begin{align*}
	\begin{split}
		2\tau(X_{18},e^{n+1}_{\textbf{u},h}) =&\tau( e^{n+1}_{\u,h} \nabla \cdot (\rho^{n+1}_h \u_h^n),e^{n+1}_{\u,h} )\\
		\leq & C \tau \| e^{n+1}_{\u,h}\|_{L^2} \|   \nabla \rho^{n+1}_h \|_{L^3} \| \u^n_h\|_{L^{\infty}} \| e^{n+1}_{\u,h} \|_{L^6} \\
		&+ C \tau \| e^{n+1}_{\u,h}\|_{L^2} \|    \rho^{n+1}_h \|_{L^\infty} \| \nabla \u_h^n\|_{L^{3}} \| e^{n+1}_{\u,h} \|_{L^6} \\
		\leq & C \tau \| \sigma^{n+1}_h e^{n+1}_{\u,h}\|_{L^2}^2+ \frac{\mu \tau}{30} \|\nabla e^{n+1}_{\u,h} \|_{L^2}^2,
	\end{split}
\end{align*}
where we use (\ref{eq8}), (\ref{nsre-4}), (\ref{eq9}). 
\begin{itemize}
	\item Estimate of $2\tau (X_{19},e^{n+1}_{\u,h} )$
\end{itemize}
\begin{align*}
	\begin{split}
		2\tau(X_{19},e^{n+1}_{\textbf{u},h}) =&\tau( e^{n+1}_{\u} \nabla \cdot (\rho^{n+1}_h \u_h^n),e^{n+1}_{\u,h} )\\
		\leq & C \tau \| e^{n+1}_{\u}\|_{L^2} \|   \nabla \rho^{n+1}_h \|_{L^3} \| \u^n_h\|_{L^{\infty}} \| e^{n+1}_{\u,h} \|_{L^6} \\
		&+ C \tau \| e^{n+1}_{\u}\|_{L^2} \|    \rho^{n+1}_h \|_{L^\infty} \| \nabla \u_h^n\|_{L^{3}} \| e^{n+1}_{\u,h} \|_{L^6} \\
		\leq & C \tau h^4+ \frac{\mu \tau}{30} \|\nabla e^{n+1}_{\u,h} \|_{L^2}^2,
	\end{split}
\end{align*}
where we use (\ref{u_projection_error}),  (\ref{nsre-4}), (\ref{eq9}). 
%
%
\begin{itemize}
	\item Estimate of $2\tau (X_{20},e^{n+1}_{\u,h} )$
\end{itemize}
\begin{align*}
	\begin{split}
		2\tau(X_{20},e^{n+1}_{\textbf{u},h}) =&2\tau (R^{n+1}_{\textbf{u}},e^{n+1}_{\textbf{u},h} ) \\
		\leq &C \tau \| R^{n+1}_{\textbf{u}}\|^2_{L^2} + C  \tau \| \sigma^{n+1}_h e^{n+1}_{\u,h}\|_{L^2}^2,
	\end{split}
\end{align*}

Substituting these estimates into (\ref{errorequation2})  we have
\begin{align*}
	\begin{split}
&	\| \sigma^{n+1}_h e^{n+1}_{\u,h} \|^2_{L^2} - \| \sigma^n_h e^n_{\u,h} \|^2_{L^2} + 2 \mu \tau \|\nabla e^{n+1}_{\u,h}\|^2_{L^2}\\
	\leq & C \tau h^4 + C\tau ( \|\sigma^{n+1}_h e^{n+1}_{\u,h} \| ^2_{L^2}+\|\sigma^{n}_h e^{n}_{\u,h} \| ^2_{L^2} ) +C \tau (\| R^{n+1}_{\textbf{u}}\|^2_{L^2} +\| R^{n+1}_{\sigma}\|^2_{L^2} )+ C \tau h^2 \|D_{\tau} e^{n+1}_{\sigma,h} \|^2_{L^2} \\
	& + C \tau   h^4 ( \| D_{\tau} \u^{n+1}\|_{H^2}^2 + \|D_{\tau} \sigma^{n+1}\|_{H^2}^2  ) +C \tau  \|  \u^n- \P_{0h}\u_h^n\|_{L^2}^2+ C \tau  \|e^{n+1}_{\sigma,h}\|^2_{L^2},
	\end{split}
\end{align*}
taking the sum and using (\ref{eq18}), (\ref{eq7}), (\ref{nsre-1}), (\ref{eq1}), we can derive
\begin{align}
&	\|\sigma^{n+1}_h \e^{n+1}_{\u,h} \|^2_{L^2} + \tau \sum_{i=1}^{n} \| \nabla \e^{n+1}_{\u,h} \|^2_{L^2}\\
	\leq & C (\tau^2+h^4) + C\tau \|\sigma_h^{n+1} \e^{n+1}_{\u,h} \|^2_{L^2} +C \tau \sum_{i=1}^{n} \|\u^n-\textbf{P}_{0h} \u^n_h\|^2_{L^2}\nonumber\\
	\leq &C(\tau^2+h^4) +C\tau \|\sigma_h^{n+1} \e^{n+1}_{\u,h} \|^2_{L^2}.\notag
\end{align}
where we notice that 
\begin{align} \label{eq21}
	&\| \u^n -\textbf{P}_{0h} \u^n_h \|_{L^2}  \nonumber\\
	\leq & C \| \u^n -\textbf{P}_{0h}\u^n + \textbf{P}_{0h}\u^n - \textbf{P}_{0h}\u^n_h\|_{L^2} \nonumber\\
	\leq& C  h^4+C \| e^n_{\u,h}\|_{L^2} ,
\end{align}
thus by applying the discrete Gronwall inequality in Lemma \ref{GronwallLemma}, we can derive 
\begin{align}
		\|\sigma^{n+1}_h \e^{n+1}_{\u,h} \|^2_{L^2} + \tau \sum_{i=1}^{n} \| \nabla \e^{n+1}_{\u,h} \|^2_{L^2 }	\leq  C (\tau^2+h^4),
\end{align}
by (\ref{eq8}), there exists some $C_1$ > 0 being independent of $\tau, h$ and such that
\begin{align}\label{nsre-8}
	\|\e^{n+1}_{\u,h} \|^2_{L^2} + \tau \sum_{i=1}^{n} \| \nabla \e^{n+1}_{\u,h} \|^2_{L^2 }	\leq  C_1 (\tau^2+h^4), \quad  \forall 1\leq n \leq N-1.
\end{align}
Thus, we prove that (\ref{u_n_error}) is valid by taking $C_0 \leq C_1$.
\subsection{Estimate of $\| e^{n+1}_{\theta,h}\|^2_{L^2}$}

Next, we need to estimate $\| e^{n+1}_{\theta,h}\|^2_{L^2}$ by the mathematical induction.
Letting $w_h=2 \tau \theta^{n+1}_h$ in error equation (\ref{errorequation3}), 
applying H\"{o}lder inequality and Young inequality, we can estimate $\sum\limits^{20}\limits_{i=1}(Y_i ,(e^{n+1}_{\sigma,h})$ as follows

\begin{itemize}
\item Estimate of $2\tau (Y_1,e^{n+1}_{\theta,h} )$
\end{itemize}

By recombining this term and  using (\ref{eq2}), (\ref{App9}), (\ref{eq8}), Lagrange’s mean value theorem,  we have
\begin{align*}
\begin{split}
2\tau (Y_1 , e^{n+1}_{\theta,h}) =&2 \tau (\sigma ^{n+1} _h D_{\tau} (\sigma ^{n+1} _h e^{n+1}_{\theta}), e^{n+1}_{\theta,h})\\
\leq&
2 \tau | (\rho ^{n+1}_h D_{\tau} e^{n+1}_{\theta} , e^{n+1} _{\theta,h})|+ 2 \tau |(\sigma ^{n+1} _h e^n_{\theta} D_{\tau} e^{n+1}_{\sigma,h} ,e^{n+1} _{\theta,h} )|+2 \tau| (\sigma^{n+1}_{h} e^n_{\theta} D_{\tau} (\Pi_h \sigma^{n+1}) , e^{n+1}_{\theta,h})|\\
\leq & C \tau \| \rho^{n+1}_h\|_{L^{\infty}} \| D_{\tau} e^{n+1}_{\theta} \|_{L^2} \| e^{n+1}_{\theta,h}\|_{L^2} + C\tau \|\sigma^{n+1}_h\|_{L^{\infty}} \| e^n_\theta\|_{L^2} \|D_{\tau} e^{n+1}_{\sigma,h}\|_{L^3}\|e^{n+1}_{\theta,h}\|_{L^6} \\
&+ C\tau\| \sigma^{n+1}_h\|_{L^{\infty}} \| e^n_\theta\|_{L^2} \| D_{\tau} (\Pi_h \sigma^{n+1}) \|_{L^3} \|e^{n+1}_{\sigma,h}\|_{L^6} \\
\leq&
 C\tau h^4 \| D_{\tau} \theta^{n+1} \| ^2_{H^2} +C\tau h^4 +\frac{\mu\tau}{30}  \| \nabla e^{n+1}_{\theta,h} \| ^2_{L^2}.
\end{split}
\end{align*}
\begin{itemize}
\item Estimate of $2\tau (Y_2,e^{n+1}_{\theta,h} )$
\end{itemize}
By recombining this term and using (\ref{App2}), Lagrange’s mean value theorem,  one has 
\begin{align*}
\begin{split}
2\tau(Y_2,e^{n+1}_{\theta,h}) =&2\tau (e^{n+1}D_{\tau}(\sigma^{n+1}\theta^{n+1}) , e^{n+1}_{\theta})\\
\leq&  2\tau |( e^{n+1}_{\sigma} \sigma^{n+1} D_{\tau}\theta^{n+1}   , e^{n+1}_{\theta,h})| + 2\tau | ( e^{n+1}_\sigma (D_{\tau}\sigma^{n+1} ) \theta^{n+1} ,e^{n+1}_{\theta,h} ) |\\
\leq & C \tau h^4 +\frac{\mu \tau}{30} \| \nabla e^{n+1}_{\theta,h}\|^2_{L^2}.
\end{split}
\end{align*}
Similar to  $2\tau (Y_2,e^{n+1}_{\theta,h} )$, we can have
\begin{itemize}
\item Estimate of $2\tau (Y_3,e^{n+1}_{\theta,h} )$
\end{itemize}
 
\begin{align*}
\begin{split}
2\tau(Y_3,e^{n+1}_{\theta,h}) =& 2 \tau(\sigma^{n+1}_h e^{n+1}_{\sigma}D_{\tau}\theta^{n+1},e^{n+1}_{\theta,h} )\\
\leq & C\tau h^4 + +\frac{\mu \tau}{30} \| \nabla e^{n+1}_{\theta,h}\|^2_{L^2}.
\end{split}
\end{align*}

\begin{itemize}
\item Estimate of $2\tau (Y_4,e^{n+1}_{\theta,h} )$
\end{itemize}
\begin{align*}
\begin{split}
2\tau(Y_4,e^{n+1}_{\theta,h}) =& 2 \tau(\sigma^{n+1}_h e^{n+1}_{\sigma,h}D_{\tau}\theta^{n+1},e^{n+1}_{\theta,h} )\\
\leq & C\tau h^4 + C\tau \|e^{n+1}_{\sigma,h}\|^2_{L^2},
\end{split}
\end{align*}

\begin{itemize}
\item Estimate of $2\tau (Y_5,e^{n+1}_{\theta,h} )$
\end{itemize}
\begin{align*}
\begin{split}
2\tau(Y_5,e^{n+1}_{\theta,h}) =&2\tau (\sigma^{n+1}_h D_{\tau}e^{n+1}_{\sigma} \theta ^n ,e^{n+1}_{\theta,h})\\
\leq & C \tau h^4 \| D_{\tau} \sigma^{n+1} \|^2_{H^2}+ C \tau \| \sigma^{n+1}_h e^{n+1}_{\theta,h} \|^2_{L^2},
\end{split}
\end{align*}
where we use (\ref{App9}), (\ref{eq8}).
\begin{itemize}
\item Estimate of $2\tau (Y_6,e^{n+1}_{\theta,h} )$
\end{itemize}
\begin{align*}
\begin{split}
2\tau(Y_6,e^{n+1}_{\theta,h}) =& 2 \tau(\sigma^{n+1}_h D_{\tau}e^{n+1}_{\sigma,h} \theta ^n ,e^{n+1}_{\theta,h})\\
\leq & 2\tau|( \sigma^{n+1}_h D_{\tau}e^{n+1}_{\sigma,h} e^n_{\theta}, e^{n+1}_{\theta,h}  )| +2\tau|( \sigma^{n+1}_h D_{\tau}e^{n+1}_{\sigma,h} e^n_{\theta,h}, e^{n+1}_{\theta,h}  )| \\
& +2\tau|( \sigma^{n+1}_h D_{\tau}e^{n+1}_{\sigma,h} \theta^n_h, e^{n+1}_{\theta,h}  )| \\
\leq & C\tau h^4 + \frac{\mu \tau}{30} \|\nabla e^{n+1}_{\theta,h}\|^2_{L^2} + C\tau \|\sigma^n_h e^n_{\sigma,h}\|^2_{L^2} +2\tau|( \sigma^{n+1}_h D_{\tau}e^{n+1}_{\sigma,h} \theta^n_h, e^{n+1}_{\theta,h}  )|,
\end{split}
\end{align*}
where we use (\ref{eq2}), (\ref{eq8}), (\ref{nsre-5}), (\ref{eq15}).

In order to estimate  $2\tau|( \sigma^{n+1}_h D_{\tau}e^{n+1}_{\sigma,h} \theta^n_h, e^{n+1}_{\theta,h}  )|$, we adapt the same projection operator in (\ref{Rh}).

By (\ref{eq15}), (\ref{nsre-2}) , one has
\begin{align*}
	\|\nabla(\theta_h^n e_{\theta,h}^{n+1})\|_{L^2}
	\leq \|\nabla\theta_h^n\|_{L^3}\|e^{n+1}_{\theta,h}\|_{L^6}+\|\theta_h^n\|_{L^\infty}\|\nabla e^{n+1}_{\theta,h}\|_{L^2}
	\leq  C \|\nabla e_{\theta,h}^{n+1}\|_{L^2}.
\end{align*}

It follows from (\ref{Rh}) that 
\begin{align}
	\label{Rh2}
	\|(\theta_h^n\cdot e_{\theta,h}^{n+1})-R_h(\theta_h^n\cdot e_{\theta,h}^{n+1})\|_{L^2}
	\leq Ch \|\nabla e_{\theta,h}^{n+1}\|_{L^2}.
\end{align}
thus
\begin{eqnarray}
	\label{X71}
	2\tau |(\sigma_h^{n+1}D_{\tau} e^{n+1}_{\sigma,h} \theta_h^n,e_{\theta,h}^{n+1}) | & \leq &
	2\tau |(\sigma_h^{n+1}D_{\tau}e^{n+1}_{\sigma,h} , R_h(\theta_h^n e_{\theta,h}^{n+1}) ) | \nn \\
	& &+  2\tau |(\sigma_h^{n+1}D_{\tau}e^{n+1}_{\sigma,h} , (\theta_h^n  e_{\theta,h}^{n+1})-R_h(\theta_h^n  e_{\theta,h}^{n+1}) ) |  \nn \\
	& \leq &       2\tau |(\sigma_h^{n+1}D_{\tau}e^{n+1}_{\sigma,h} , R_h(\theta_h^n e_{\theta,h}^{n+1}) ) | \nn \\
	& & +\frac{\mu\tau}{21}\| \nabla e_{\theta,h}^{n+1}\| ^2_{L^2} + C \tau h^2 \|D_\tau e^{n+1}_{\sigma,h} \|_{L^2}^2 .
\end{eqnarray}

Taking $r_h= 2\tau \sigma^{n+1}_h R_h (\theta^n_h  e^{n+1}_{\theta,h})$ in (\ref{errorequation1}), we have 

\begin{align}
	2\tau | (D_{\tau}e^{n+1}_{\sigma,h} ,   \sigma^{n+1}_h R_h (\theta^n_h  e^{n+1}_{\theta,h} ) | = &2\tau |(\nabla\sigma^{n+1}\cdot (\u^n- \P_{0h}\u_h^n),\sigma^{n+1}_h R_h (\theta^n_h \cdot e^{n+1}_{\theta,h} ) | \nonumber  \\
	+ &2\tau|( \nabla e^{n+1}_{\sigma,h}\cdot (\P_{0h}\u_h^n-\u^n),\sigma^{n+1}_h R_h (\theta^n_h  e^{n+1}_{\theta,h} ) |\nonumber\\
	+&2\tau| (\nabla e_{\sigma,h}^{n+1}\cdot \u^n,\sigma^{n+1}_h R_h (\theta^n_h  e^{n+1}_{\theta,h}) |\\
	+& 2\tau|(\nabla e^{n+1}_\sigma \cdot (\textbf{P}_{0h}\u^n_h-\u^n),\sigma^{n+1}_h R_h (\theta^n_h  e^{n+1}_{\theta,h})| \nonumber\\
	+&2\tau| (\nabla e^{n+1}_\sigma \cdot \u^n,\sigma^{n+1}_h R_h (\theta^n_h  e^{n+1}_{\theta,h})| \nonumber\\
	+&2\tau|(R^{n+1}_{\sigma} ,\sigma^{n+1}_h R_h (\theta^n_h  e^{n+1}_{\theta,h})| \nonumber\\
	=& 2\tau \sum\limits^{6}\limits_{i=1}|(S_i ,\sigma^{n+1}_h R_h (\theta^n_h  e^{n+1}_{\theta,h})|,
\end{align}
\begin{itemize}
	\item Estimate of $2\tau |(S_1,\sigma^{n+1}_h R_h (\theta^n_h  e^{n+1}_{\theta,h} )|$
\end{itemize}
\begin{align*}
	\begin{split}
		2\tau |(S_1,\sigma^{n+1}_h R_h (\u^n_h \cdot e^{n+1}_{\u,h} )| =&2\tau |(\nabla\sigma^{n+1}\cdot (\u^n- \P_{0h}\u_h^n),\sigma^{n+1}_h R_h (\theta^n_h  e^{n+1}_{\theta,h} ) |\\
		\leq&
		C\tau \| \nabla \sigma^{n+1} \|_{L^{\infty}}  \| \sigma^{n+1}_h\|_{L^{\infty}} \|  \u^n- \P_{0h}\u_h^n\|_{L^2}  \|R_h ( \theta^n_h  e^{n+1}_{\theta,h})\|_{L^2}\\
		\leq &
		C \tau \|  \u^n- \P_{0h}\u_h^n\|_{L^2}  \|R_h (\theta^n_h  e^{n+1}_{\theta,h})\|_{L^2}\\
		\leq &
		C \tau  \|  \u^n- \P_{0h}\u_h^n\|_{L^2}^2+C \tau \| \sigma^{n+1}_h e^{n+1}_{\theta,h}\|^2_{L^2},
	\end{split}
\end{align*}
where we use (\ref{eq15}), (\ref{eq8}).
\begin{itemize}
	\item Estimate of $2\tau |(S_2,\sigma^{n+1}_h R_h (\theta^n_h e^{n+1}_{\theta,h} )|$
\end{itemize}

For sufficiently small $h$ such that $(1+C_0) h^{\frac{1}{2}} \leq 1$, we have 
\begin{align*}
	\begin{split}
		2\tau |(S_2,\sigma^{n+1}_h R_h (\theta^n_h \cdot e^{n+1}_{\theta,h} )| =&2\tau|( \nabla e^{n+1}_{\sigma,h}\cdot (\P_{0h}\u_h^n-\u^n),\sigma^{n+1}_h R_h (\theta^n_h  e^{n+1}_{\theta,h} ) |\\
		\leq&
		C\tau \| \nabla  e^{n+1}_{\sigma,h} \|_{L^{\infty}}  \| \sigma^{n+1}_h\|_{L^{\infty}} \|  \u^n- \P_{0h}\u_h^n\|_{L^2}  \|R_h ( \theta^n_h  e^{n+1}_{\theta,h})  -\theta^n_h  e^{n+1}_{\theta,h}\|_{L^2}\\
		&+	C\tau \| \nabla  e^{n+1}_{\sigma,h} \|_{L^{2}}  \| \sigma^{n+1}_h\|_{L^{\infty}} \|  \u^n- \P_{0h}\u_h^n\|_{L^3}  \|\theta^n_h\|_{L^{\infty}}\| e^{n+1}_{\theta,h}\|_{L^6}\\
		\leq &
		C \tau (1+C_0) \tau h^{\frac{1}{2}} \|e^{n+1}_{\sigma,h} \|_{L^2}  \| \nabla e^{n+1}_{\theta,h}\|_{L^2}  \\
		\leq &
		C \tau  \|e^{n+1}_{\sigma,h} \|_{L^2}^2+ \frac{\mu \tau}{30} \|\nabla e^{n+1}_{\theta,h} \|_{L^2}^2,
	\end{split}
\end{align*}
where by inverse inequality (\ref{inverse2}), (\ref{eq19}).
\begin{itemize}
	\item Estimate of $2\tau |(S_3,\sigma^{n+1}_h R_h (\theta^n_h  e^{n+1}_{\theta,h} )|$
\end{itemize}
\begin{align*}
	\begin{split}
		2\tau |(S_3,\sigma^{n+1}_h R_h (\theta^n_h \cdot e^{n+1}_{\theta,h} )| =&2\tau| (\nabla e_{\sigma,h}^{n+1}\cdot \u^n,\sigma^{n+1}_h R_h (\theta^n_h  e^{n+1}_{\theta,h}) |\\
		\leq&
		2 \tau |  (\nabla e_{\sigma,h}^{n+1}\cdot \u^n,\sigma^{n+1}_h  (R_h (\theta^n_h  e^{n+1}_{\theta,h} ) - \theta^n_h \cdot e^{n+1}_{\theta,h})  |\\
		&+2 \tau |  (\nabla e_{\sigma,h}^{n+1}\cdot \u^n,\sigma^{n+1}_h  ( \theta^n_h  e^{n+1}_{\theta,h})  |\\
		\leq & C  \tau\| \u^n\|_{L^{\infty}} \| \sigma^{n+1}_h \|_{l^{\infty}} \| \nabla e^{n+1}_{\sigma,h} \|_{L^2} \| R_h (\theta^n_h e^{n+1}_{\theta,h} ) - \theta^n_h e^{n+1}_{\theta,h}\|_{L^2}\\
		&+ C  \tau\| \nabla \u^n\|_{L^3}  \|\theta^n_h\|_{L^{\infty}} \| e^{n+1}_{\sigma,h}\|_{L^2} \| \sigma^{n+1}_h\|_{L^{\infty}} \|e^{n+1}_{\theta,h} \|_{L^2} \\
		&+ C  \tau\|e^{n+1}_{\sigma,h} \|_{L^2} \| \u^n\|_{L^{\infty}}   \| \theta^n_h\|_{L^{\infty}}\| \nabla \sigma^{n+1}_h \|_{L^3} \| e^{n+1}_{\theta,h} \|_{L^6}\\
		&+ C  \tau\|e^{n+1}_{\sigma,h} \|_{L^2} \| \u^n\|_{L^{\infty}} \|  \sigma^{n+1}_h \|_{L^\infty}\| \nabla (\theta^n_h  e^{n+1}_{\theta,h}) \|_{L^6}\\
		\leq &C \tau \|e^{n+1}_{\sigma,h} \|_{L^2}\| \nabla e^{n+1}_{\theta,h} \|_{L^2}\\
		\leq &  \frac{\mu \tau}{30} \|\nabla e^{n+1}_{\theta,h} \|_{L^2}^2 + C\tau \| e^{n+1}_{\sigma,h} \|^2_{L^2},
	\end{split}
\end{align*}

where we use the integration by parts and (\ref{nsre-6}), (\ref{nsre-4}).

\begin{itemize}
	\item Estimate of $2\tau |(S_4,\sigma^{n+1}_h R_h (\theta^n_h  e^{n+1}_{\theta,h} )|$
\end{itemize}

For sufficiently small $h$ such that $Ch^{\frac{1}{2}} \leq 1$, we can get 
\begin{align*}
	\begin{split}
		2\tau |(S_4,\sigma^{n+1}_h R_h (\theta^n_h e^{n+1}_{\theta,h} )| =&2\tau|(\nabla e^{n+1}_\sigma \cdot (\textbf{P}_{0h} \u^n_h-\u^n),\sigma^{n+1}_h R_h (\theta^n_h  e^{n+1}_{\theta,h})|\\
		\leq& C \tau \| \nabla e^{n+1}_{\sigma}\|_{L^\infty}\|\textbf{P}_{0h} \u^n_h-\u^n\|_{L^2}  \| R_h (\theta^n_h e^{n+1}_{\theta,h})\|_{L^2} \|\sigma^{n+1}_h\|_{\infty} \\
		\leq & C  \tau h^{\frac{1}{2}} \|\textbf{P}_{0h} \u^n_h-\u^n\|_{L^2}\| R_h (\theta^n_h e^{n+1}_{\theta,h})\|_{L^2}\\
		\leq &	C\tau\|\textbf{P}_{0h} \u^n_h-\u^n\|_{L^2}^2+C\tau \|  \sigma^{n+1}_h e^{n+1}_{\theta,h}\|^2_{L^2},
	\end{split}
\end{align*}
where we use (\ref{eq15}), (\ref{eq8}), (\ref{inverse2}) and (\ref{App2}),

\begin{itemize}
	\item Estimate of $2\tau |(S_5,\sigma^{n+1}_h R_h (\theta^n_h  e^{n+1}_{\theta,h} )|$
\end{itemize}
\begin{align*}
	\begin{split}
		2\tau |(S_5,\sigma^{n+1}_h R_h (\theta^n_h \cdot e^{n+1}_{\theta,h} )| =&2\tau| (\nabla e^{n+1}_\sigma  \u^n,\sigma^{n+1}_h R_h (\theta^n_h  e^{n+1}_{\theta,h})|\\
		\leq& C\tau   \| \nabla e^{n+1}_{\sigma} \|_{L^2} \|  \u^n \|_{L^{\infty}} \| \sigma^{n+1}_h\|_{L^{\infty}} \|  R_h (\theta^n_h \cdot e^{n+1}_{\theta,h})\|_{L^2}\\
		\leq & C \tau h^4 + C \tau \|  \sigma^{n+1}_h e^{n+1}_{\theta,h}\|^2_{L^2},
	\end{split}
\end{align*}
where we use (\ref{eq15}),(\ref{eq8}).
\begin{itemize}
	\item Estimate of $2\tau |(S_6,\sigma^{n+1}_h R_h (\theta^n_h \cdot e^{n+1}_{\theta,h} )|$
\end{itemize}
\begin{align*}
	\begin{split}
		2\tau |(S_6,\sigma^{n+1}_h R_h (\theta^n_h  e^{n+1}_{\theta,h} )| =&2\tau|(R^{n+1}_{\sigma} ,\sigma^{n+1}_h R_h (\theta^n_h  e^{n+1}_{\theta,h})| \\
		\leq & C \tau \| R^{n+1}_{\sigma}\|_{L^2} \|\sigma^{n+1}_h\|_{L^{\infty}}  \|  R_h (\theta^n_h \cdot e^{n+1}_{\theta,h})\|_{L^2}\\
		\leq  &C \tau  \| R^{n+1}_{\sigma}\|_{L^2}^2 + C \tau  \|  \sigma^{n+1}_h e^{n+1}_{\theta,h}\|^2_{L^2},
	\end{split}
\end{align*}
Substituting these estimates into (\ref{eq11}), we have
\begin{align}
	2\tau | (D_{\tau}e^{n+1}_{\sigma,h} ,   \sigma^{n+1}_h R_h (\theta^n_h  e^{n+1}_{\theta,h} ) | 
	\leq & C \tau h^4+C \tau \| e^{n+1}_{\theta,h}\|^2_{L^2}+  \frac{\mu \tau}{30} \|\nabla e^{n+1}_{\theta,h} \|_{L^2}^2 + C\tau \| e^{n+1}_{\sigma,h} \|^2_{L^2} \nonumber\\
	&+ C \tau  \| R^{n+1}_{\sigma}\|_{L^2}^2+C \tau h^2 \|D_\tau e^{n+1}_{\sigma,h} \|_{L^2}^2 +C\tau\|\textbf{P}_{0h} \u^n_h-\u^n\|_{L^2}^2.
\end{align}
thus
\begin{align}
	\begin{split}
		2\tau(Y_6,e^{n+1}_{\textbf{u},h}) =& 2 \tau(\sigma^{n+1}_h D_{\tau}e^{n+1}_{\sigma,h} \theta ^n ,e^{n+1}_{\theta,h})\\
		\leq &   C\tau \|\sigma^n_h e^n_{\sigma,h}\|^2_{L^2} +C \tau h^4+C \tau \|  \sigma^{n+1}_h e^{n+1}_{\theta,h}\|^2_{L^2}+  \frac{\mu \tau}{30} \|\nabla e^{n+1}_{\theta,h} \|_{L^2}^2 + C\tau \| e^{n+1}_{\sigma,h} \|^2_{L^2}\\
		&+ C \tau  \| R^{n+1}_{\sigma}\|_{L^2}^2+C \tau h^2 \|D_\tau e^{n+1}_{\sigma,h} \|_{L^2}^2 +C\tau\|\textbf{P}_{0h} \u^n_h-\u^n\|_{L^2}^2.
	\end{split}
\end{align}

\begin{itemize}
	\item  Estimate of $2 \tau(  Y_7 ,e^{n+1}_{\theta,h} )  $
\end{itemize}
\begin{align*}
	\begin{split}
		2 \tau (Y_7 ,e^{n+1}_{\theta,h})  \leq &2 \tau | e^{n+1}_{\sigma,h}  \sigma^{n+1} D_{\tau} \theta^{n+1} ,e^{n+1}_{\theta,h} )  | +  2 \tau | e^{n+1}_{\sigma,h}  \theta^{n+1} D_{\tau} \sigma^{n+1} ,e^{n+1}_{\theta,h} )  | \\
		\leq & C \tau \| \sigma^{n+1}\|_{L^{\infty}}   \|e^{n+1}_{\sigma,h} \|_{L^2} \| D_{\tau} \theta^{n+1} \|_{L^3}  \|  e^{n+1}_{\theta,h}\|_{L^6}\\
		&+ C \tau \| \theta^{n+1}\|_{L^{\infty}}   \|e^{n+1}_{\sigma,h} \|_{L^2} \|D_{\tau} \sigma^{n+1} \|_{L^3}  \|  e^{n+1}_{\theta,h}\|_{L^6}\\
		\leq & C \tau \| e^{n+1} _{\sigma,h} \| ^2_{L^2} + \frac{\mu \tau}{30} \|\nabla e^{n+1}_{\theta,h}\|^2_{L^2},
 	\end{split}
\end{align*}
where we use Lagrange’s mean value theorem.

\begin{itemize}
	\item  Estimate of $2 \tau(  Y_8 ,e^{n+1}_{\theta,h} )  $
\end{itemize}
\begin{align*}
	\begin{split}
		2 \tau (Y_8,e^{n+1}_{\theta,h})  &\leq  2 \tau \| e^{n+1} _{\rho}\|_{L^2} \| \textbf{u} ^n \|_{L^{\infty}} \| \nabla \theta^{n+1} \|_{L^3} \| e^{n+1}_{\theta,h}\|_{L^6}\\
		\leq &  \frac{\mu \tau}{30} \|\nabla e^{n+1}_{\theta,h}\|^2_{L^2} + C  \tau h^4,
 	\end{split}
\end{align*}

\begin{itemize}
	\item Estimate of $2\tau (Y_9,e^{n+1}_{\theta,h} )$
\end{itemize}
\begin{align*}
	\begin{split}
		2\tau(Y_9,e^{n+1}_{\theta,h}) =&2\tau(e^{n+1}_{\rho,h} (\u^n \cdot  \nabla ) \theta^{n+1} ,e^{n+1}_{\theta,h}  )\\
		\leq & C \tau  \| e^{n+1}_{\rho,h}\|_{L^2} \|   \| \u^n\|_{L^{\infty}} \| \theta^{n+1}\|_{L^3} \| e^{n+1}_{\theta,h}  \|_{L^6}\\
		\leq &C \tau \|e^{n+1}_{\sigma,h}\|_{L^2}^2+  \frac{\mu \tau}{30} \|\nabla e^{n+1}_{\theta,h} \|_{L^2}^2+C\tau h^4,
	\end{split}
\end{align*}
where we use (\ref{eq8}) and (\ref{eq12}),
\begin{itemize}
	\item Estimate of $2\tau (Y_{10},e^{n+1}_{\theta,h} )$
\end{itemize}
\begin{align*}
	\begin{split}
		2\tau(Y_{10},e^{n+1}_{\textbf{u},h}) =&2\tau( \rho^{n+1}_h ( e^n_{\u} \cdot \nabla) \theta^{n+1} ,e^{n+1}_{\theta,h} )\\
		\leq & C \tau \| \rho^{n+1}_h\|_{L^{\infty}} \| e^n_{\u}\|_{L^2}  \| \nabla \theta^{n+1} \|_{L^3} \| e^{n+1}_{\theta,h}\|_{L^6}\\
		\leq & C \tau h^4 + \frac{\mu \tau}{30} \|\nabla e^{n+1}_{\theta,h} \|_{L^2}^2,
	\end{split}
\end{align*}
where we use (\ref{u_projection_error}).
\begin{itemize}
	\item Estimate of $2\tau (Y_{11},e^{n+1}_{\theta,h} )$
\end{itemize}
\begin{align*}
	\begin{split}
		2\tau(Y_{11},e^{n+1}_{\theta,h}) =&2\tau( \rho^{n+1}_h ( e^n_{\u,h} \cdot \nabla) \theta^{n+1} ,e^{n+1}_{\theta,h} )\\
		\leq & C \tau \| \rho^{n+1}_h\|_{L^{\infty}} \| e^n_{\u,h}\|_{L^2}  \| \nabla \theta^{n+1} \|_{L^3} \| e^{n+1}_{\theta,h}\|_{L^6}\\
		\leq & C \tau \| \sigma^n_h e^n_{\u,h}\|_{L^2}^2 + \frac{\mu \tau}{30} \|\nabla e^{n+1}_{\theta,h} \|_{L^2}^2,
	\end{split}
\end{align*}
where we use (\ref{eq8}).
\begin{itemize}
	\item Estimate of $2\tau (Y_{12},e^{n+1}_{\theta,h} )$
\end{itemize}
\begin{align*}
	\begin{split}
		2\tau(Y_{12},e^{n+1}_{\textbf{u},h}) =&2\tau( \rho^{n+1}_h ( \u^n_h \cdot \nabla)e^{n+1}_{\theta} ,e^{n+1}_{\theta,h} )\\
		\leq & C \tau \| \nabla \rho^{n+1}_h\|_{L^{3}}  \|  \u^{n}_h \|_{L^\infty}  \|  e^{n+1}_{\theta}\|_{L^2} \| e^{n+1}_{\theta,h}\|_{L^6}\\
		&+  C \tau \|  \rho^{n+1}_h\|_{L^{\infty}}  \| \nabla \u^{n}_h \|_{L^3}  \|  e^{n+1}_{\theta}\|_{L^2} \| e^{n+1}_{\theta,h}\|_{L^6}\\
		\leq & C \tau h^4 + \frac{\mu \tau}{30} \|\nabla e^{n+1}_{\theta,h} \|_{L^2}^2,
	\end{split}
\end{align*}
where we use integration by parts and (\ref{eq9}), (\ref{nsre-4}).

\begin{itemize}
	\item Estimate of $2\tau (Y_{13},e^{n+1}_{\theta,h} )$
\end{itemize}
\begin{align*}
	\begin{split}
		2\tau(Y_{13},e^{n+1}_{\theta,h}) =&2\tau( \rho^{n+1}_h ( \u^n_h \cdot \nabla)e^{n+1}_{\theta,h} ,e^{n+1}_{\theta,h} )\\
		\leq & C \tau \| \rho^{n+1}_h\|_{L^{\infty}}  \|  \u^{n}_h \|_{L^\infty}  \| \nabla e^{n+1}_{\theta,h}\|_{L^2} \| e^{n+1}_{\theta,h}\|_{L^2}\\
		\leq & C \tau \|  \sigma^{n+1}_h  e^{n+1}_{\theta,h}\|_{L^2}^2 + \frac{\mu \tau}{30} \|\nabla e^{n+1}_{\theta,h} \|_{L^2}^2,
	\end{split}
\end{align*}

\begin{itemize}
	\item Estimate of $2\tau (Y_{14},e^{n+1}_{\theta,h} )$
\end{itemize}
\begin{align*}
	\begin{split}
		2\tau(Y_{14},e^{n+1}_{\textbf{u},h}) =&\frac{1}{2} ( \theta^{n+1} \nabla \cdot (e^{n+1}_{\rho} \textbf{u}^n),e^{n+1}_{\theta,h})\\
		\leq & C \tau \| \nabla \theta^{n+1}\|_{L^3} \|e^{n+1}_{\rho} \|_{L^2} \| \u^n\|_{L^{\infty}} \| e^{n+1}_{\theta,h} \|_{L^6} \\
		& + C \tau \|  \theta^{n+1}\|_{L^\infty} \|e^{n+1}_{\rho} \|_{L^2} \| \u^n\|_{L^{\infty}} \| e^{n+1}_{\theta,h} \|_{L^2} \\
		\leq  &C \tau h^4 + \frac{\mu \tau}{30} \|\nabla e^{n+1}_{\theta,h} \|_{L^2}^2,
	\end{split}
\end{align*}
where we use integration by parts,

\begin{itemize}
	\item Estimate of $2\tau (Y_{15},e^{n+1}_{\theta,h} )$
\end{itemize}

By using the above same method and (\ref{eq12}), we can get
\begin{align*}
	\begin{split}
		2\tau(Y_{15},e^{n+1}_{\theta,h}) =&\frac{1}{2} ( \theta^{n+1} \nabla \cdot (e^{n+1}_{\rho,h} \textbf{u}^n),e^{n+1}_{\theta,h} )\\
		\leq & C \tau \| \nabla \theta^{n+1}\|_{L^3} \|e^{n+1}_{\rho,h} \|_{L^2} \| \u^n\|_{L^{\infty}} \| e^{n+1}_{\theta,h} \|_{L^6} \\
		& + C \tau \|  \theta^{n+1}\|_{L^\infty} \|e^{n+1}_{\rho,h} \|_{L^2} \| \u^n\|_{L^{\infty}} \| \nabla e^{n+1}_{\theta,h} \|_{L^2} \\
		\leq  &C \tau \| e^{n+1}_{\sigma,h} \|^2_{L^2} + \frac{\mu \tau}{30} \|\nabla e^{n+1}_{\theta,h} \|_{L^2}^2,
	\end{split}
\end{align*}

\begin{itemize}
	\item Estimate of $2\tau (Y_{16},e^{n+1}_{\theta,h} )$
\end{itemize}
\begin{align*}
	\begin{split}
		2\tau(Y_{16},e^{n+1}_{\theta,h}) =&\frac{1}{2} ( \theta^{n+1} \nabla \cdot (\rho^{n+1}_h e^n_{\u}),e^{n+1}_{\u,h} )\\
		\leq & C \tau \| \nabla \theta^{n+1}\|_{L^3} \|  \rho^{n+1}_h \|_{L^\infty} \| e^n_{\u}\|_{L^{2}} \| e^{n+1}_{\u,h} \|_{L^6} \\
		&+ C \tau \|  \theta^{n+1}\|_{L^\infty} \|  \rho^{n+1}_h \|_{L^\infty} \| e^n_{\u}\|_{L^{2}} \| \nabla e^{n+1}_{\u,h} \|_{L^2} \\
		\leq & C \tau h^4 + \frac{\mu \tau}{30} \|\nabla e^{n+1}_{\u,h} \|_{L^2}^2,
	\end{split}
\end{align*}
where we use integration by parts and (\ref{u_projection_error}).
\begin{itemize}
	\item Estimate of $2\tau (Y_{17},e^{n+1}_{\theta,h} )$
\end{itemize}
\begin{align*}
	\begin{split}
		2\tau(Y_{17},e^{n+1}_{\theta,h}) =&\frac{1}{2} ( \theta^{n+1} \nabla \cdot (\rho^{n+1}_h e^n_{\u,h}),e^{n+1}_{\theta,h} )\\
		\leq & C \tau \| \nabla \theta^{n+1}\|_{L^3} \|  \rho^{n+1}_h \|_{L^\infty} \| e^n_{\u,h}\|_{L^{2}} \| e^{n+1}_{\theta,h} \|_{L^6} \\
		&+ C \tau \|  \theta^{n+1}\|_{L^\infty} \|  \rho^{n+1}_h \|_{L^\infty} \| e^n_{\u,h}\|_{L^{2}} \| \nabla e^{n+1}_{\theta,h} \|_{L^2} \\
		\leq & C \tau \| \sigma^{n}_h e^{n}_{\u,h}\|_{L^2}^2+ \frac{\mu \tau}{30} \|\nabla e^{n+1}_{\theta,h} \|_{L^2}^2,
	\end{split}
\end{align*}
where we use integration by parts and (\ref{eq8}).
\begin{itemize}
	\item Estimate of $2\tau (Y_{18},e^{n+1}_{\theta,h} )$
\end{itemize}
\begin{align*}
	\begin{split}
		2\tau(Y_{18},e^{n+1}_{\theta,h}) =&\frac{1}{2} ( e^{n+1}_{\theta,h} \nabla \cdot (\rho^{n+1}_h \u_h^n),e^{n+1}_{\theta,h} )\\
		\leq & C \tau \| e^{n+1}_{\theta,h}\|_{L^2} \|   \nabla \rho^{n+1}_h \|_{L^3} \| \u^n_h\|_{L^{\infty}} \| e^{n+1}_{\theta,h} \|_{L^6} \\
		&+ C \tau \| e^{n+1}_{\theta,h}\|_{L^2} \|    \rho^{n+1}_h \|_{L^\infty} \| \nabla \u_h^n\|_{L^{3}} \| e^{n+1}_{\theta,h} \|_{L^6} \\
		\leq & C \tau \|\sigma^{n+1}_h e^{n+1}_{\theta,h}\|_{L^2}^2+ \frac{\mu \tau}{30} \|\nabla e^{n+1}_{\theta,h} \|_{L^2}^2,
	\end{split}
\end{align*}
where we use (\ref{eq8}), (\ref{nsre-4}), (\ref{eq9}). 

\begin{itemize}
	\item Estimate of $2\tau (Y_{19},e^{n+1}_{\theta,h} )$
\end{itemize}
\begin{align*}
	\begin{split}
		2\tau(Y_{19},e^{n+1}_{\theta,h}) =&\frac{1}{2} ( e^{n+1}_{\theta} \nabla \cdot (\rho^{n+1}_h \u_h^n),e^{n+1}_{\theta,h} )\\
		\leq & C \tau \| e^{n+1}_{\theta}\|_{L^2} \|   \nabla \rho^{n+1}_h \|_{L^3} \| \u^n_h\|_{L^{\infty}} \| e^{n+1}_{\theta,h} \|_{L^6} \\
		&+ C \tau \| e^{n+1}_{\theta}\|_{L^2} \|    \rho^{n+1}_h \|_{L^\infty} \| \nabla \u_h^n\|_{L^{3}} \| e^{n+1}_{\theta,h} \|_{L^6} \\
		\leq & C \tau h^4+ \frac{\mu \tau}{30} \|\nabla e^{n+1}_{\theta,h} \|_{L^2}^2,
	\end{split}
\end{align*}
where we use (\ref{App2}),  (\ref{nsre-4}), (\ref{eq9}). 

\begin{itemize}
	\item Estimate of $2\tau (Y_{20},e^{n+1}_{\theta,h} )$
\end{itemize}
\begin{align*}
	\begin{split}
		2\tau(Y_{20},e^{n+1}_{\theta,h}) =&2\tau (R^{n+1}_{\theta},e^{n+1}_{\theta,h} ) \\
		\leq &C \tau \| R^{n+1}_{\theta}\|^2_{L^2} + C  \tau \| \sigma^{n+1}_h e^{n+1}_{\theta,h}\|_{L^2}^2,
	\end{split}
\end{align*}

Substituting these estimates into (\ref{errorequation3}), we have
\begin{align*}
	\begin{split}
		&	\| \sigma^{n+1}_h e^{n+1}_{\theta,h} \|^2_{L^2} - \| \sigma^n_h e^n_{\theta,h} \|^2_{L^2} + 2 \mu \tau \|\nabla e^{n+1}_{\theta,h}\|^2_{L^2}\\
		\leq & C \tau h^4 + C\tau ( \|\sigma^{n+1}_h e^{n+1}_{\theta,h} \| ^2_{L^2}+\|\sigma^{n}_h e^{n}_{\theta,h} \| ^2_{L^2} ) +C (\tau \| R^{n+1}_{\theta}\|^2_{L^2} +\| R^{n+1}_{\theta}\|^2_{L^2})+ C \tau h^2 \|D_{\tau} e^{n+1}_{\sigma,h} \|^2_{L^2} \\
		&+ C \tau   h^4 ( \| D_{\tau} \u^{n+1}\|_{H^2}^2 + \|D_{\tau} \sigma^{n+1}\|_{H^2}^2  ) +C \tau  \|  \u^n- \P_{0h}\u_h^n\|_{L^2}^2+ C \tau  \|e^{n+1}_{\sigma,h}\|^2_{L^2} .
	\end{split}
\end{align*}
taking the sum and using (\ref{eq18}), (\ref{eq7}), (\ref{eq1}), (\ref{eq21}), (\ref{nsre-1}), we can derive
\begin{align}
		\|\sigma^{n+1}_h e^{n+1}_{\theta,h} \|^2_{L^2} + \tau \sum_{i=1}^{n} \| \nabla e^{n+1}_{\theta,h} \|^2_{L^2}
	\leq &  C (\tau^2+h^4) + C\tau \|\sigma_h^{n+1} e^{n+1}_{\theta,h} \|^2_{L^2}+C \tau  \|  \u^n- \P_{0h}\u_h^n\|_{L^2}^2 \notag\\
	\leq &C (\tau^2+h^4) + C\tau \|\sigma_h^{n+1} e^{n+1}_{\theta,h} \|^2_{L^2} +C \tau \| e^{n}_{\u,h}\|^2_{L^2}\\
	\leq&C (\tau^2+h^4) + C\tau \|\sigma_h^{n+1} e^{n+1}_{\theta,h} \|^2_{L^2} .\notag
\end{align}

Thus by applying the discrete Gronwall inequality in Lemma \ref{GronwallLemma}, we can derive 
\begin{align}
	\|\sigma^{n+1}_h e^{n+1}_{\theta,h} \|^2_{L^2} + \tau \sum_{i=1}^{n} \| \nabla e^{n+1}_{\theta,h} \|^2_{L^2 }	\leq  C (\tau^2+h^4),
\end{align}

By (\ref{eq8}), there exists some $C_2$ > 0 being independent of $\tau, h$ and such that
\begin{align}\label{nsre-7}
	\| e^{n+1}_{\theta,h} \|^2_{L^2} + \tau \sum_{i=1}^{n} \| \nabla e^{n+1}_{\theta,h} \|^2_{L^2 }	\leq  C (\tau^2+h^4),
\end{align}

By (\ref{nsre-8}) and (\ref{nsre-7}),
the proof of Lemma \ref{lemma3} is completed.
\subsection{Main result}
Finally we present the main result on error estimates of the density $\rho$, the velocity $\u$ and the temperature  $\theta$ in $L^2$-norm.

\begin{theorem}
	\label{theorem3}
	Under the assumptions ({\bf A1}), ({\bf A2}) and the time step condition $\tau \leq C h^2$, , if $h$ and $\tau$ are sufficiently small, then one has
	\begin{align}
		\label{mainerror2}
		\| \sigma^{n+1} - \sigma^{n+1}_h \|_{L^2}+\| \rho^{n+1} - \rho^{n+1}_h \|_{L^2}+	\|\u(t_{n+1})-\u_h^{n+1}\|_{L^2}+	\| \theta^{n+1} - \theta^{n+1}_h \|_{L^2} \leq C (\tau+h^2)
	\end{align}for any $0\leq n\leq N-1$.
\end{theorem}
\begin{proof}

By Lemma \ref{lemma1}, (\ref{App2}),(\ref{eq21}), (\ref{nsre-1})we have 

\begin{align}
 	\| \sigma^{n+1} - \sigma^{n+1}_h \|^2_{L^2} \leq  &C ( \| e^{n+1}_\sigma\|^2_{L^2} +\| e^{n+1}_{\sigma,h}\|^2_{L^2}   ) \nonumber\\
 	\leq &C (\tau^2+h^4) + C\tau\sum_{i=1}^n\| \u^n-\P_{0h}\u_h^n\|^2_{L^2}  \nonumber\\
 	\leq & C (\tau^2+h^4) + C \tau \sum_{i=1}^n \| e^{n}_{\u,h}\|^2_{L^2}   \nonumber   \\
 	\leq &C  (\tau^2+h^4).
\end{align}

Thus by  (\ref{eq8}) we have 
\begin{align}
	\| \rho^{n+1} - \rho^{n+1}_h \|_{L^2} \leq  \| \sigma^{n+1} +\sigma^{n+1}_h \|_{L^\infty} \| \sigma^{n+1} -\sigma^{n+1}_h\|_{L^2} \leq C (\tau+h^2).
\end{align}

In terms of (\ref{u_projection_error}), (\ref{eq8}) and  (\ref{u_n_error}) we have
\begin{align} 
	 	\| \u^{n+1} - \u^{n+1}_h \|_{L^2} \leq  &C ( \| \e^{n+1}_\u\|^2_{L^2} +\| \e^{n+1}_{\u,h}\|^2_{L^2}   )  \leq C (\tau+h^2).
\end{align}

In terms of (\ref{projection_error_theta}), (\ref{eq8}) and  (\ref{theta_n_error}) we have
\begin{align}
	\| \theta^{n+1} - \theta^{n+1}_h \|_{L^2} \leq  &C ( \| e^{n+1}_\theta\|^2_{L^2} +\| e^{n+1}_{\theta,h}\|^2_{L^2}   )  \leq C (\tau+h^2).
\end{align}

The proof of Theorem \ref{theorem3} is completed.
\end{proof}


\section{Numerical results}

For simplicity, we consider the  time dependent natural convection problem with variable density (\ref{newns1}) - (\ref{eq17}) in the convex domain and an artificial function $g_2$ is add in the right hand side of (\ref{newns1}), i.e. we solve the following coupled system:

\begin{eqnarray}
	\sigma_t+  \nabla\cdot ( \sigma \u) &=&g_2,\\
	\sigma(\sigma \u)_t-\mu\Delta \u+\rho(\u\cdot \nabla )\u+\frac{1}{2}\u\nabla\cdot (\rho\u)+\nabla p&=&\f,\\
	\nabla\cdot \u&=&0,\\
	\sigma (\sigma \theta)_t -\kappa \Delta \theta+\rho (\textbf{u}\cdot \nabla) \theta +\frac{1}{2} \theta \nabla \cdot (\rho \textbf{u}) &=&g,
\end{eqnarray}
in order to choose the approximate functions $f,g$ and $g_2$, we consider a known analytical solution  in $\Omega\times [0, T]$, where $\Omega=[0,1]^d$, $d=2,3$.

\begin{eqnarray*}
        \sigma(x,y,t) &=& 2+x(1-x)cos(sin(t))+y(1-y)sin(sin(t)),\\
        \u(x,y,t) &=& (t^3y^2(1-y), \ t^3x^2(1-x))^{ T},\\
        p(x,y,t) &=& tx+y-(t+1)/2.\\
        \theta(x,y,t)&=& t^3y^2(1-y)+t^3x^2(1-x),
       \end{eqnarray*}
       in two-dimensional case and 
       \begin{eqnarray*}
       	\sigma(x,y,t) &=&2+x(1-x)cos(sin(t))+y(1-y)sin(sin(t))+z(1-z)sin(sin(t)),\\
       	\u(x,y,t) &=& (t^3y^2(1-y), \ t^3z^2(1-z)  ,\  t^3x^2(1-x) )^{ T},\\
       	p(x,y,t) &=&(2x-1)(2y-1)(2z-1)exp(-t).\\
       	\theta(x,y,t)&=& t^3y^2(1-y)+t^3x^2(1-x) +t^3z^2(1-z),
       \end{eqnarray*}
       in three-dimensional case.
In addition, we take the viscosity coefficient $\mu=0.1, \kappa=0.1, \gamma_1=0.1, \gamma_2=0.1$ and the final time $T=1$. All programs are implemented by using the free finite element software FreeFem++ \cite{hechet2012}.

Numerical results are showed by taking different grid size $\frac{1}{4}, \frac{1}{8},\frac{1}{16},\dots, \frac{1}{128}$ the meshes are given from the uniform triangles meshes.
For the two-dimensional problem, we give the $L^2$-norm error and convergent rate for $\rho, \u, \theta, p$ in Table \ref{label1}, \ref{label2}, when we set $\tau=h$, we can clearly see the first-order convergent  rate in Table \ref{label1}, when $\tau=h^2$, the second-order convergent rate are showed in Table \ref{label2}, especially, we can find that the second-order convergent rate for $\rho$ when $\tau =h^3$ in Table \ref{label5}, because finite element solutions of the
scalar hyperbolic equation (\ref{ns1}) have lower-order convergence rates
\cite{cpietro2012}.
In addition, we present the errors and convergent rate for $\rho, \u, \theta, p$ for three-dimensional problem in Table \ref{label3}, \ref{label4}.  
The numerical solutions for  the velocity, pressure, density, temperature  at  $t = 0, 0.2, 0.4, 0.6, 0.8, 1.0$ are presented in Figure \ref{velocity}, \ref{pressure}, \ref{density}, \ref{temperature}.  
In conclusion, all
numerical results and tests have well verified the effectiveness and accuracy of the proposed algorithm.

\begin{table}[htbp] 
	 \centering  
	 \caption{$L^2$-errors and convergence rates for 2D problem }    
	 	 \setlength{\tabcolsep}{1.2mm}{
	 \begin{tabular}{ccccccccc}    
	 	        \hline\hline
	 	\textbf{$\tau=h$} &$\| \rho(T)-\rho^N_h\|_{L^2}$ & rate & $\|\u(T)-\u^N_h\|_{L^2}$ & rate &$\|\theta(T)-\theta^N_h\|_{L^2}$& rate & $\|p(T)-p^N_h\|_{L^2}$ & rate \\   
	 	     \hline
	 	  1/4  & 0.116211 &       & 0.0164303 &       & 0.034649 &       & 0.358521 &  \\   
	 	    1/8  & 0.0596104 & 0.96  & 0.00586277 & 1.49  & 0.0158444 & 1.13  & 0.158757 & 1.18  \\     
	 	     1/16 & 0.0301046 & 0.99  & 0.00261228 & 1.17  & 0.00752981 & 1.07  & 0.0743927 & 1.09  \\    
	 	       1/32 & 0.0151193 & 0.99  & 0.00126633 & 1.04  & 0.00366573 & 1.04  & 0.0360589 & 1.04  \\     
	 	        1/64 & 0.00757439 & 1.00  & 0.000628496 & 1.01  & 0.00180799 & 1.02  & 0.017761 & 1.02  \\    
	 	           1/128 & 0.00379062 & 1.00  & 0.00031377 & 1.00  & 0.000897767 & 1.01  & 0.00881533 & 1.01  \\  
	 	                   \hline\hline
	 	             \end{tabular}}
	 	              \label{label1}%
	 	              \end{table}%

\begin{table}[htbp] 
	 \centering  
	 \caption{$L^2$-errors and convergence rates for 2D problem}   
	  	 \setlength{\tabcolsep}{1.2mm}{
	  \begin{tabular}{ccccccccc}   
	  	        \hline\hline
	 	\textbf{$\tau=h^2$} &$\| \rho(T)-\rho^N_h\|_{L^2}$ & rate & $\|\u(T)-\u^N_h\|_{L^2}$ & rate &$\|\theta(T)-\theta^N_h\|_{L^2}$& rate & $\|p(T)-p^N_h\|_{L^2}$ & rate \\ 
	  	 	 	     \hline
	  	   1/2  & 0.115364 &       & 0.0534366 &       & 0.035 &       & 0.527472 &  \\    
	  	    1/4  & 0.0324759 & 1.83  & 0.0155145 & 1.78  & 0.00853127 & 2.04  & 0.143624 & 1.88  \\    
	  	     1/8  & 0.00845427 & 1.94  & 0.00400495 & 1.95  & 0.00192685 & 2.15  & 0.0351823 & 2.03  \\    
	  	       1/16 & 0.00214018 & 1.98  & 0.00100696 & 1.99  & 0.000468026 & 2.04  & 0.00869251 & 2.02  \\  
	  	           1/32 & 0.00053738 & 1.99  & 0.000251983 & 2.00  & 0.00011619 & 2.01  & 0.00216691 & 2.00  \\   
	  	              1/64 & 0.00013452 & 2.00  & 6.30E-05 & 2.00  & 2.90E-05 & 2.00  & 0.000542729 & 2.00  \\  
	  	                           \hline\hline
	  	                \end{tabular}}
	  	                 \label{label2}%
	  	                 \end{table}%

\begin{table}[htbp]  
	\centering  
	\caption{$L^2$-errors and convergence rates for 2D problem}  
	\setlength{\tabcolsep}{1.2mm}{
		\begin{tabular}{ccccccccc}        
			\hline\hline  
			\textbf{$\tau=h^3$} &$\| \rho(T)-\rho^N_h\|_{L^2}$ & rate & $\|\u(T)-\u^N_h\|_{L^2}$ & rate &$\|\theta(T)-\theta^N_h\|_{L^2}$& rate & $\|p(T)-p^N_h\|_{L^2}$ & rate \\ 
			\hline 
			1/4  & 0.0124115 &       & 0.015454 &       & 0.00463455 &       & 0.0944332 &  \\     1/8  & 0.00263021 & 2.24  & 0.00396712 & 1.96  & 0.000876496 & 2.40  & 0.0215494 & 2.13  \\      
			1/16 & 0.00060222 & 2.13  & 0.000995465 & 1.99  & 0.0002007 & 2.13  & 0.00511456 & 2.07  \\     
			1/32 & 0.00014445 & 2.06  & 0.000248985 & 2.00  & 4.93E-05 & 2.03  & 0.00125065 & 2.03  \\   
			\hline\hline
	\end{tabular}}
	\label{label5}%
\end{table}%

\begin{table}[htbp]  
	\centering  \caption{$L^2$-errors and convergence rates for 3D problem}    
		  	 \setlength{\tabcolsep}{1.2mm}{
	\begin{tabular}{rcccccccc}    
	 	        \hline\hline
	 	\textbf{$\tau=h$} &$\| \rho(T)-\rho^N_h\|_{L^2}$ & rate & $\|\u(T)-\u^N_h\|_{L^2}$ & rate &$\|\theta(T)-\theta^N_h\|_{L^2}$& rate & $\|p(T)-p^N_h\|_{L^2}$ & rate \\  
		 	     \hline
		 1/4       & 0.181671 &       & 0.018811 &       & 0.064867 &       & 1.66915 &  \\   
		  1/8       & 0.0930417 & 0.97  & 0.00725636 & 1.37  & 0.0307612 & 1.08  & 0.6876 & 1.28  \\    
		  1/12       & 0.0624487 & 0.98  & 0.00478507 & 1.03  & 0.0199439 & 1.07  & 0.438388 & 1.11  \\   
		   1/16       & 0.0469825 & 0.99  & 0.00362164 & 0.97  & 0.0147326 & 1.05  & 0.327372 & 1.02  \\   
		    1/20       & 0.0376517 & 0.99  & 0.00293124 & 0.95  & 0.0116747 & 1.04  & 0.264926 & 0.95  \\   
		                \hline\hline
		      \end{tabular}}
		       \label{label3}%
		       \end{table}%
\begin{table}[htbp]  
	\centering  
	\caption{$L^2$-errors and convergence rates for 3D problem}    
		  	 \setlength{\tabcolsep}{1.2mm}{
	\begin{tabular}{rcccccccc}    
		       \hline\hline
	 	\textbf{$\tau=h^2$} &$\| \rho(T)-\rho^N_h\|_{L^2}$ & rate & $\|\u(T)-\u^N_h\|_{L^2}$ & rate &$\|\theta(T)-\theta^N_h\|_{L^2}$& rate & $\|p(T)-p^N_h\|_{L^2}$ & rate \\ 
		     \hline
		 1/4     & 0.0482913 &       & 0.0181513 &       & 0.0161436 &       & 0.578801 &  \\   
		  1/8     & 0.0128767 & 1.91  & 0.00458475 & 1.99  & 0.00374441 & 2.11  & 0.164685 & 1.81  \\   
		   1/12       & 0.00598432 & 1.89  & 0.00207245 & 1.96  & 0.00162733 & 2.06  & 0.0994962 & 1.24  \\    
		   1/16       & 0.00352684 & 1.84  & 0.00121412 & 1.86  & 0.000907367 & 2.03  & 0.0794992 & 0.78  \\   
		        \hline\hline
		    \end{tabular}}
		    \label{label4}%
		    \end{table}%

\begin{figure}[htbp] 
	
	\centering 

	\begin{minipage}{0.32\textwidth} 
		\centering  
		\includegraphics[width=\textwidth]{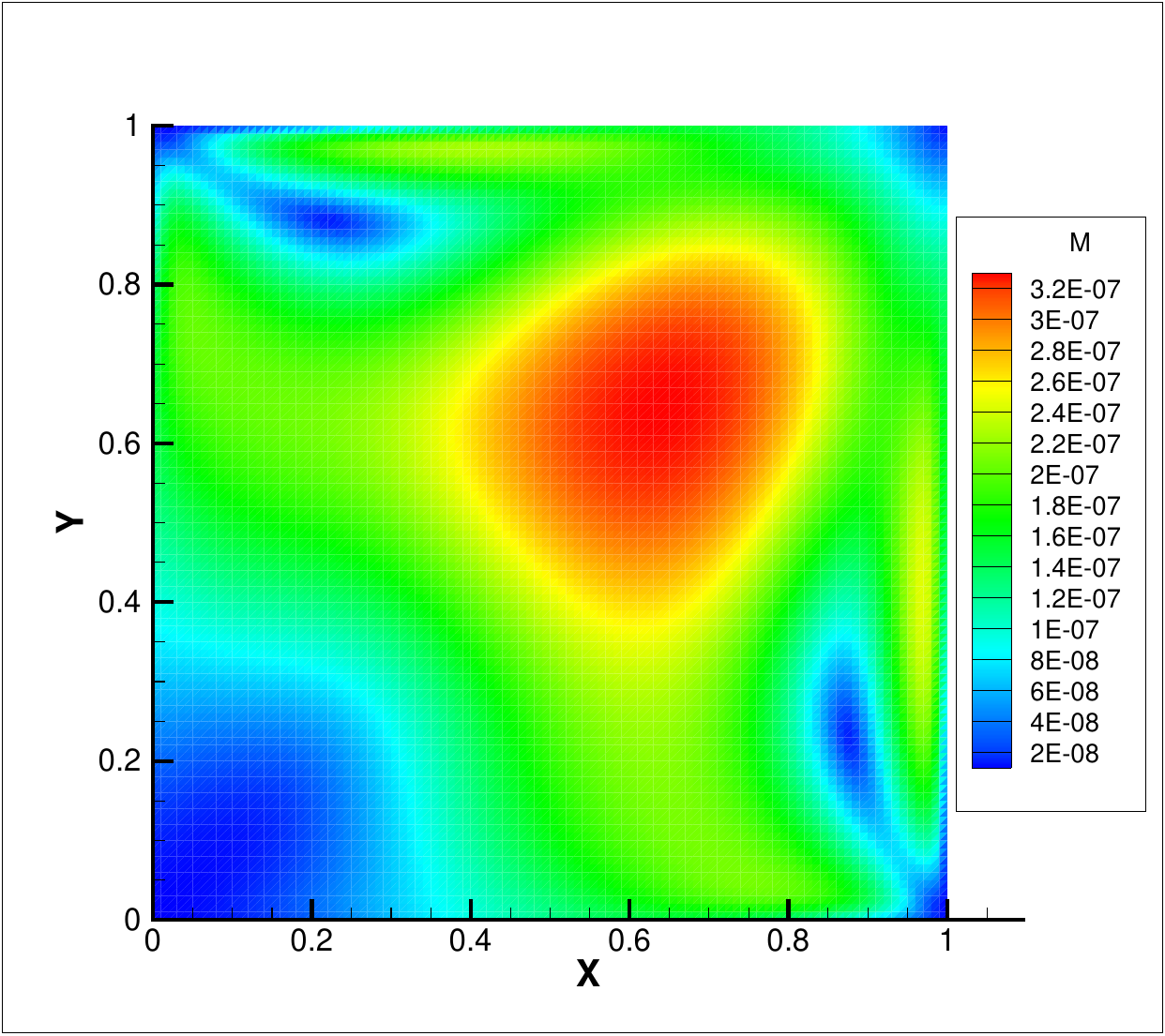}  
	\end{minipage}  
	\begin{minipage}{0.32\textwidth} 
		\centering  
		\includegraphics[width=\textwidth]{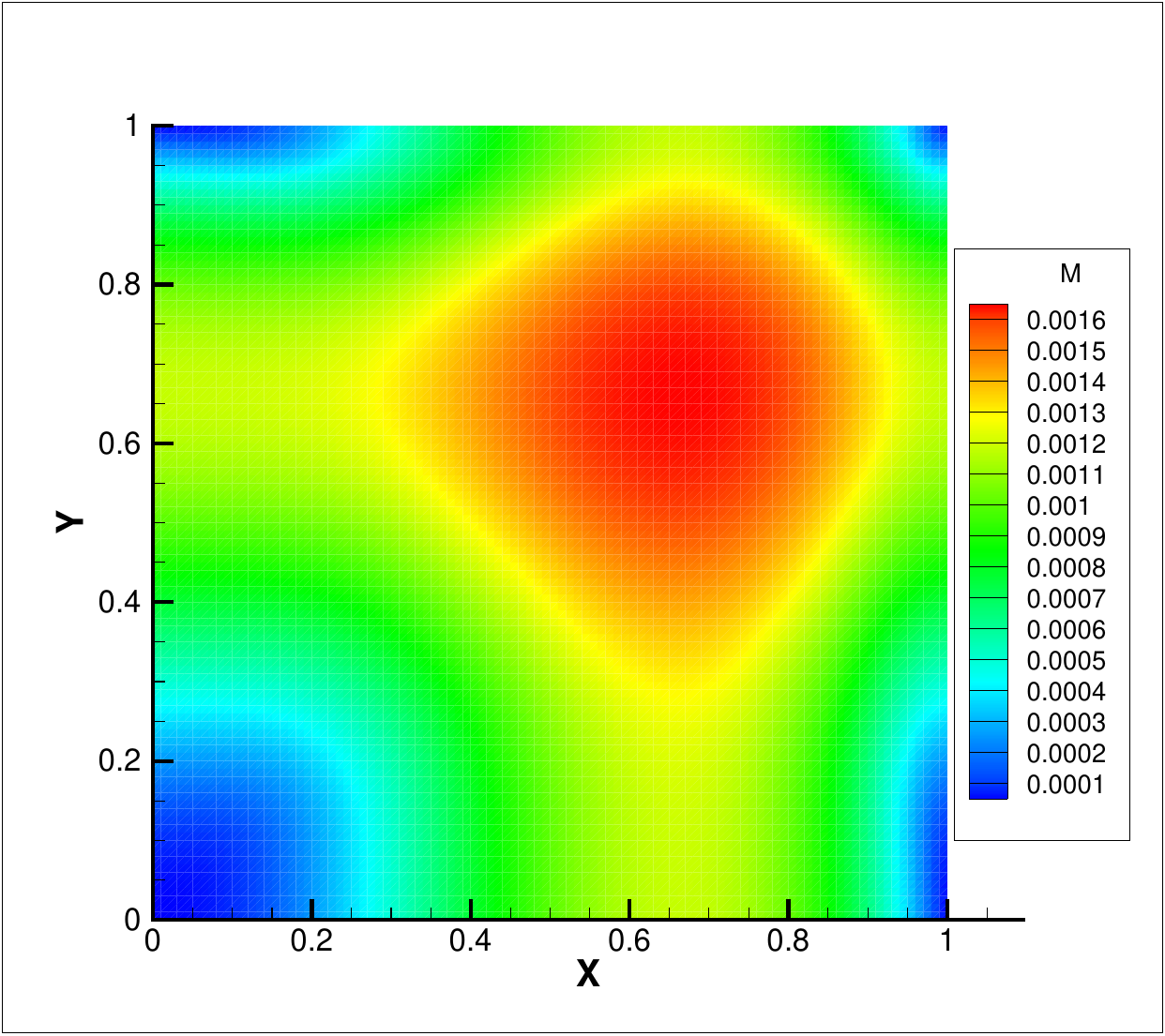} 
	\end{minipage}  
	\begin{minipage}{0.32\textwidth}  
		\centering  
		\includegraphics[width=\textwidth]{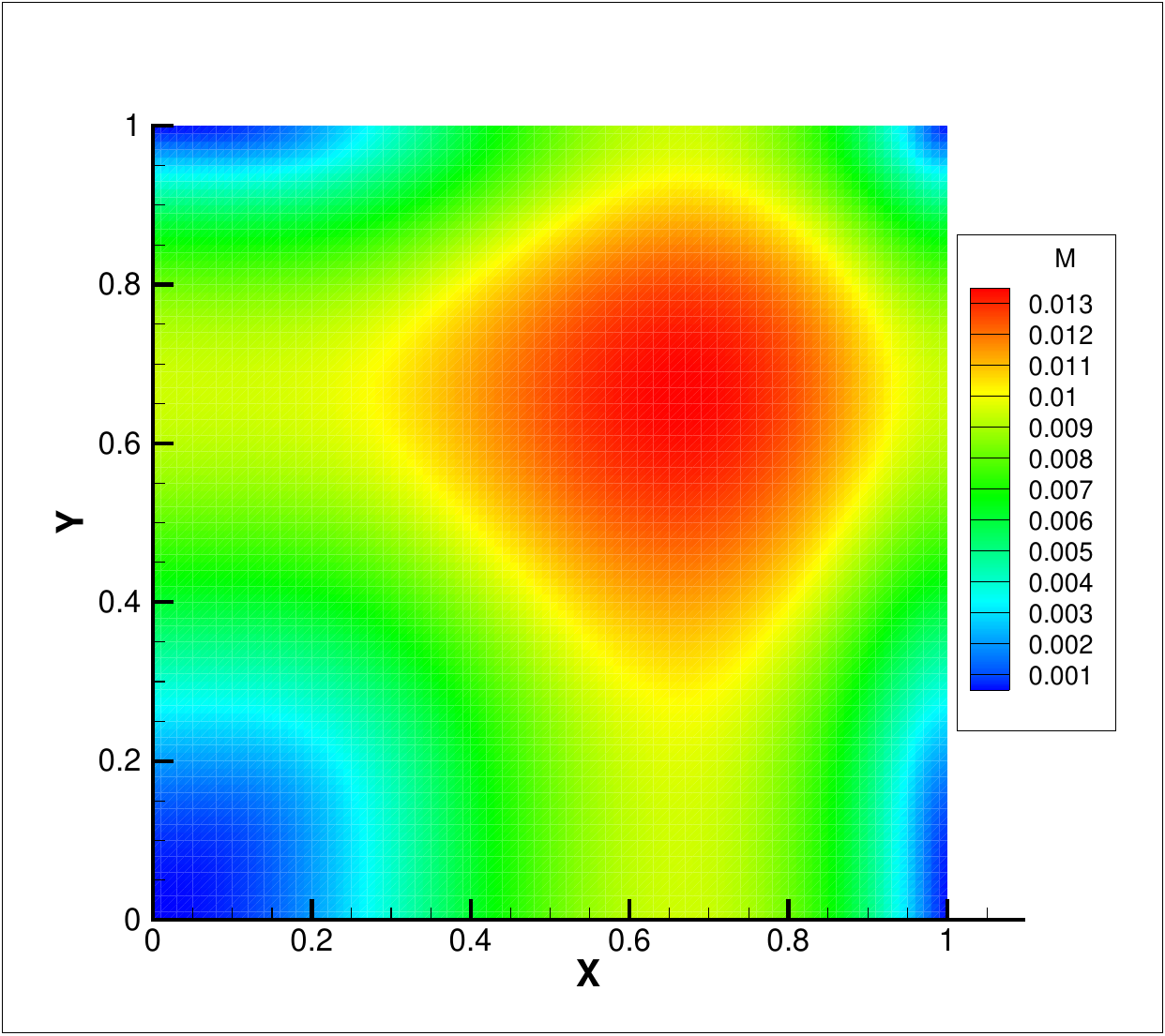}  
	\end{minipage}  
	
	\begin{minipage}{0.32\textwidth} 
		\centering  
		\includegraphics[width=\textwidth]{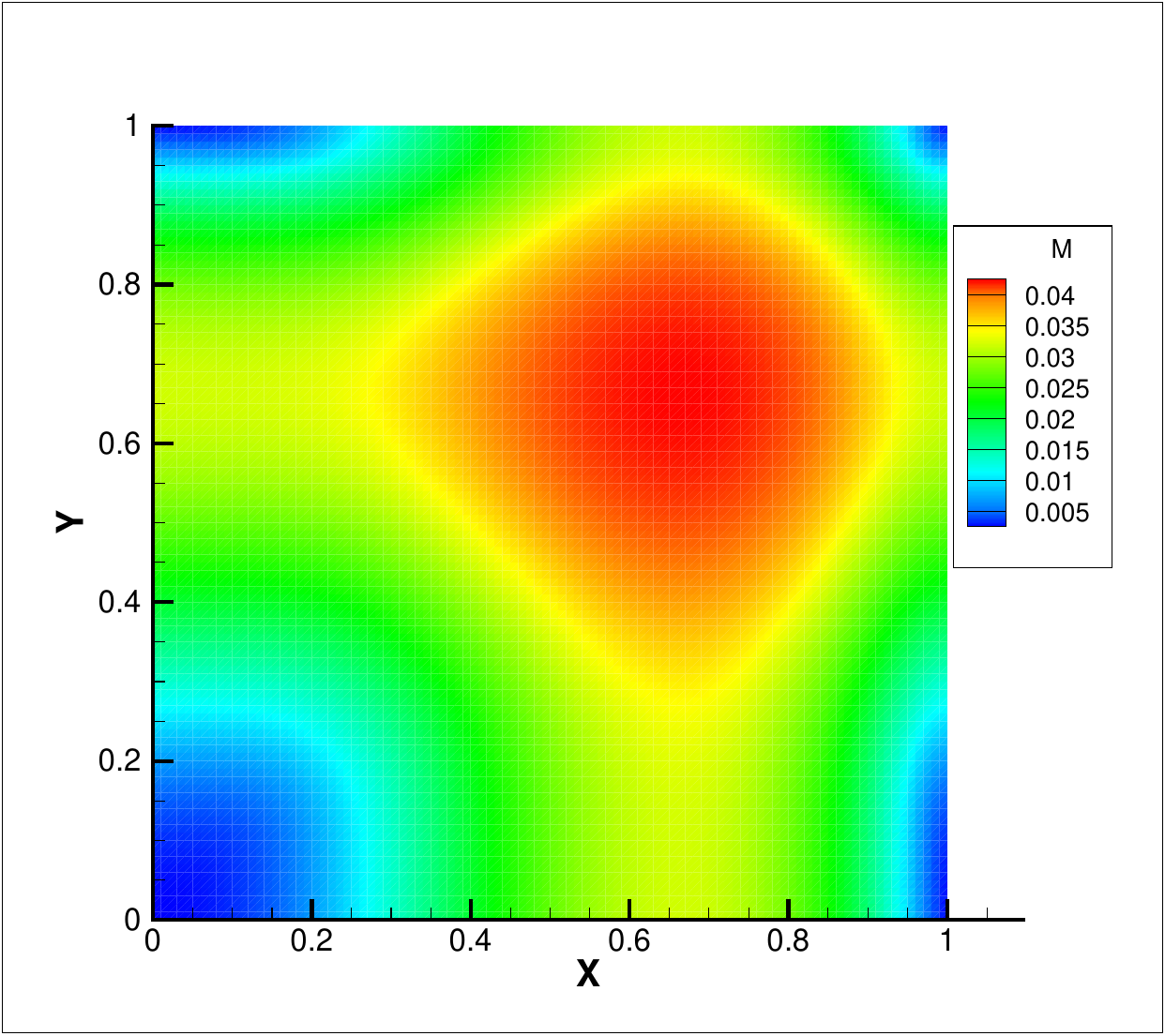}  
	\end{minipage}  
	\begin{minipage}{0.32\textwidth} 
		\centering  
		\includegraphics[width=\textwidth]{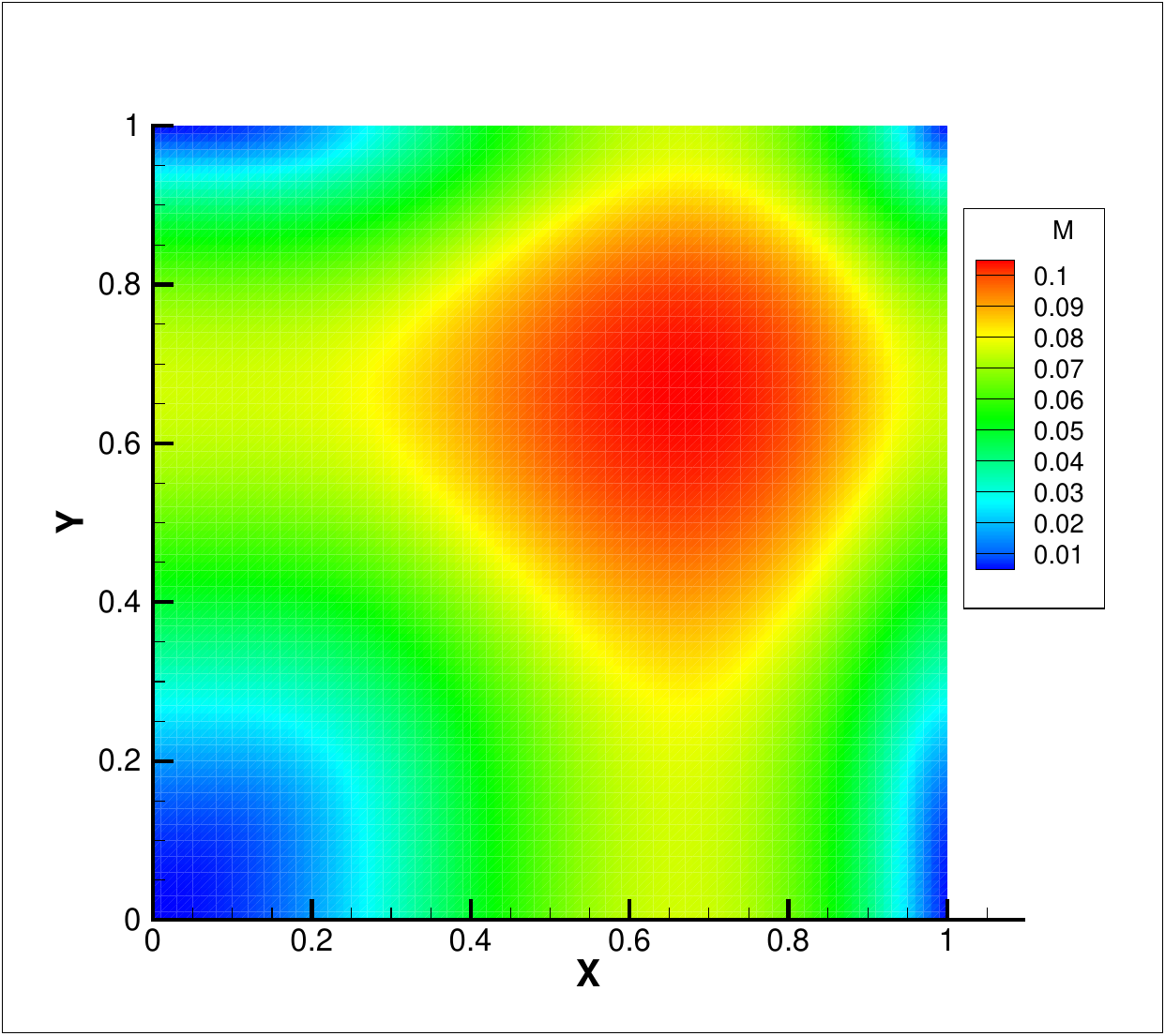} 
	\end{minipage}  
	\begin{minipage}{0.32\textwidth}  
		\centering  
		\includegraphics[width=\textwidth]{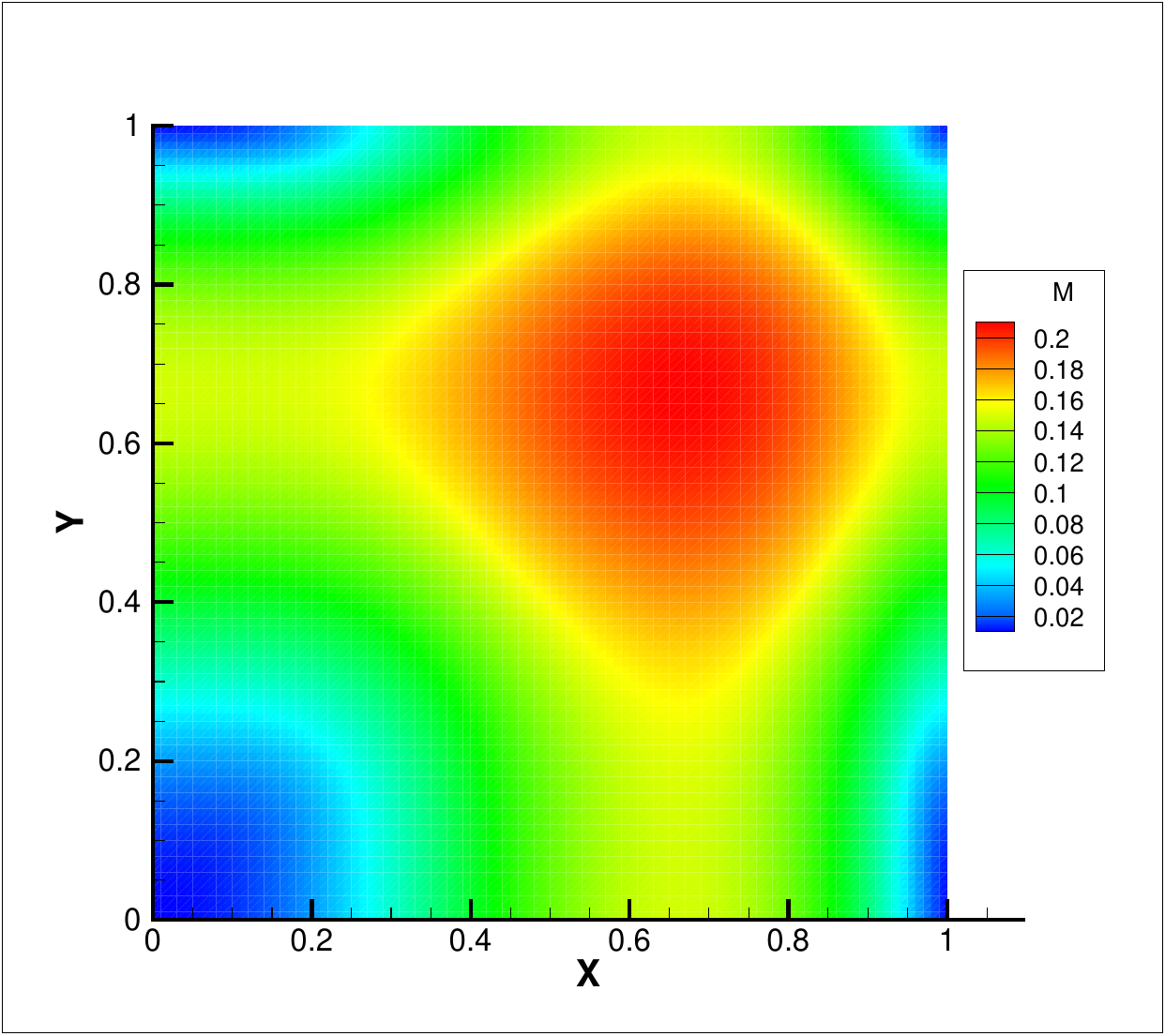}  
	\end{minipage}  
	\caption{Numerical solutions of velocity at times t = 0, 0.2, 0.4, 0.6, 0.8, 1.0.}  
	\label{velocity}  
\end{figure}

\begin{figure}[htbp] 
	
	\centering 

	\begin{minipage}{0.32\textwidth} 
		\centering  
		\includegraphics[width=\textwidth]{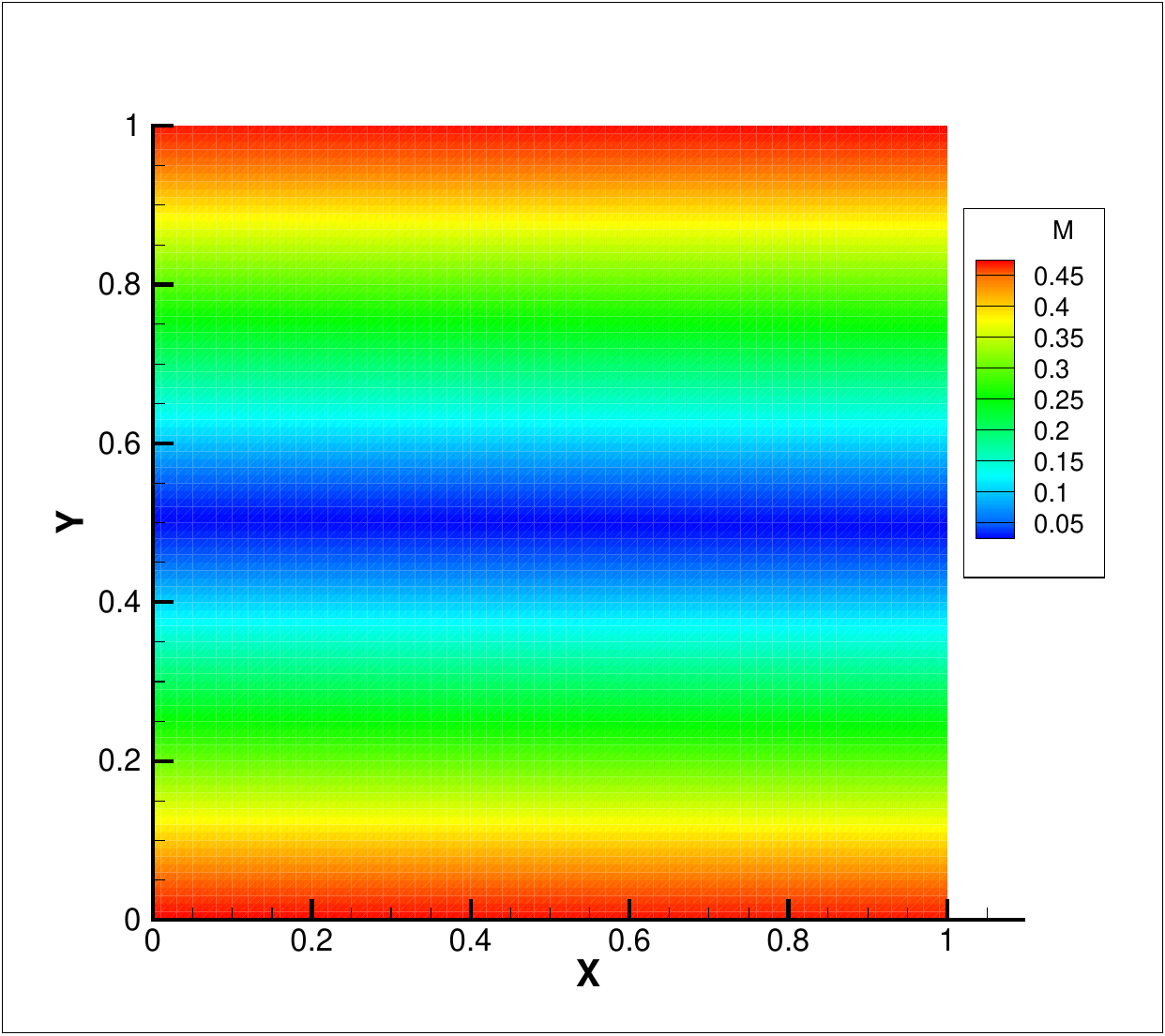}  
	\end{minipage}  
	\begin{minipage}{0.32\textwidth} 
		\centering  
		\includegraphics[width=\textwidth]{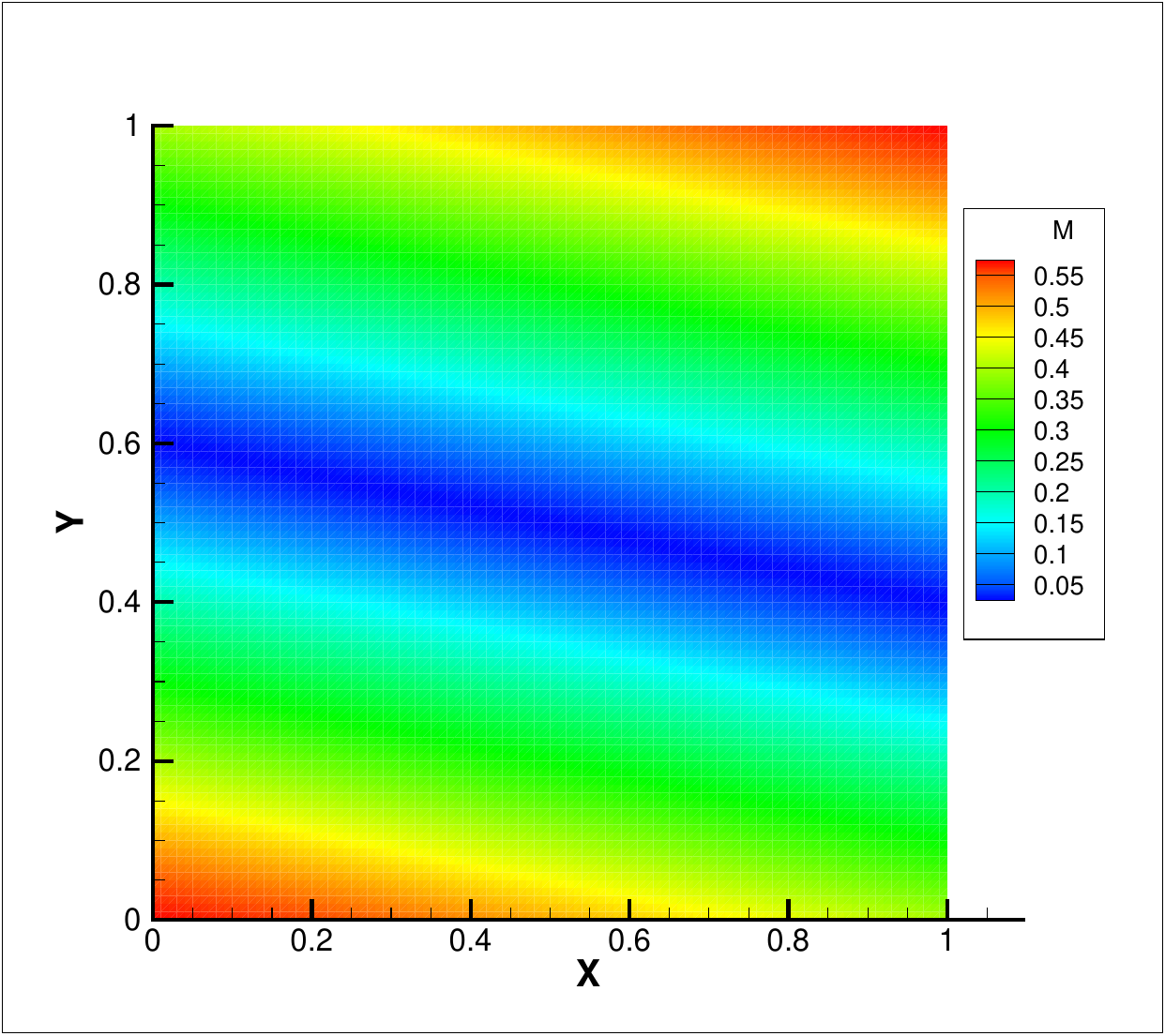} 
	\end{minipage}  
	\begin{minipage}{0.32\textwidth}  
		\centering  
		\includegraphics[width=\textwidth]{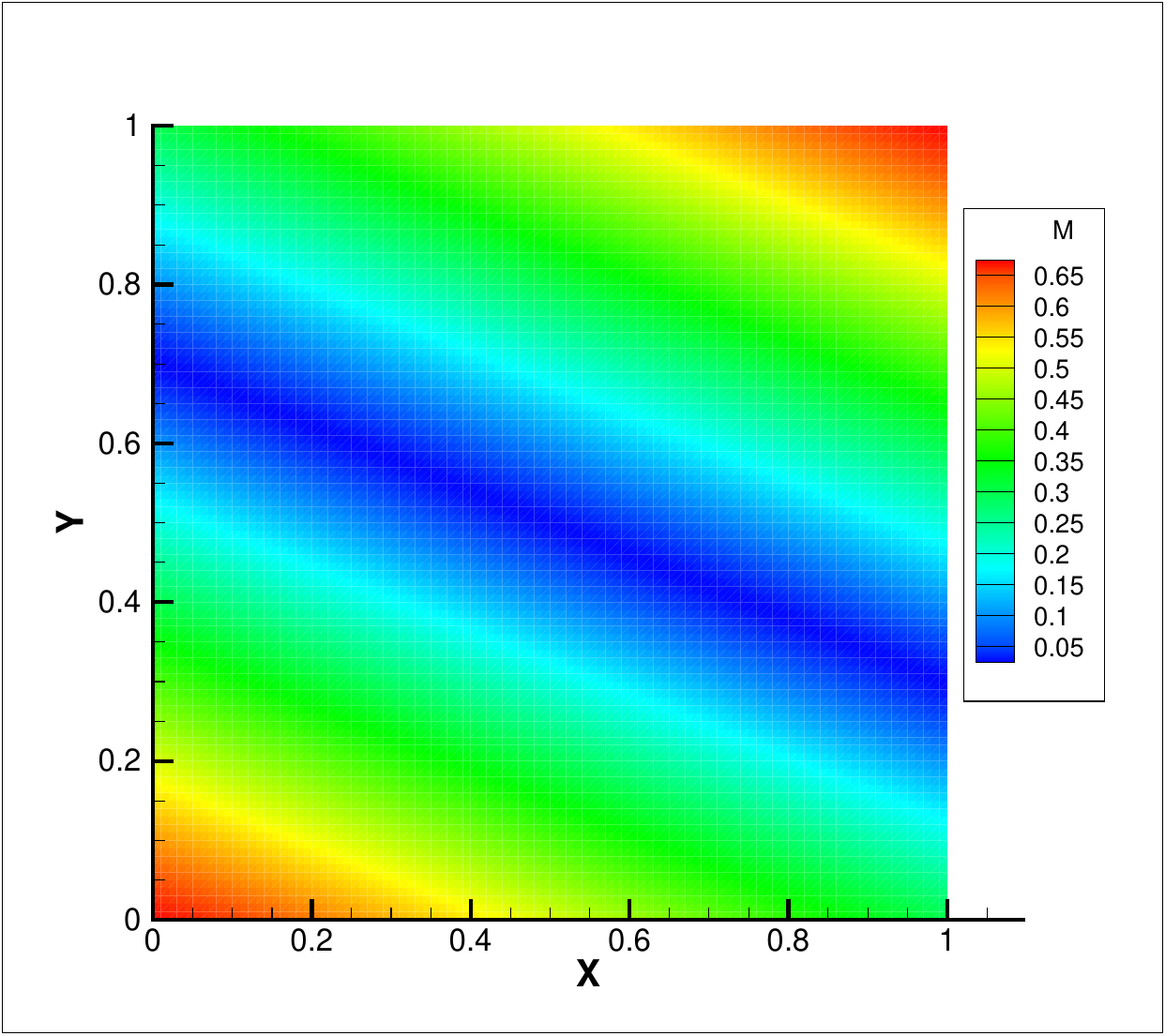}  
	\end{minipage}  
	
	\begin{minipage}{0.32\textwidth} 
		\centering  
		\includegraphics[width=\textwidth]{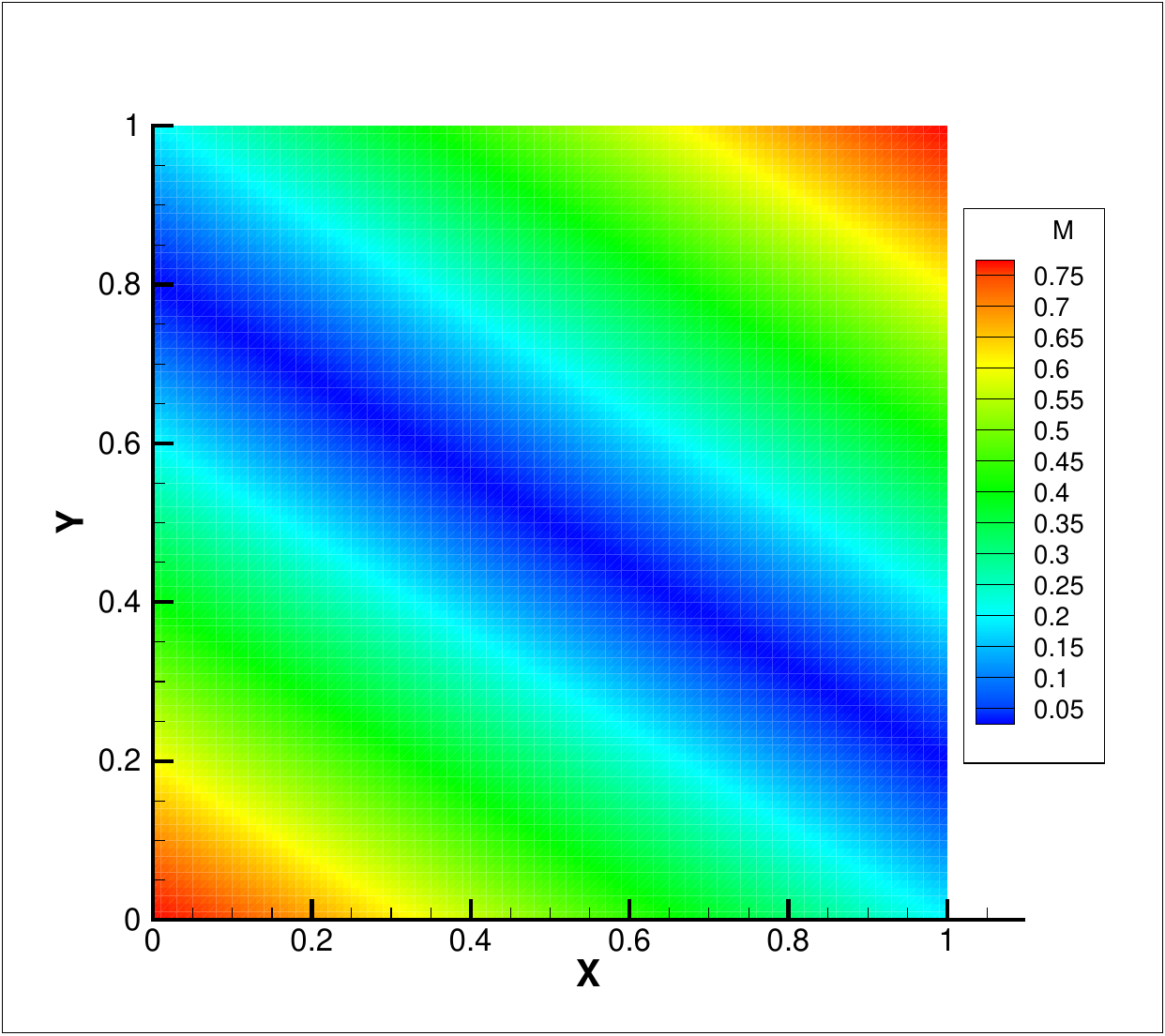}  
	\end{minipage}  
	\begin{minipage}{0.32\textwidth} 
		\centering  
		\includegraphics[width=\textwidth]{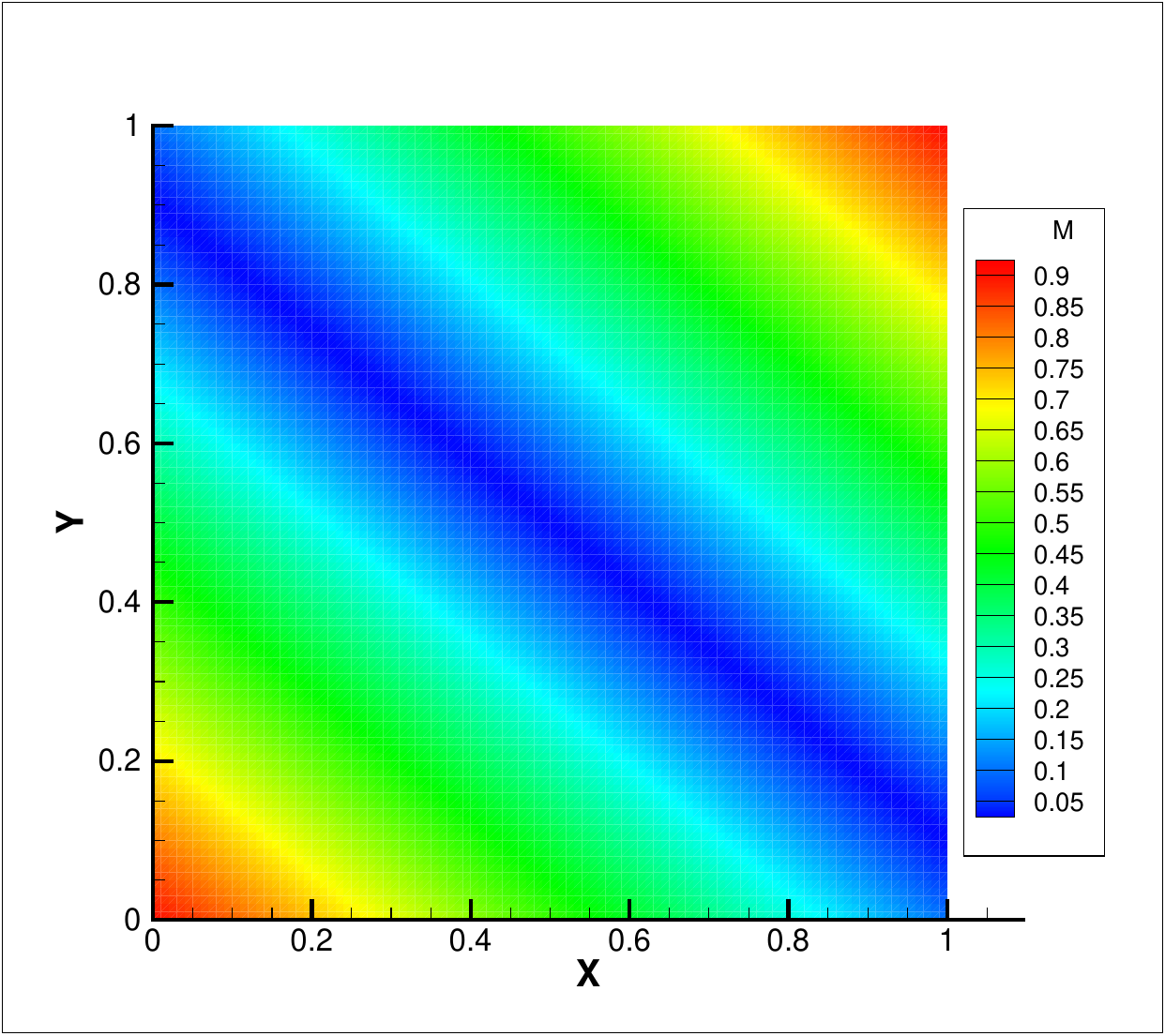} 
	\end{minipage}  
	\begin{minipage}{0.32\textwidth}  
		\centering  
		\includegraphics[width=\textwidth]{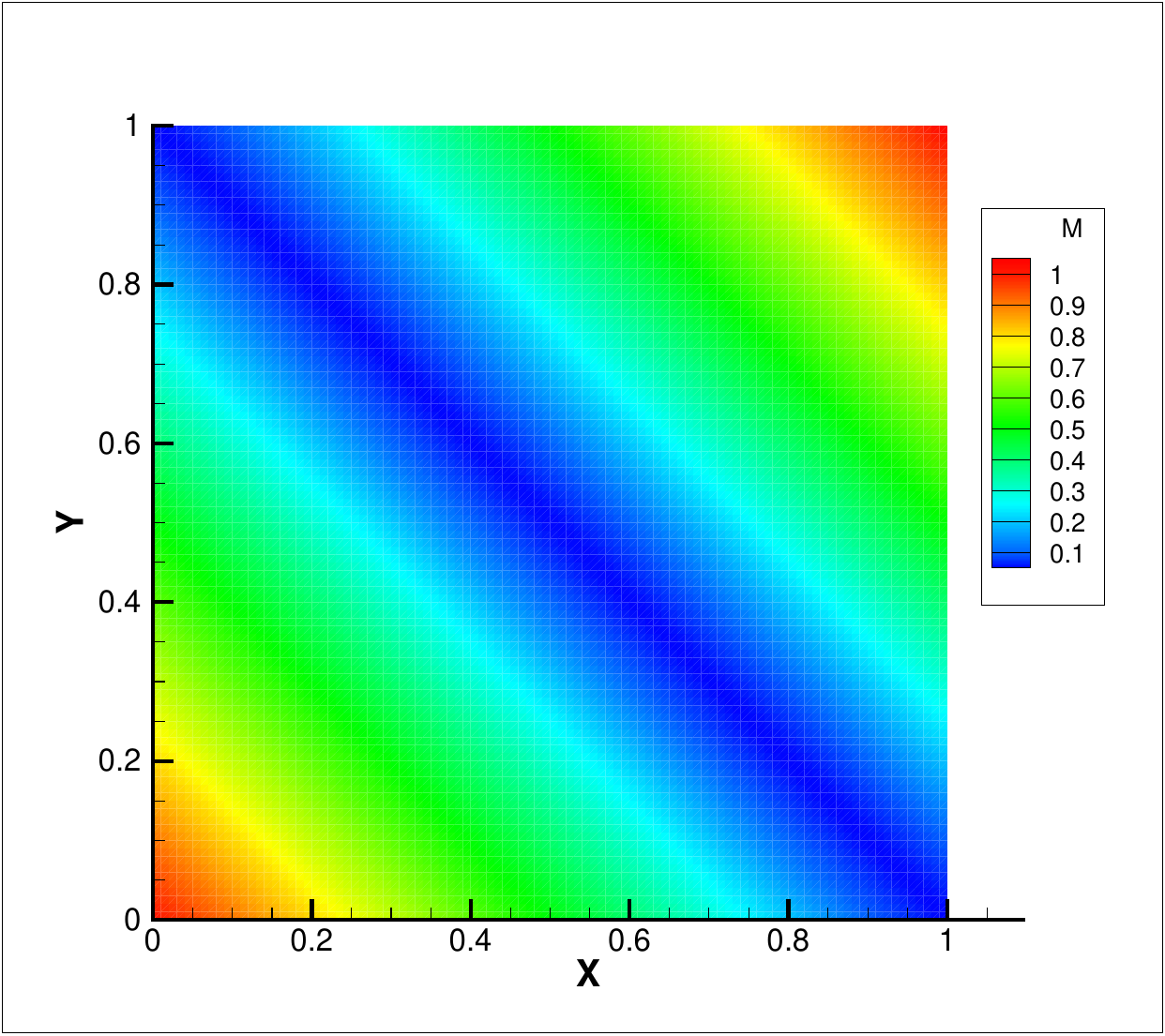}  
	\end{minipage}  
	\caption{Numerical solutions of pressure at times t = 0, 0.2, 0.4, 0.6, 0.8, 1.0.}  
	\label{pressure}  
\end{figure}

\begin{figure}[htbp] 
	
	\centering 

	\begin{minipage}{0.32\textwidth} 
		\centering  
		\includegraphics[width=\textwidth]{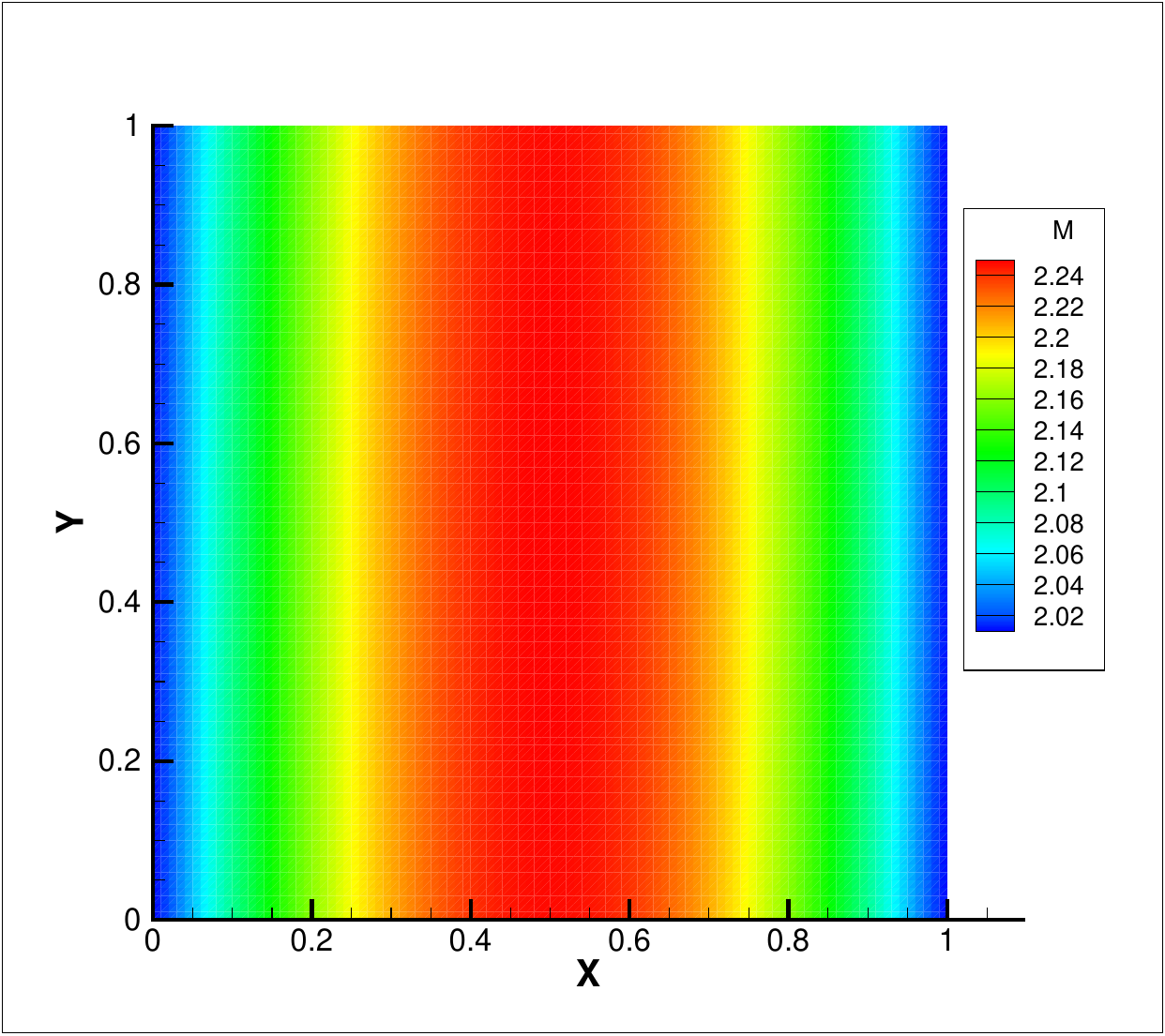}  
	\end{minipage}  
	\begin{minipage}{0.32\textwidth} 
		\centering  
		\includegraphics[width=\textwidth]{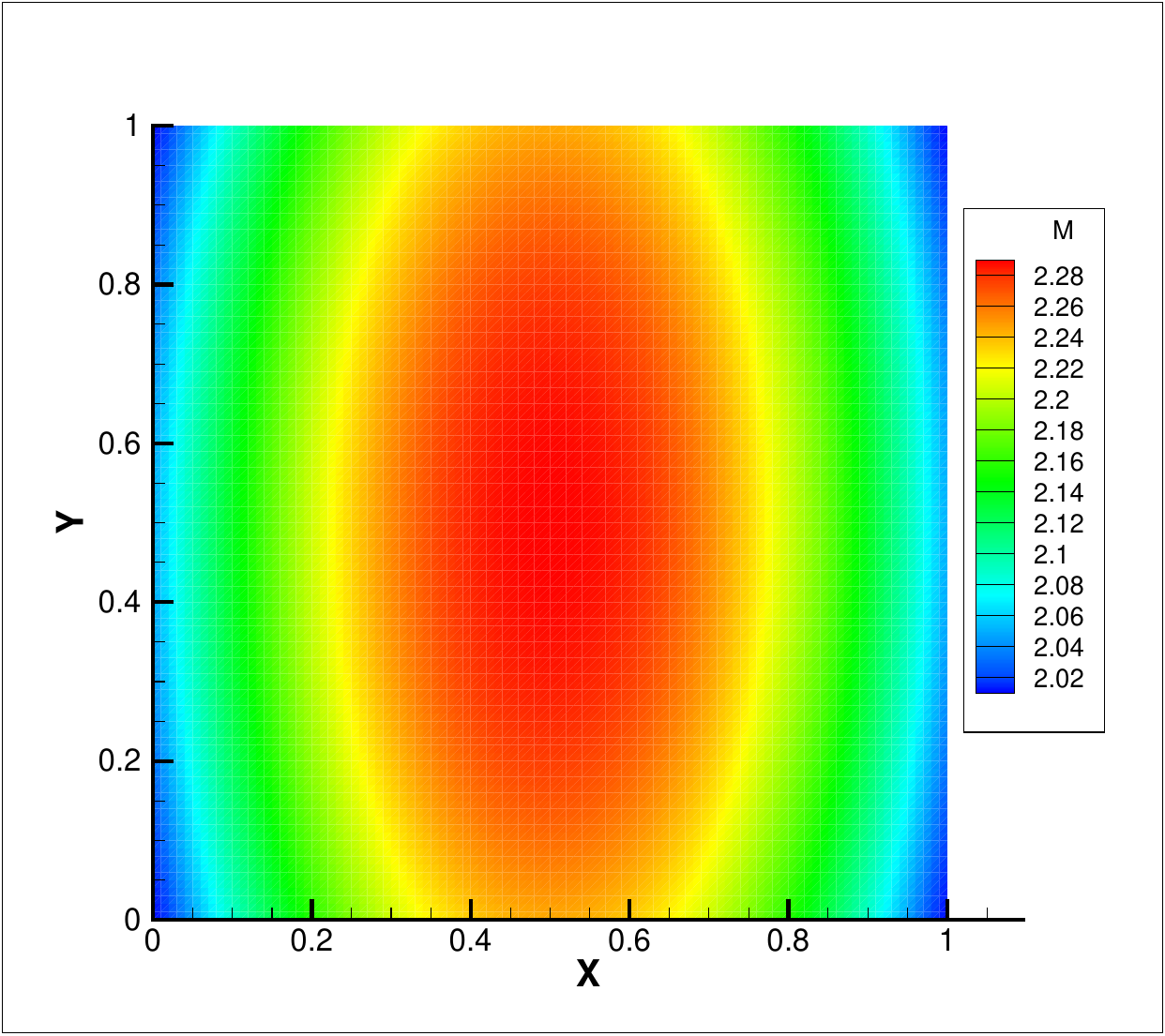} 
	\end{minipage}  
	\begin{minipage}{0.32\textwidth}  
		\centering  
		\includegraphics[width=\textwidth]{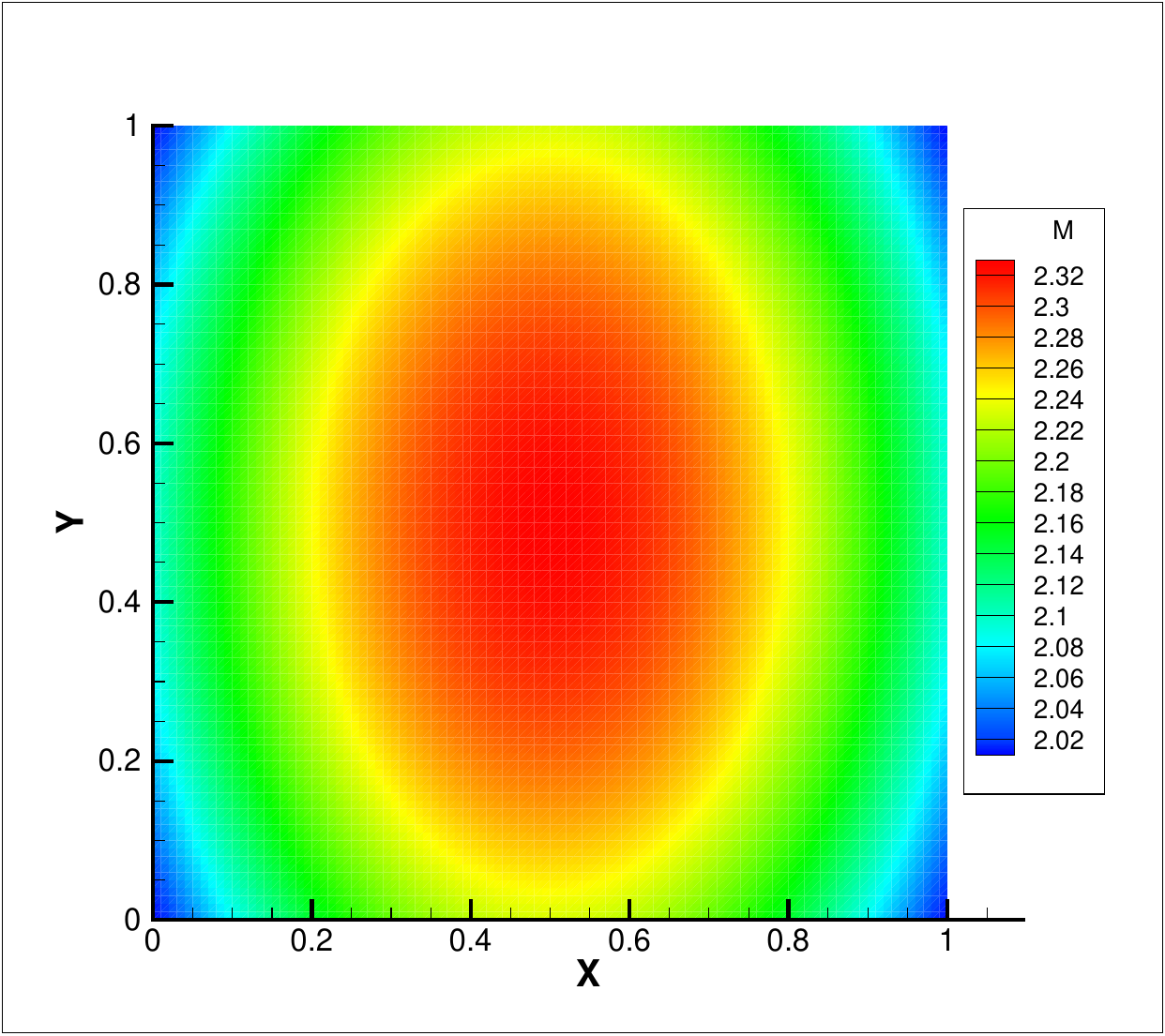}  
	\end{minipage}  
	
	\begin{minipage}{0.32\textwidth} 
		\centering  
		\includegraphics[width=\textwidth]{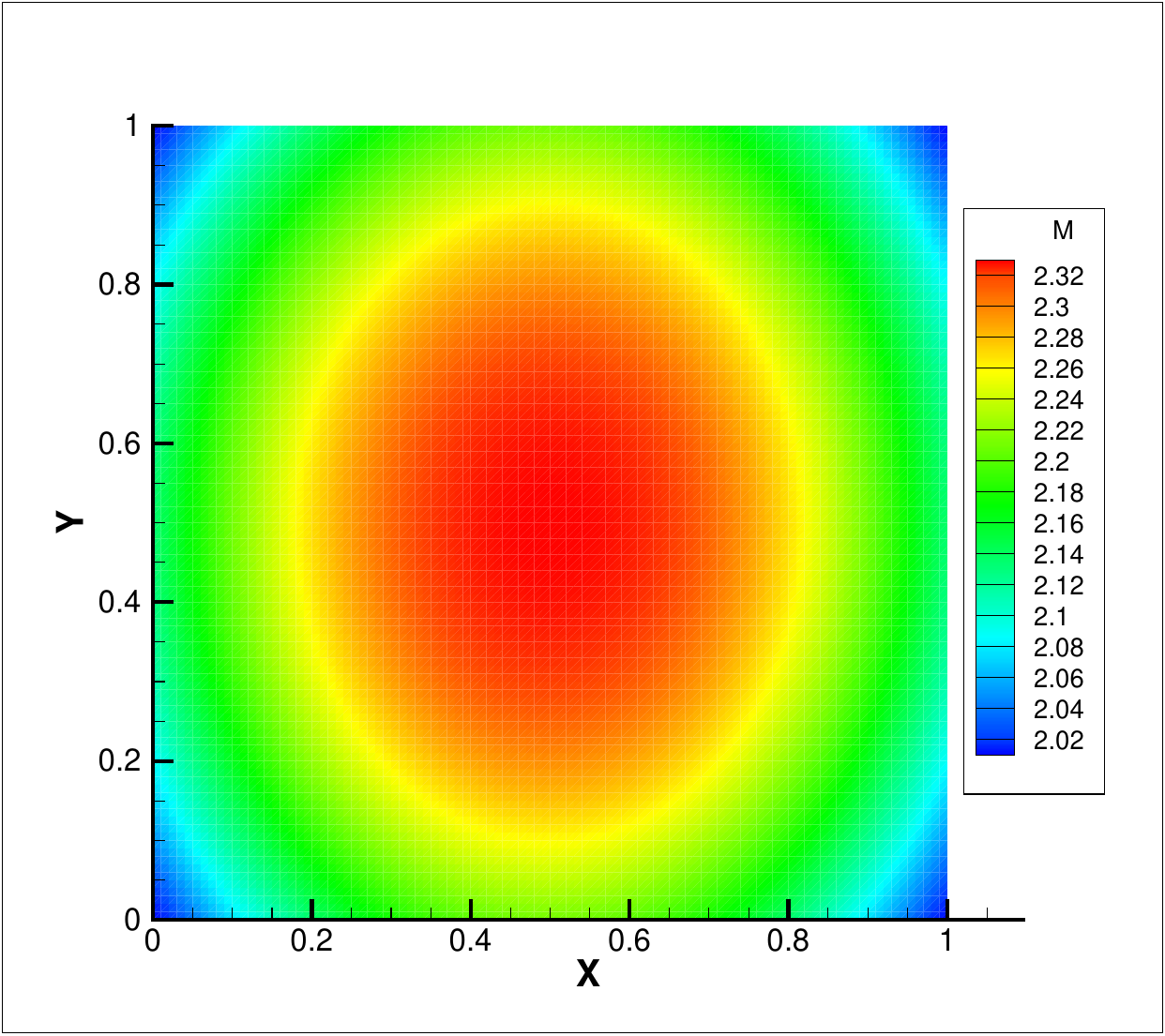}  
	\end{minipage}  
	\begin{minipage}{0.32\textwidth} 
		\centering  
		\includegraphics[width=\textwidth]{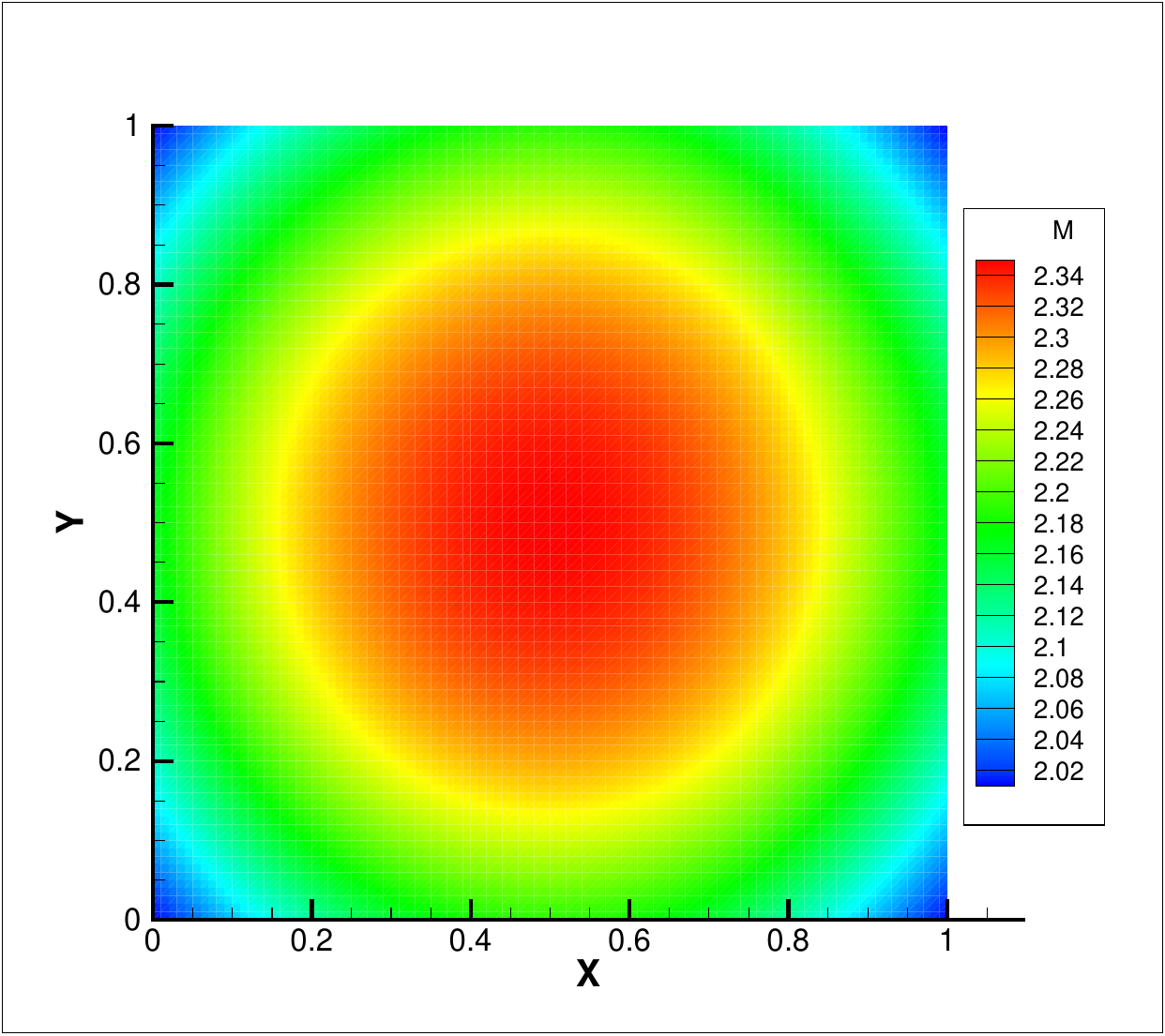} 
	\end{minipage}  
	\begin{minipage}{0.32\textwidth}  
		\centering  
		\includegraphics[width=\textwidth]{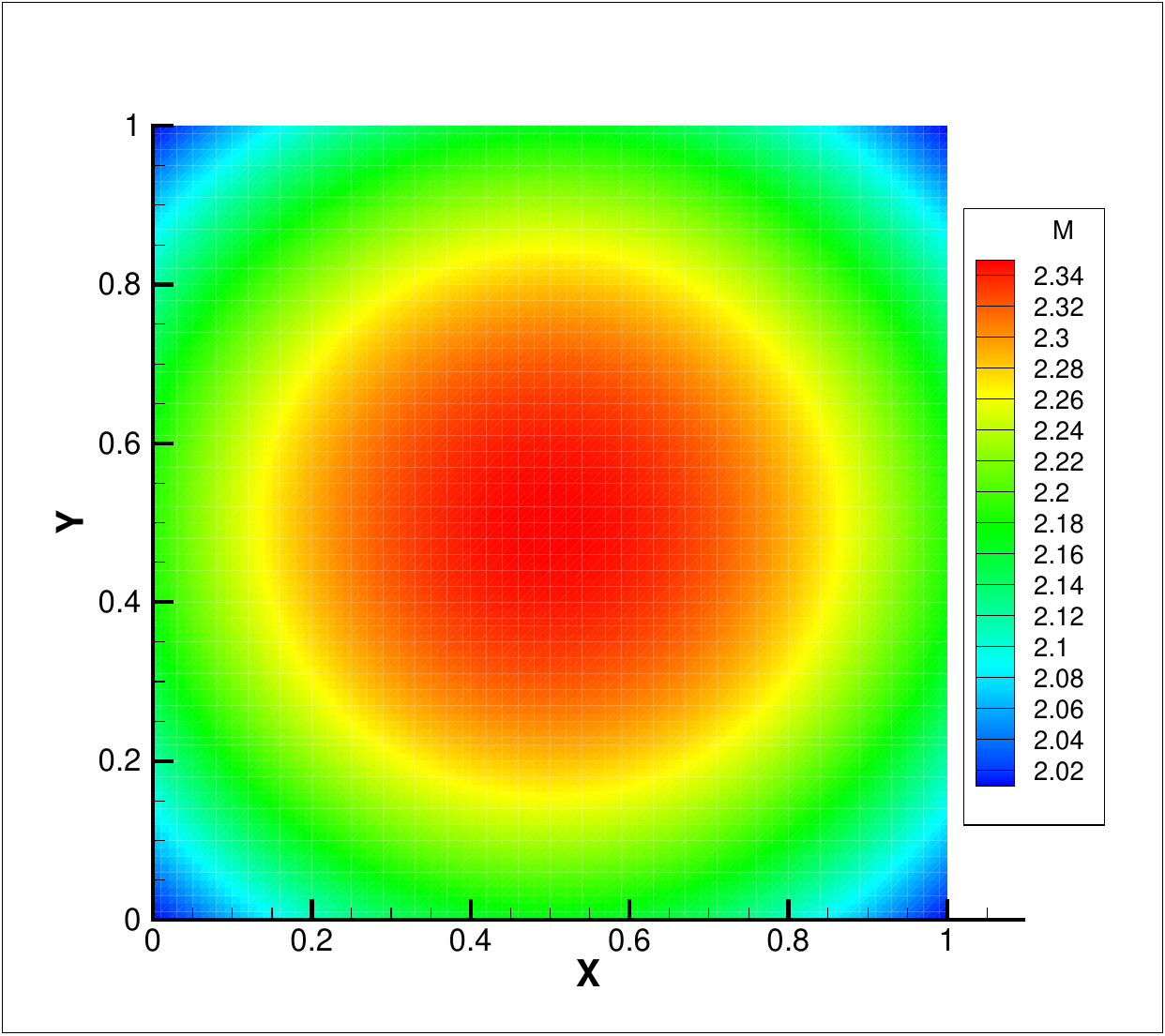}  
	\end{minipage}  
	\caption{Numerical solutions of density at times t = 0, 0.2, 0.4, 0.6, 0.8, 1.0.}  
	\label{density}  
\end{figure}

\begin{figure}[htbp] 
	
	\centering 

	\begin{minipage}{0.32\textwidth} 
		\centering  
		\includegraphics[width=\textwidth]{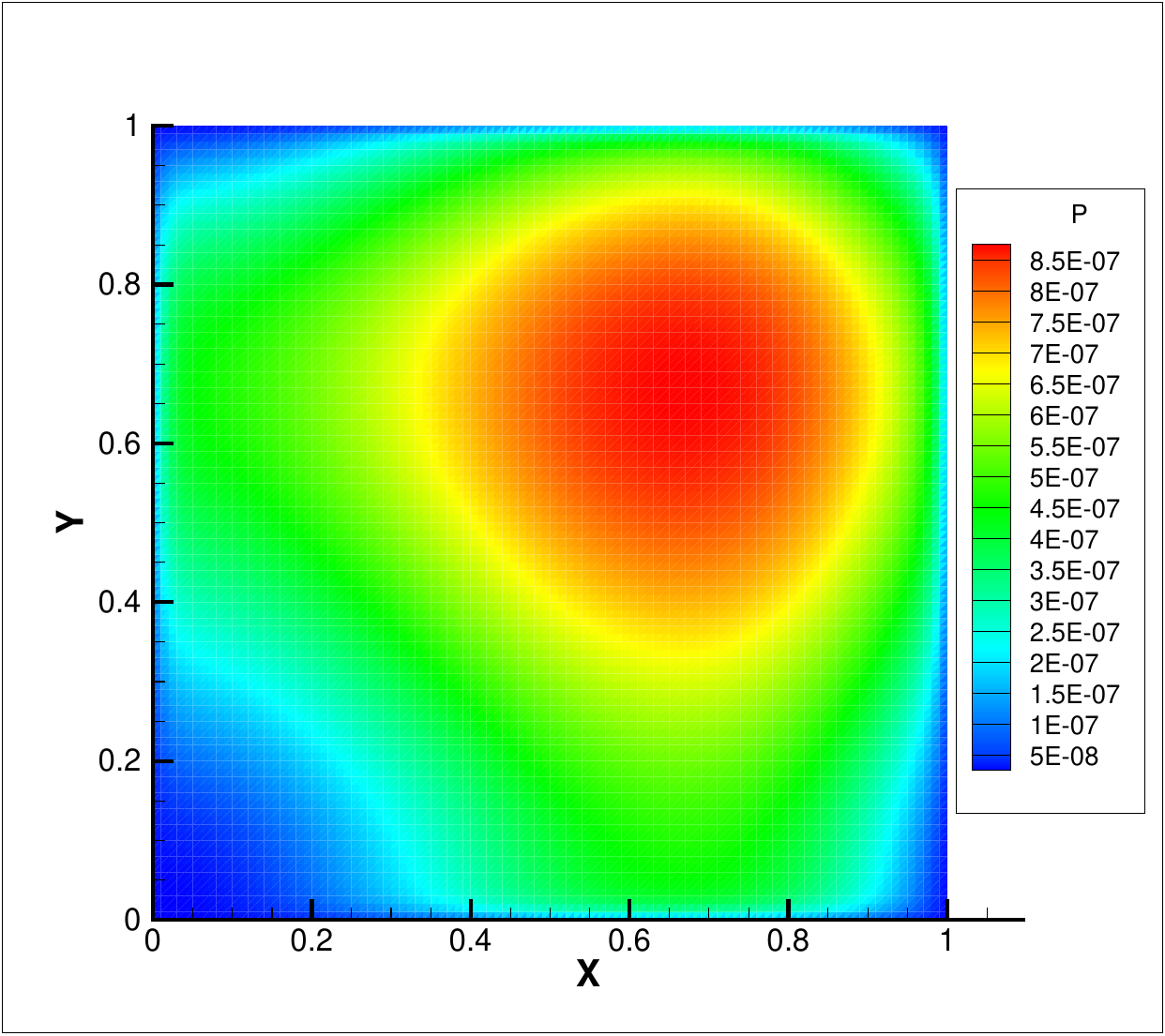}  
	\end{minipage}  
	\begin{minipage}{0.32\textwidth} 
		\centering  
		\includegraphics[width=\textwidth]{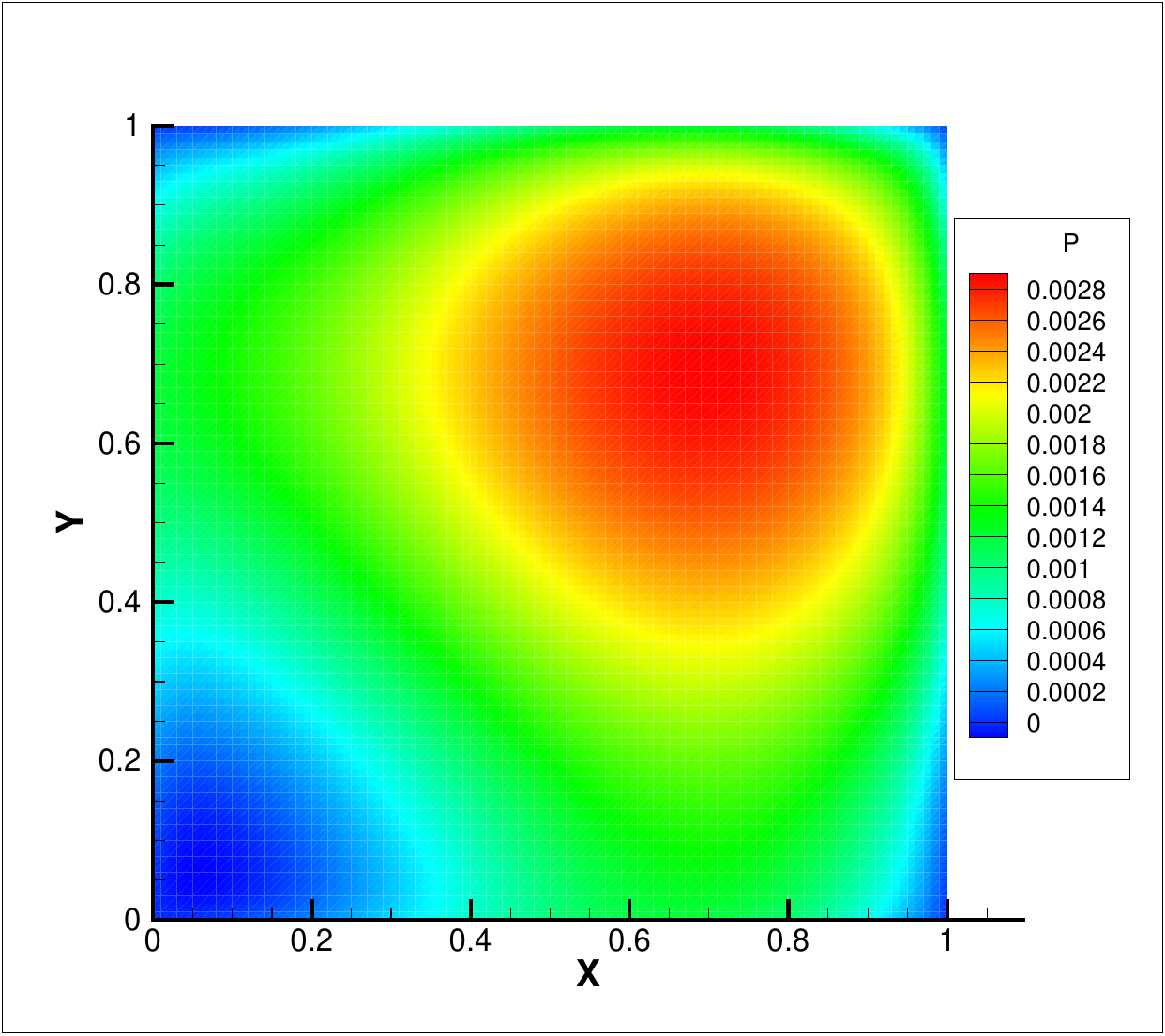} 
	\end{minipage}  
	\begin{minipage}{0.32\textwidth}  
		\centering  
		\includegraphics[width=\textwidth]{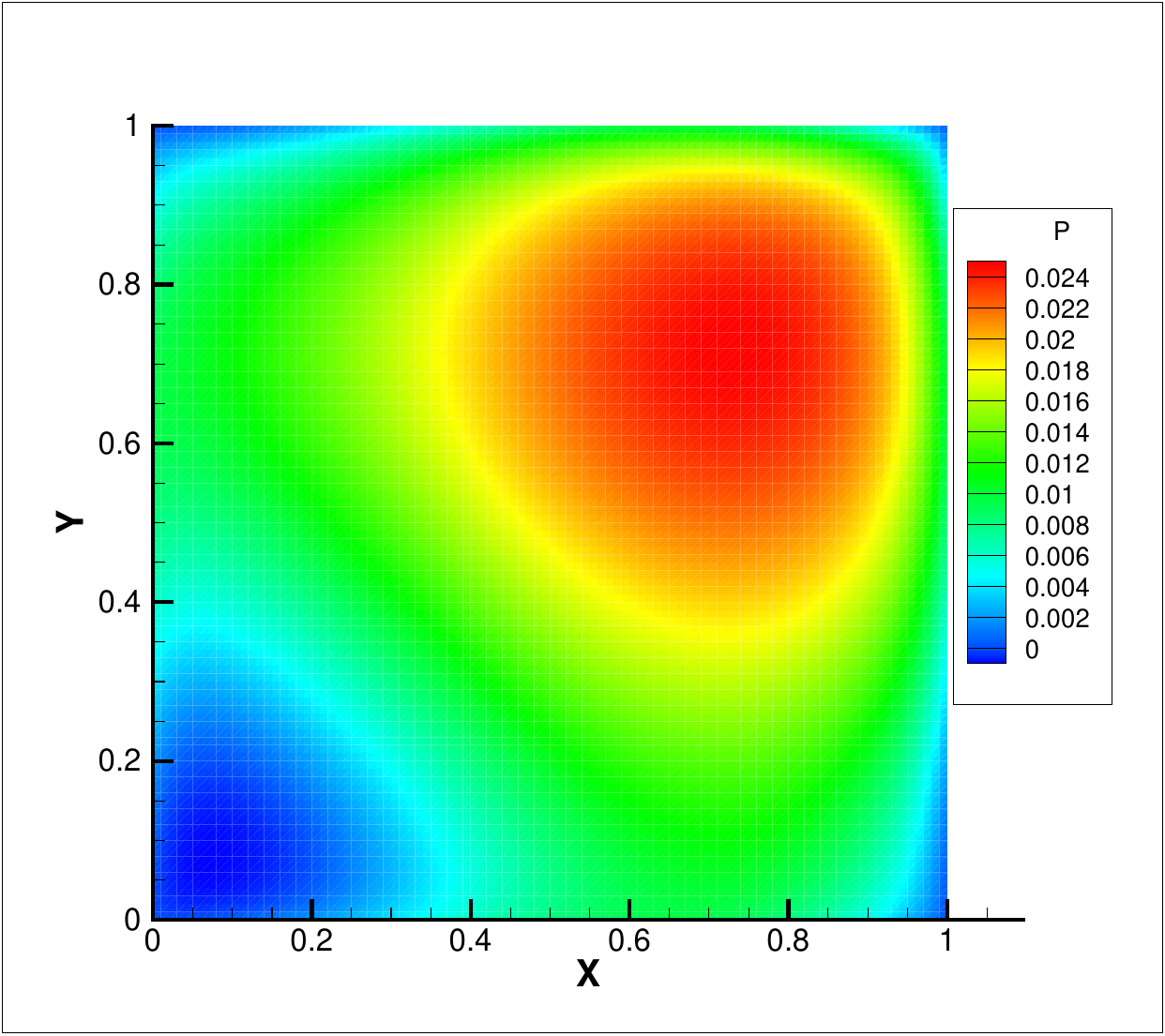}  
	\end{minipage}  
	
	\begin{minipage}{0.32\textwidth} 
		\centering  
		\includegraphics[width=\textwidth]{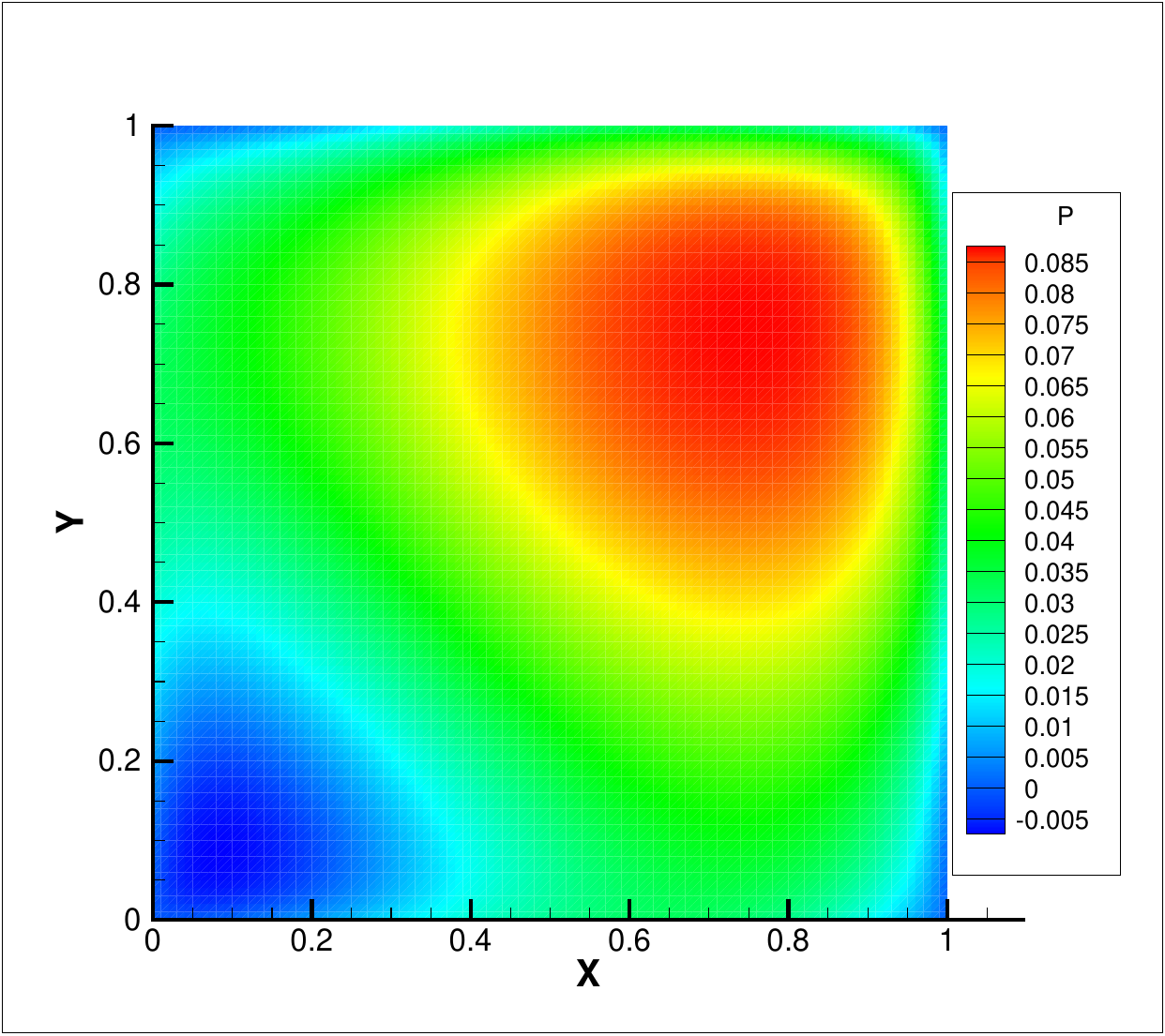}  
	\end{minipage}  
	\begin{minipage}{0.32\textwidth} 
		\centering  
		\includegraphics[width=\textwidth]{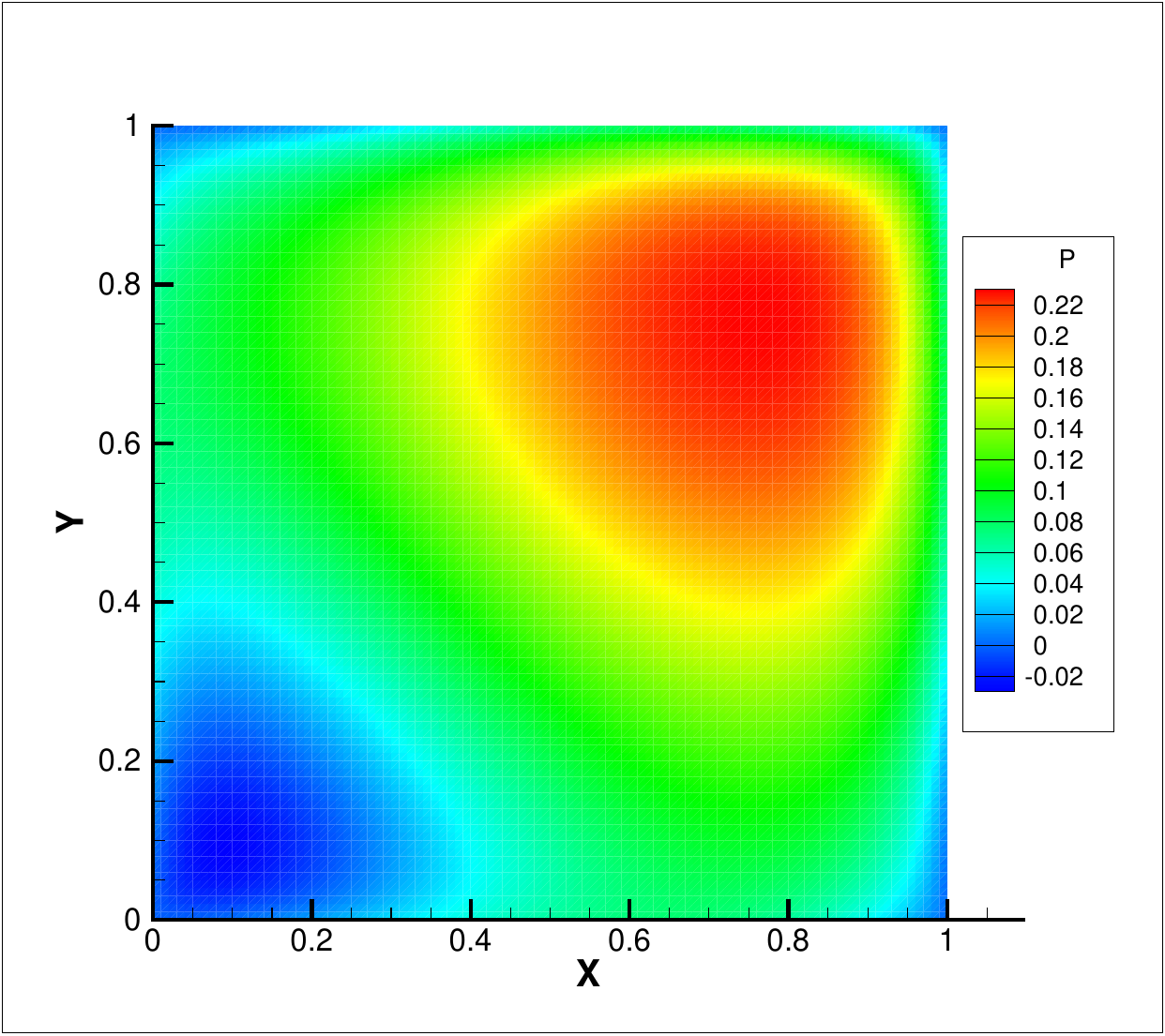} 
	\end{minipage}  
	\begin{minipage}{0.32\textwidth}  
		\centering  
		\includegraphics[width=\textwidth]{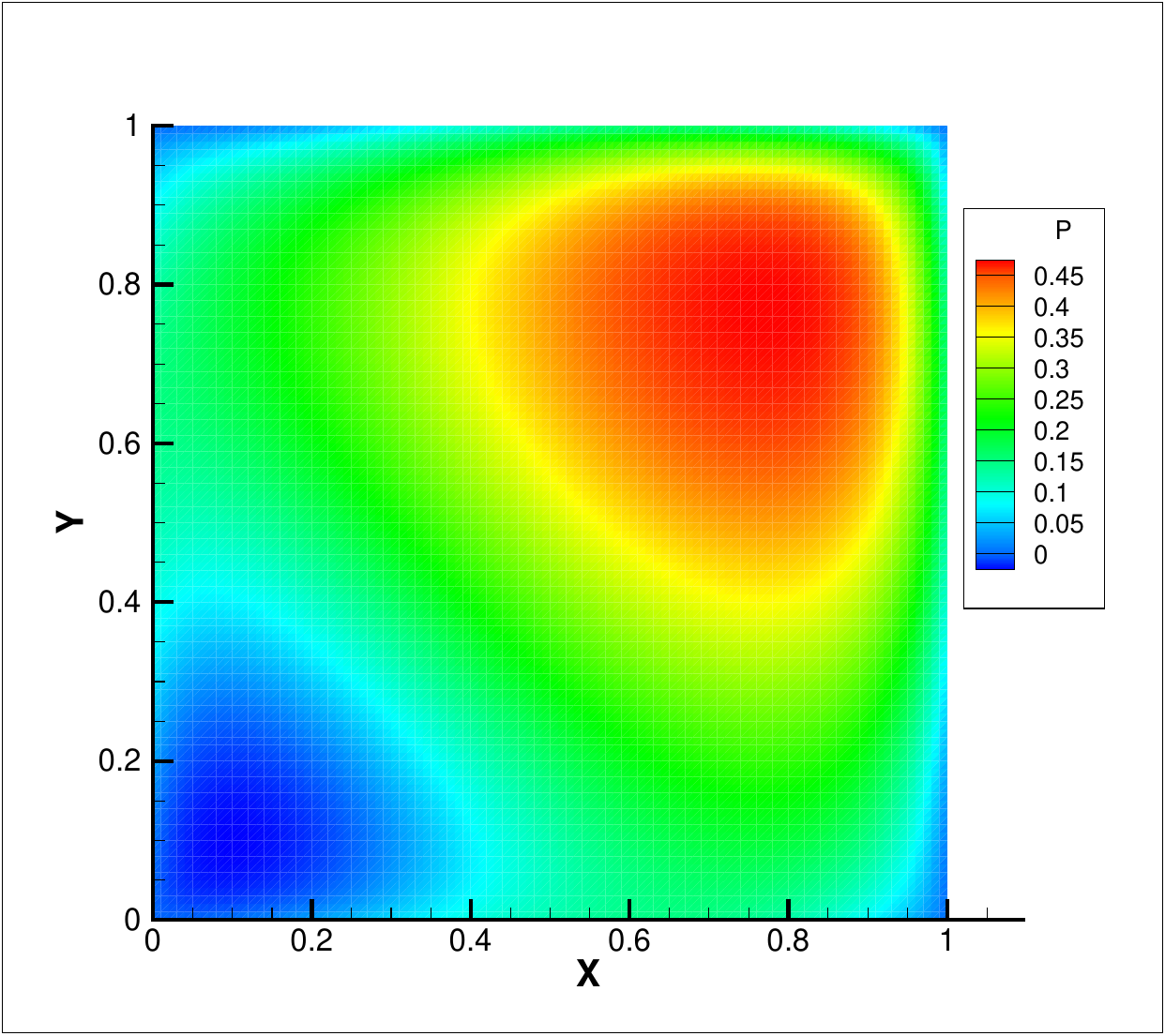}  
	\end{minipage}  
	\caption{Numerical solutions of temperature at times t = 0, 0.2, 0.4, 0.6, 0.8, 1.0.}  
	\label{temperature}  
\end{figure}

\section{Conclusions}

In this work, we developed and discussed a first-order Euler  FEMs to deal with the natural
convection problem with variable density, where nonlinear terms were treated by a linearized
semi-implicit approximation such that it is easy for implementation. Stability and error analysis of the Euler FEMs is deduced. Finally, a lot of numerical tests show that the proposed method not only can
deal with the incompressible natural convection problem with variable density but also save time
very well. 

\section{Acknowledgments}
The authors would like to thank the editor and referees for their valuable comments and suggestions
which helped us to improve the results of this paper.

\end{document}